\documentclass[letterpaper,10pt]{amsart}

\usepackage{epsfig}
\usepackage{amssymb}
\usepackage{amsfonts}
\usepackage{amsmath}
\usepackage{graphicx}
\usepackage{mathrsfs}
\usepackage{stmaryrd}
\usepackage{amsthm}
\usepackage{bbm}
\usepackage[all]{xy}
\usepackage[top=1in, bottom=1in, left=1.25in, right=1.25in]{geometry}

\newtheorem{theorem}{Theorem}
\newtheorem{lemma}{Lemma}
\newtheorem{proposition}{Proposition}
\newtheorem{corollary}{Corollary}

\newtheorem{definition}{Definition}
\newtheorem{remark}{Remark}

\numberwithin{equation}{section}
\numberwithin{lemma}{section}
\numberwithin{proposition}{section}
\numberwithin{corollary}{section}
\renewcommand\Re{\operatorname{Re}}

\def\slashpar{\slash\mkern-10mu \partial}

\title{Radiation Field for Einstein Vacuum Equations\\ with Spacial Dimension $n\geq 4$}
\author{Fang Wang}
\address{Shanghai Jiao Tong University}
\email{fangwang1984@sjtu.edu.cn}
\date{October 16, 2014}

\begin{document}

\begin{abstract}
In this paper, the radiation field is defined for solutions to Einstein vacuum equations which are close to Minkowski space-time with spacial dimension $n\geq 4$. The regularity properties and asymptotic behavior of those Einstein vacuum solutions are established at the same time. In particular, the map from Cauchy intial data to the radiation field is proved to be an isomorphism when restricting to a small neighborhood of Minkowski data in suitable weighted b-Sobolev spaces. 
\end{abstract}

\maketitle

\section{Introduction}
The Einstein vacuum equations determine a manifold $M^{1+n}$ with a Lorentzian metric with vanishing Ricci curvature:
\begin{equation}\label{eq.1}
R_{\mu\nu}=0. 
\end{equation}
The set $(\mathbb{R}_{t,x}^{1+n},m)$: standard Minkowski metric $m=-dt^2+\sum_{i=1}^n(dx^i)^2$ on $\mathbb{R}_{t,x}^{1+n}$ describes the Minkowski space-time solution of the system (\ref{eq.1}), which is stable under small perturbation according to the remarkable work \cite{CK} of D. Christodoulou and S. Klainerman in 1993. In this paper, the authors showed that  for $n\geq 3$ given $n$-dimensional manifold $\Sigma_0$ with a Riemannian metric $g_0$ and a symmetric two tensor $k_0$, such that $(\Sigma_0,g_0)$ is close to the Euclidean space, $k_0$ is close to $0$ and $(g_0,k_0)$  satisfy the constraint equations: 
\begin{equation}\label{constraint.1}
R_0-[k_0]_j^i[k_0]_i^j+[k_0]_i^i[k_0]_j^j=0,\quad
\nabla^j[k_0]_{ij}-\nabla_i[k_0]_j^j=0
\end{equation}
where $R_0$ is the scalar curvature of $g_0$ and $\nabla$ is the covariant differentiation w.r.t. $g_0$, 
then we can find out a $n+1$-dimensional Lorentzian manifold $(M,g)$ satisfying (\ref{eq.1}) and an embedding $\Sigma_0 \subset M$ such that $g_0$ is the restriction of $g$ to $\Sigma_0$ and $k_0$ is the second fundamental form. 

The Einstein vacuum equations are invariant under diffeomorphism. In the work of Y. Choquet-Bruhat,  \cite{CB1} followed by  \cite{CBG}, this allows her to choose a special \textit{harmonic gauge} to prove the existence and uniqueness up to diffeomorphism of a maximal globally hyperbolic smooth space-time arising from any set of smooth initial data. 
The harmonic gauge is also referred as \textit{wave coordinates}, in which the Einstein vacuum equations become a system of quasilinear wave equations on the components of the unknown metric $g_{\mu\nu}=m_{\mu\nu}+h_{\mu\nu}$: 
\begin{equation}\label{eq.2}
\Box_{g} h_{\mu\nu}=F_{\mu\nu}(h)(\partial h,\partial h)
\end{equation}
where $\Box_{g}=D_{\alpha}D^{\alpha}$ is the geometric wave operator of $g$ and  $F(u)(v,v)$ depends quadratically on $v$ and analytically on $u$ for $u$ small. By using Christoffel symbols, 
$\Box_{g}=g^{\alpha\beta}\partial_{\alpha}\partial_{\beta}-g^{\alpha\beta}\Gamma^{\mu}_{\alpha\beta}\partial_{\mu}$, which reduces to $\Box_{g}=g^{\alpha\beta}\partial_{\alpha}\partial_{\beta}$ in wave coordinates. This is because wave coordinates are required to be solutions to the equations 
\begin{equation*}
\Box_{g}x^{\mu}=0
\end{equation*}
for $\mu=0,1,...,n$. 
The metric $g_{\mu\nu}$ relative to wave coordinates $\{x^{\mu}\}$ satisfies the \textit{harmonic gauge condition}:
\begin{equation}\label{harmonic.3}
\Gamma_{\mu}=0
\end{equation}
for $\mu=0,1,...,n$, where
$$\Gamma_{\mu}=g_{\mu\nu}g^{\alpha\beta}\Gamma_{\alpha\beta}^{\nu}=g^{\alpha\beta}\partial_{\beta}g_{\alpha\mu}
-\tfrac{1}{2}g^{\alpha\beta}\partial_{\mu}g_{\alpha\beta}.
$$
Under this condition the stability of Minkowski space-time can reformulate as follows: given a pair of small symmetric two tensors $(h^0,h^1)$ on $\mathbb{R}^n\simeq \{t=0\}\subset\mathbb{R}^{1+n}_{t,x}$ such that
\begin{equation}\label{harmonic.1}
\Gamma_{\mu}|_{t=0}=0, \quad \partial_{t}\Gamma_{\mu}|_{t=0}=0
\end{equation}
for $\mu=0,1,...,n$, we want to find a Lorentian metric $g=m+h$ of signature $(n,1)$ satisfying the reduced Einstein equations (\ref{eq.2}) and such that $(h|_{t=0},\partial_th|_{t=0})=(h^0,h^1)$. 

\begin{theorem}[H. Lindblad, I. Rodnianski]\label{thm.lr}
Given asymptotically flat Cauchy data $(h^0,h^1)$ in coordinates $(t,x)$ satisfying the harmonic gauge condition (\ref{harmonic.1}) and small, then the solution to the reduced Einstein equations (\ref{eq.2}) provides a solution to the Einstein vacuum equations (\ref{eq.1}).
\end{theorem}

The reduced Einstein vacuum equations satisfies the \textit{null condition} when the spacial dimension $n\geq 4$, which ensures the global existence theorem for small Cauchy data. For $n=3$, it can be shown that the Einstein vacuum equations in harmonic gauge do not satisfy the null condition. Moreover, Y. Choquet-Bruhat showed in \cite{CB2} that even without imposing a specific gauge the Einstein equations violate the null condition. However, the reduced Einstein vacuum equations (\ref{eq.2}) satisfy the \textit{weak null condition}. In the work of \cite{LR1} and \cite{LR2}, H. Lindblad and I. Rodnianski reproved the global existence for Einstein vacuum equations in harmonic gauge for general asymptotic flat initial data  by combining this condition with the vector field method. 

In this paper, we apply the radiation field theory due to F. G. Friedlander to study the asymptotic behavior of solutions to Einstein vacuum equations in harmonic gauge. Friedlander's radiation field was used by L. H\"{o}rmander to study the asymptotic behavior of solutions to linear hyperbolic equation in the following coordinates: 
\begin{equation*}
\rho=\frac{1}{|x|}, \quad\tau=t-r, \quad\theta=\frac{x}{r}
\end{equation*}
for $|x|$ large. For instance, consider the Cauchy problem in Minkowski space-time as follows: 
\begin{equation*}
\Box_{m}u(t,x)=0,\quad u(0,x)=u_0(x),\quad \partial_t u(0,x)=u_1(x),
\end{equation*}
where $u_0,u_1\in C_c^{\infty}(\mathbb{R}^n).$
Near $\rho=0$, by writing $u=\rho^{\frac{n-1}{2}}\tilde{u}$, $\tilde{u}$ satisfies the following equation 
\begin{equation*}
(\Box_{\tilde{m}}+\tfrac{(n-1)(n-3)}{4})\tilde{u}=0
\end{equation*}
where $\tilde{m}=\rho^2m$. F. G. Friedlander showed that $\tilde{u}$ is smooth up to $\rho=0$. Then the radiation field is the image of the map
\begin{equation*}
\mathcal{R_{LP}}: \dot{H}^1(\mathbb{R}^n)\times L^2(\mathbb{R}^n) \ni (u_0,u_1)\longrightarrow \partial_{\tau}\tilde{u}|_{\rho=0}\in L^2(\mathbb{R}_{\tau}\times S^{n-1}_{\theta})
\end{equation*}
which is an isometric isomorphism. Here $\mathcal{R_{LP}}$ is the free space translation representation of Lax and Phllips.

I generalize this idea by considering the conformal transformation of the reduced Einstein vacuum equations on a suitable compactification of $\mathbb{R}_{t,x}^{1+n}$ as follows.
\begin{figure}[htp]
\centering
\includegraphics[totalheight=2.5in,width=2.5in]{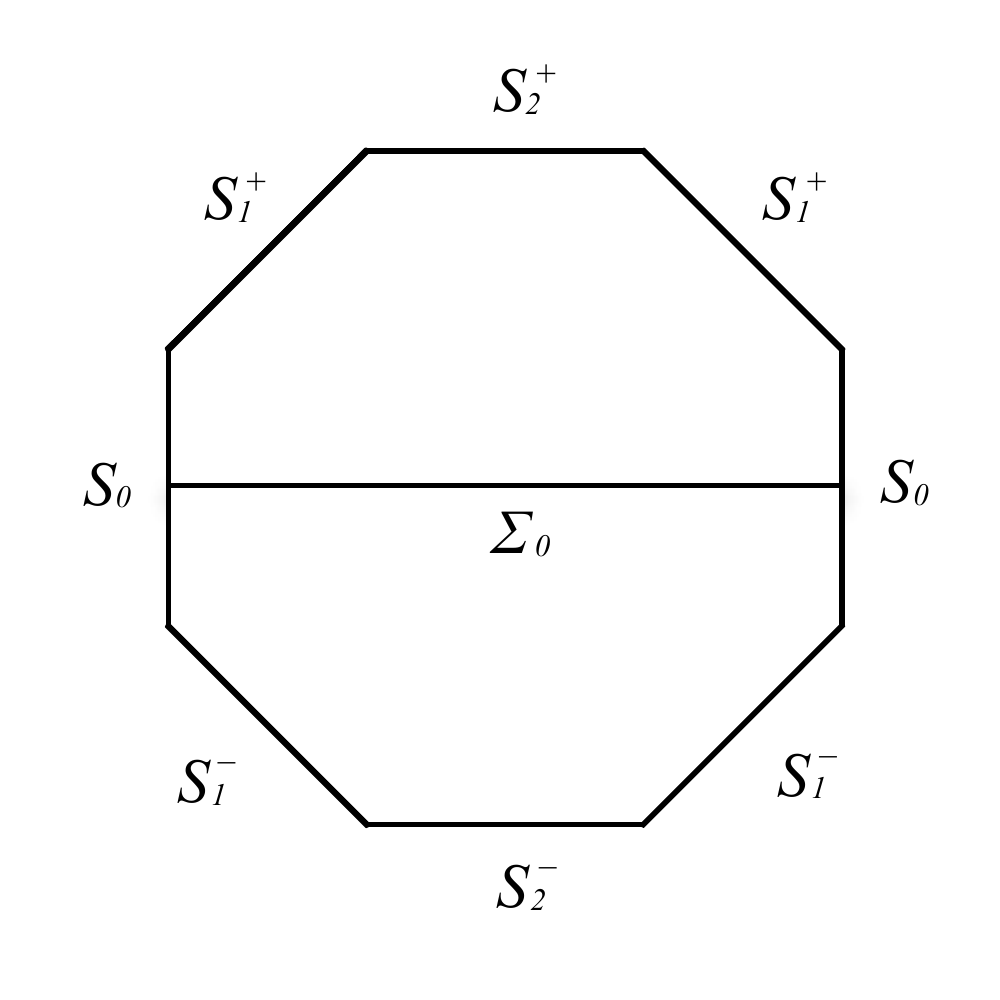}
\caption{the blown-up space $X=\left[\overline{\mathbb{R}_{t,x}^{1+n}}:\overline{\{|t|=|x|\}}\cap \partial\overline{\mathbb{R}_{t,x}^{1+n}}\right]$}\label{fig.1}
\end{figure}
Here $S_1^{\pm}$ is the compactification of null infinity and $\Sigma_0=\overline{\{t=0\}}$ is Cauchy surface. Let $\tilde{\rho}$ be the total boundary defining function and $\rho_0,\rho_1, \rho_2$ be the defining functions for corresponding boundary hypersurfaces $S_0, S_1^{\pm}, S_2^{\pm}$. With a conformal change $\tilde{g}=\tilde{\rho}^2g$ and $\tilde{h}=\tilde{\rho}^{\frac{1-n}{2}}h$, the reduced Einstein equations (\ref{eq.2}) are equivalent to 
\begin{equation}\label{eq.conf}
(\Box_{\tilde{g}}+\gamma(\tilde{h}))\tilde{h}_{\mu\nu} 
=\rho_1^{\frac{n-5}{2}} (\rho_0\rho_2)^{\frac{n-1}{2}} \tilde{F}_{\mu\nu}(\tilde{h},\tilde{h})
\end{equation}
where $\gamma(\tilde{h})= - \tilde{\rho}^{\frac{n-1}{2}}\Box_{\tilde{g}}\tilde{\rho}^{\frac{1-n}{2}}
$ and $\tilde{F}_{\mu\nu}(\tilde{h},\tilde{h})=\rho_1^{1-n}(\rho_0\rho_2)^{-1-n}F_{\mu\nu}(\tilde{\rho}^{\frac{n-1}{2}}\tilde{h}, \tilde{\rho}^{\frac{n-1}{2}}\tilde{h})$. See Section 5 for more details. 
Denote by $\widetilde{\mathcal{V}}^{N,\delta}_{\epsilon}$ the space of elements in the weighted b-Sobolev space 
$$\rho_0^{\frac{n-1}{2}+\delta}H_b^{N+1}(\Sigma_0)\times \rho_0^{\frac{n+1}{2}+\delta}H_b^{N+1}(\Sigma_0)$$
with norm less then $\epsilon$ and
satisfying the harmonic gauge conditions (\ref{harmonic.1}), for some $N\geq n+6$, $\frac{1}{2}>\delta>0$ and $\epsilon>0$ small. Then by energy estimate method we can show that for $n\geq 4$, if $(h^0,h^1)\in \widetilde{\mathcal{V}}^{N,\delta}_{\epsilon}$ then $\tilde{h}=\tilde{\rho}^{\frac{1-n}{2}}h$ is $C^{0,\delta}$ up to $S_1^{\pm}$ and hence the radiation field for $h$ is well defined which satisfies the corresponding harmonic gauge conditions:
\begin{equation}\label{harmonic.2}
\tilde{\Gamma}_{\mu}|_{S_1^{\pm}}=0
\end{equation}
for $\mu=0,1,...,n$, where $\tilde{\Gamma}_{\mu}=\tilde{\rho}^{\frac{1-n}{2}}\Gamma_{\mu}$. Combining this with the linear radiation field theory and the implicit function theorem, we further show that $\tilde{h}|_{S_1^{\pm}}$ also lies in some  weighted b-Sobolev space on $S_1^{\pm}$. More explicitly, denote by $\widetilde{\mathcal{W}}^{N,\delta}_{\epsilon'}$ the elements in 
$$
(\rho_0\rho_2)^{\delta} [H_b^{\frac{1}{2}-\delta}(\overline{\mathbb{R}};H^{N+\frac{1}{2}+\delta} (\mathbb{S}^{n-1}))\cap L^2(\mathbb{S}^{n-1};H_b^{N+1}(\overline{\mathbb{R}}))]
$$
with norm less than $\epsilon'$ and satisfying (\ref{harmonic.2}). Here we take the equivalence $S_1^{\pm}=\overline{\mathbb{R}}_{\tau}\times \mathbb{S}^{n-1}_{\theta}$.  Then $\tilde{h}|_{S_1^{\pm}}\in\widetilde{\mathcal{W}}^{N,\delta}_{\epsilon'}$ for some $\epsilon'>0$. Here both  $\widetilde{\mathcal{V}}^{N,\delta}_{\epsilon}$ and $\widetilde{\mathcal{W}}^{N,\delta}_{\epsilon'}$ are Banach manifolds. See Section \ref{sec.geosetting} for the definition of weighted b-Sobolev spaces and Section \ref{sec.initialdata}, \ref{sec.mollerop} for more details about $\widetilde{\mathcal{V}}^{N,\delta}_{\epsilon}$ and $\widetilde{\mathcal{W}}^{N,\delta}_{\epsilon'}$.

The main theorem of this paper is the following: 
\begin{theorem}\label{mainthm}
For $n\geq 4$, $N\geq n+6$, $\frac{1}{2}>\delta>0$ and $\epsilon>0$ small, if $(h^0,h^1)\in  \widetilde{\mathcal{V}}^{N,\delta}_{\epsilon}$, then the Einstein equation (\ref{eq.1}) has a global solution $h$ such that $\tilde{h}$ is $C^{0,\delta}$ up to $S_1^{\pm}$. Hence the Radiation field is well defined and 
the nonlinear M\o ller wave operator defines a continuous and open map
\begin{equation*}
\mathscr{R_{F}}^{\pm}:  \widetilde{\mathcal{V}}^{N,\delta}_{\epsilon}\ni(h^0,h^1)
\longrightarrow \tilde{h}|_{S^{\pm}_1}\in \widetilde{\mathcal{W}}^{N,\delta}_{C\epsilon}
\end{equation*}
for some constant $C>0$.
\end{theorem}

The open property guarantees the solution for characteristic initial value problem for Einstein Vacuum equations. 

\begin{theorem}
For $n\geq 4$, given characteristic data $\tilde{h}^{S_1^+}\in \widetilde{\mathcal{W}}^{N,\delta}_{\epsilon}$ for some $N\geq n+6, \delta\in (0,\frac{1}{2})$ and $\epsilon>0$ small enough, there exists an asymptotically flat solution $g$ to the Einstein vacuum equation (\ref{eq.1}) such that $\tilde{\rho}^{\frac{1-n}{2}}(g-m)|_{S_1^+}=\tilde{h}^{S_1^+}$. Moreover, $(t,x)$ are wave coordinates w.r.t. $g$ globally.
\end{theorem}

For $n=3$, the Cauchy data of interest for reduced Einstein vacuum equations has an asymptotic leading term $-M|x|^{-1}\delta_{ij}$ for some constant $M>0$ small, which has a long range effect at null infinity. Since $-M|x|^{-1}\delta_{ij}$ provides a solution to the linearization of (\ref{eq.2}) in a neighborhood of null infinity in $X$, we may expect that the essential change of the geometry of perturbed solution $g=m+h$ only comes from the asymptotic leading term, i.e. the constant $M$. By a different change of coordinates to $\tau=t-|x|-M\log |x|$, which was suggested by Friedlander to study the linear equation $\Box_{g}u=0$ with such background metric, I find a corresponding compactification of $\mathbb{R}_{t,x}^{1+n}$ such that the radiation filed is well defined in a similar way. This will be discussed in details in a separate paper. 


The outline of this paper is as follows: in Section \ref{sec.geosetting} we define a compactification of $\mathbb{R}_{t,x}^{1+n}$ and give basic geometric notations used through out the paper; in Section 3, we study the solution spaces of constraint equations  (\ref{constraint.1}) and  of harmonic gauge conditions (\ref{harmonic.1}) separately; in Section 4, we make a conformal transformation of the reduced Einstein equations and study the metric and wave operator in the conformal setting; in Section 5, we define a family of time-like functions, which will be used to define a special type of energy norm, called the weighted b-Sobolev norm; in Section 6, we do energy estimates for the equations (\ref{eq.conf}), show the regularity of Einstein solutions near boundary and define the radiation field; in Section 7, we modify the result obtained in Section 6 and obtain the isomorphism property for the nonlinear M\o ller wave operator. 

\textbf{Acknowledgment}: I want to thank to my adviser Richard Melrose for inspirations and valuable discussion on this topic. Without his help, this paper would not exist. I also would like to apologize to all who expressed interest in this work for the long delay between the first announcement of these results and the appearance of this monograph.

\vspace{0.2in}
\section{Geometric Setting}\label{sec.geosetting}

This section is a preparation for describing the problem and dealing with it in a manifold with corners arising from a suitable compactification of $\mathbb{R}^{1+n}_{t,x}$.  To simplify the notation, we take $t=x^0, x=(x^1,...,x^n)$ and use the lower case English alphabet $i,j,k,l,...$ as indices taking value in $\{1,2,...,n\}$ and the Greek alphabet $\alpha,\beta,\mu,\nu,...$ as indices taking value in $\{0,1,2,...,n\}$. Moreover, we use the capital English alphabet $I,J,K,L,...$ as multi indices. The change between superscript and subscript is taken over Minkowski metric $m=-d^2t+\sum_{i=1}^nd^2x^i$. For example,  $x_0=-x^0, x_i=x^i$.

\subsection{Compactification of the Space-time}

The usual radial compactification of $\mathbb{R}_{t,x}^{1+n}$, denoted by $X_0$,  can be realized by embedding 
	$$
 	\Phi: \mathbb{R}_{t,x}^{1+n}\ni (t,x) \longrightarrow (\frac{1}{\sqrt{1+t^2+|x|^2}},\frac{t}{\sqrt{1+t^2+|x|^2}}, \frac{x}{\sqrt{1+t^2+|x|^2}})\in \mathbb{R}^{2+n}
 	$$
and taking $X_0=\overline{\Phi(\mathbb{R}^{1+n}_{t,x})}$, which is a closed half sphere of dimension $n+1$.  Then all null rays w.r.t. Minkowski metric converge to  an embedded submanifold of $X_0$ on the bounday:
	 $$
	 L=\overline{\Phi(\{|t|=|x|\})}  \cap \partial X_0=\{(0,\pm\frac{1}{\sqrt{2}}, \pm\frac{x}{\sqrt{2}}): x\in \mathbb{R}^n,\ |x|=1\}.
	 $$
The compactification space $X$, which we use in this paper to study the asymptotic behavior of Einstein vacuum solutions that are close to Minkowski metric,  is $X_0$ blown up $ L$. Hence $X$  is a manifold with corners up to codimension $2$. See Figure \ref{fig.1}. 

More explicitly, let $d(p,q)$ be the distance function on $X_0$ w.r.t. standard spherical metric and $U_{\epsilon}=\{p\in X_0: d(p,L)<\epsilon\}$. Choose some $\epsilon>0$ small and define
	$$
	\begin{gathered}
	\Psi: X_0\backslash L \longrightarrow  X_0\backslash U_{\epsilon}
	\end{gathered}
	$$
as follows: if $d(p,L)>\epsilon$ then $\Psi(p)=p$; if $d(p,L)\leq \epsilon$ then
	$$
	d(\Psi(p),L)=d(\Psi(p),p)+d(p, L), \ d(\Psi(p),L) = \psi (d(p,L))
	$$
where $\psi: [0,\epsilon]\longrightarrow [\epsilon/2,\epsilon]$ is smooth and satisfies: 
$\psi(0)=\epsilon/2$, $\psi'>0$ and $\psi(d)=d$ for $7\epsilon/8$. Then the blowup can be realized by taking 
	\begin{equation}\label{eq.blowup}
	X=\overline{\Psi( X_0\backslash L)}=X_0\backslash U_{\epsilon}\subset \mathbb{R}^{2+n}.
	\end{equation}
	Obviously, $\Psi\Phi$ embeds $\mathbb{R}^{1+n}_{t,x}$ into $\mathbb{R}^{2+n}$ with image $\mathring{X}$, the interior of $X$. 
	
	This compactification is closely related to Penrose's diagram. The three essentially different boundary hypersurfaces can be represented with the three regions at infinity: spatial infinity $S_0$, null infinity $S_1^{\pm}$ and temporal infinity $S_2^{\pm}$ where the latter two both have forward and backward components. In particular, we can identify $S_0, S_1^{\pm}$ with $\overline{\mathbb{R}}\times \mathbb{S}^{n-1}$, where $\overline{\mathbb{R}}$ is the radial compactification of $\mathbb{R}$, and identify $S_2^{\pm}$ with closed ball of dimension $n$. Denote the set of boundary hyper surfaces of $X$ by
	$$
	M_1(X)=\{S_0, S_1^{\pm}, S_2^{\pm}\}. 
	$$

\begin{definition}\label{def.definingfun}
For $S\subset M_1(X)$,  we say $\rho_S$  a defining function for $S$ iff
	$$
	0\leq \rho_S \in C^{\infty}(X), \ S=\{\rho_S=0\}, \ d\rho_S|_S\neq 0. 
	$$
Two boundary defining functions $\rho_S, \rho'_S$ are equivalent on a domain $\Omega\subset X$  iff 
	$$
	\rho_S=e^{\psi}\rho'_S\ \textrm{for some}\ \psi\in C^{\infty}(\overline{\Omega}).
	$$
We say $\tilde{\rho}$ a total boundary defining function for $X$ iff
	$$
	\tilde{\rho}=\prod_{S\in M_1(X)} \rho_S.
	$$
\end{definition}
Throughout this paper, we use $\rho_0,\rho_1, \rho_2$ to  denote the defining functions for $S_0,S_1^{\pm}, S_2^{\pm}$ and $\tilde{\rho}=\rho_0\rho_1\rho_2$ the total boundary defining function for $\partial X$. We make a special choice of them in local coordinates in the following. 

\begin{figure}[htp]
\centering
\includegraphics[totalheight=2.5in,width=2.5in]{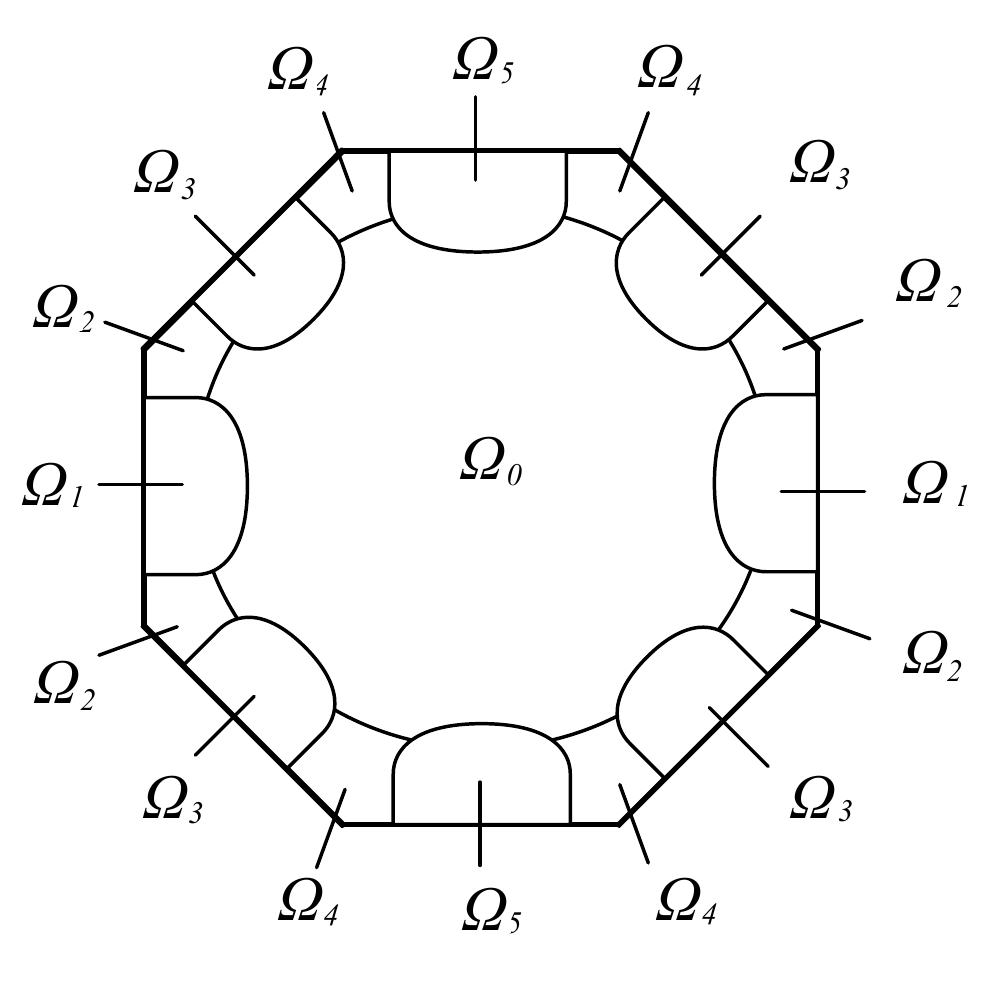}
\caption{The local coordinate charts on $X$}\label{fig.3}\end{figure}
Choose a covering $\{\Omega_i:0\leq i\leq 5\}$ for $X$ such that $\Omega_0$ is a bounded domain in the interior of $X$ and 
 $\bigcup_{i=1}^{5}\Omega_i$ covers $\partial X$.
The corresponding local coordinates can be chosen from the following functions:
	$$
	\begin{gathered}
	r=|x|,\ \tau=|t|-|x|,\ \rho=\frac{1}{|x|},\ s=\frac{t}{|x|},\ \phi=\frac{1}{|t|},\ y=\frac{x}{|t|},
	\\
	a=1-\frac{|t|}{|x|},\  b=\frac{1}{|x|-|t|},\  \bar{a}=1-\frac{|x|}{|t|},\  \bar{b}=\frac{1}{|t|-|x|},\ \theta^i=\frac{x^i}{|x|},\ \theta^0=1.
	\end{gathered}
	$$
More explicitly, 
\begin{itemize}
\item In the domain $\Omega_0=\{(t,x):|t|^2+|x|^2<1000\},$ choose
	$$
	\tilde{\rho}=\rho_0=\rho_1=\rho_2=1.
	$$
\item In the domain $\Omega_1=\{(s,\rho,\theta): - \frac{7}{8}< s<  \frac{7}{8}, 0\leq \rho <1, \theta\in\mathbb{S}^{n-1}\}$, choose
	$$
	\tilde{\rho}=\rho_0=\rho, \  \rho_2=\rho_1=1.
	$$
\item In the domain $\Omega_2=\{(a,b,\theta):0\leq a<\frac{7}{8}, 0\leq b<1, \theta\in \mathbb{S}^{n-1}\}$,  choose
	$$
	\tilde{\rho}=ab,\  \rho_0=b,\  \rho_1=a, \  \rho_2=1.
	$$
\item In the domain $\Omega_3=\{(\tau,\rho, \theta):-\tau_0\leq \tau\leq \tau_0, 0\leq \rho< 1, \theta\in \mathbb{S}^{n-1}\}$ for some $\tau_0>8$,  choose
	$$
	\tilde{\rho}=\rho_1=\rho, \  \rho_0=\rho_2=1.
	$$
\item In the domain $\Omega_4=\{(\bar{a},\bar{b},\theta):0\leq \bar{a}<\frac{7}{8}, 0\leq \bar{b}<1, \theta\in \mathbb{S}^{n-1}\}$, choose
	$$
	\tilde{\rho}=\bar{a}\bar{b}, \  \rho_1=\bar{a}, \  \rho_2=\bar{b}, \  \rho_0=1.
	$$
\item In the domain $\Omega_5=\{(\phi,y):0\leq \phi<1,|y|< \frac{7}{8}\}$,  choose
	$$
	\tilde{\rho}=\rho_2=\phi, \  \rho_1=\rho_0=1.
	$$
\end{itemize}
Notice that, in the intersection of any two domains, the different choices of defining functions are equivalent. In many cases, for example when doing energy estimates, equivalent defining functions give equivalent statements. So we change the choice of $\rho_i$ freely from one set of coordiantes to the other in the rest of this paper if they give equivalent results. However, when performing energy estimates, we may not been able to restrict the discussion in a single coordinate chart and hence have to consider some of them together. See Section 6 for more details.

We also denote by $[X]^2$ a double space of $X$ across the null infinity $S^+_1\cup S_1^-$ such that $[X]^2$ is a manifold with boundary and the smooth structure on $X$ extends to $[X]^2$, i.e.  $C^{\infty}(X)=C^{\infty}([X]^2)|_{X}$. Here $[X]^2$ is not uniquely determined. Refer to \cite{Me} for more details. The boundary of $[X]^2$ consists of three components,  $[S_0]^2$, $[S_2^+]^2$ and $[S_2^-]^2$, which are  also double spaces of $S_0$, $S_2^+$ and $S_2^-$ across their boundaries and hence are close manifolds. Moreover, $[S_0]^2$ can be identified with $\mathbb{S}^1\times \mathbb{S}^{n-1}$ and $[S_2^{\pm}]$ can be identified with $\mathbb{S}^n$. We introduce $[X]^2$ for the reason that the conformal Minkowski metric, or with small perturbation, extends to a b-type Lorentz metric on $[X]^2$. See Section \ref{sec.conf} for more details.

\subsection{Vector Fields.}
Let $\mathscr{V}_b(X)$ (resp. $\mathscr{V}_b([X]^2)$) be the space of smooth vector fields on $X$ which are tangent to $\partial X$ (resp. $\partial [X]^2$). The relation between $\mathscr{V}_b([X]^2)|_X$ and $\mathscr{V}_b(X)$ is 
\begin{equation*}
\mathscr{V}_b([X]^2)|_X = \mathscr{V}_b(X) + C^{\infty}(X)V_1. 
\end{equation*}
where $V_1\in \mathscr{V}_b([X]^2)|_{X}$ transverse to $S_1^{\pm}$, i.e. for any $p\in S_1^{\pm}$, $V_1(p)\notin T_pS_1^{\pm}$.  In particular, in domain $\Omega_2,\Omega_3,\Omega_4$, we can choose $V_1=\partial_{\rho_1}$. Set
\begin{equation*}
\begin{gathered}
\partial_0=\partial_t,\quad \partial_i=\partial_{x^i},\quad \slashpar_i=r(\delta_i^j-\theta_i\theta^j)\partial_j,
\\ 
Z_{00}=r\partial_r+t\partial_t, \quad Z_{0i}=t\partial_i+x_i\partial_t, \quad Z_{ij}=x_{i}\partial_{j}-x_{j}\partial_{i}.
\end{gathered}
\end{equation*}
Here $\slashpar_i$ are the projection of $\partial_i$ on $T\mathbb{S}^{n-1}\subset T\mathbb{R}^n$ satisfying $\theta^i\slashpar_i=0$. 
We denote by $\tilde{Z}$ any vector field in $\{\partial_{\mu}, Z_{\mu\nu}: \mu, \nu=0,...,n\}$ and $\tilde{\partial}^k$ any element in $\mathscr{V}^k_b(X)$ for $k\geq 1$. 
\begin{lemma}\label{lem.0}
The vector fields defined above satisfy the following properties:
\begin{equation*}
\begin{aligned}
(\mathrm{i})&\quad
\mathscr{V}(\mathbb{S}^{n-1}) =\mathrm{Span}_{C^{\infty}(\mathbb{S}^{n-1})}\{Z_{ij}: i,j=1,...,n\};
\\
(\mathrm{ii})&\quad
\mathscr{V}_b(X) = \mathrm{Span}_{C^{\infty}(X)} \{\partial_{\mu},Z_{\mu\nu}:\mu,\nu=0,1,...,n\};
\\
(\mathrm{iii})&\quad
\partial_{\mu} \in \rho_0\rho_2\mathscr{V}_b(X),\ \mathrm{for}\ \mu=0,...,n.
\end{aligned}
\end{equation*}
%
\end{lemma}
\begin{proof}
The first property follows directly by $
Z_{ij}=\theta_i\slashpar_j-\theta_j\slashpar_i\in T\mathbb{S}^{n-1}$ and $ \slashpar_j=Z_{ij}\theta^i$. 
To prove (ii) and (iii), we only need to check them in local coordinates near $\partial X$:
\begin{itemize}
\item In $\Omega_1$ with coordinates $(s,\rho,\theta)$,  $\mathscr{V}_b(\Omega_1)=Span_{C^{\infty}(\Omega_1)} \{\partial_s, \rho\partial_{\rho}, \slashpar_i:1\leq i\leq n\}$. Hence (ii) and (iii)  follow by
\begin{equation*}
\begin{gathered}
\partial_t=\rho\partial_s,\  
\partial_i=\rho(-\theta_i(\rho\partial_{\rho}+s\partial_s)+\slashpar_i),
\\
Z_{00}=-\rho\partial_{\rho},\  
Z_{0i}=\theta_i(1-s^2)\partial_s-\theta_is\rho\partial_{\rho}+s\slashpar_i,
\\ 
\Longrightarrow\quad 
\partial_s=(1-s^2)^{-1}(Z_{0i}\theta^i-sZ_{00}),\
\rho\partial_{\rho}=-Z_{00}.
\end{gathered}
\end{equation*}
\item In $\Omega_2$ with coordinates $(a,b,\theta)$,  $\mathscr{V}_b(\Omega_2)=Span_{C^{\infty}(\Omega_2)} \{a\partial_a, b\partial_b, \slashpar_i:1\leq i\leq n\}$. Hence (ii) and (iii) follow by
\begin{equation*}
\begin{gathered}
\partial_t=b(b\partial_b-a\partial_a),\ 
\partial_i=-b\theta_i(b\partial_b-a\partial_a)+ab(-\theta_ia\partial_a+\slashpar_i),
\\
Z_{00}=-b\partial_b,\  
Z_{0i}=-\theta_i((2-a)a\partial_a-b\partial_b)+(1-a)\slashpar_i,
\\
\Longrightarrow\quad  b\partial_b=-Z_{00},\ 
a\partial_a=-(2-a)^{-1}(Z_{00}+Z_{0i}\theta^i).
\end{gathered}
\end{equation*}
\item In $\Omega_3$ with coordinates $(\tau,\rho,\theta)$, $\mathscr{V}_b(\Omega_3)=Span_{C^{\infty}(\Omega_3)} \{\partial_{\tau}, \rho\partial_{\rho}, \slashpar_i:1\leq i\leq n\}$. Hence (ii) and (iii)follow by
\begin{equation*}
\begin{gathered}
\partial_t=\partial_{\tau},\ \partial_i= -\theta_i\partial_{\tau}+\rho(-\theta_i\rho\partial_{\rho}+\slashpar_i),
\\
Z_{00}=\tau\partial_{\tau}-\rho\partial_{\rho},\
Z_{0i}=-\theta_i(\tau\partial_{\tau}+(1+\rho\tau)\rho\partial_{\rho}) +(1+\rho\tau)\slashpar_i,
\\
\Longrightarrow\quad  \partial_{\tau}=\partial_t,\  
\rho\partial_{\rho}=-(2+\rho\tau)^{-1}(Z_{00}+Z_{0i}\theta^i).
\end{gathered}
\end{equation*}
\item In $\Omega_4$ with coordinates $(\bar{a},\bar{b},\theta)$, $\mathscr{V}_b(\Omega_4)=Span_{C^{\infty}(\Omega_4)} \{\bar{a}\partial_{\bar{a}}, \bar{b}\partial_{\bar{b}}, \slashpar_i:1\leq i\leq n\}$. Hence (ii) and (iii) follow by
\begin{equation*}
\begin{gathered}
\partial_t=-\bar{b}(\bar{b}\partial_{\bar{b}}-\bar{a}\partial_{\bar{a}})-\bar{a}\bar{b}(\bar{a}\partial_{\bar{a}}),\  
\partial_i=\bar{b}\theta_i(\bar{b}\partial_{\bar{b}}-\bar{a}\partial_{\bar{a}}) +\bar{a}\bar{b}(1-\bar{a})^{-1}\slashpar_i,
\\
Z_{00}=-\bar{b}\partial_{\bar{b}},\ 
Z_{0i}=-\theta_i((2-\bar{a})\bar{a}\partial_{\bar{a}} -\bar{b}\partial_{\bar{b}})+(1-\bar{a})^{-1}\slashpar_i,
\\
\Longrightarrow \quad
\bar{b}\partial_{\bar{b}}=-Z_{00},\  
\bar{a}\partial_{\bar{a}}=-(2-\bar{a})^{-1}(Z_{00}+Z_{0i}\theta^i).
\end{gathered}
\end{equation*}
\item In $\Omega_5$ with coordinates $(\phi,y)$, $\mathscr{V}_b(\Omega_5)=Span_{C^{\infty}(\Omega_5)} \{\phi\partial_{\phi}, \partial_{y^i}:1\leq i\leq n\}$. Hence (ii) and (iii) follow by
\begin{equation*}
\begin{gathered}
\partial_{t}=-\phi(\phi\partial_{\phi}+y^i\partial_{y^i}),\ \partial_i=\phi\partial_{y^i},
\\
Z_{00}=-\phi\partial_{\phi},\ Z_{0i}=\partial_{y^i}-y_i(\phi\partial_{\phi}+y^j\partial_{y^j}),
\\
\Longrightarrow \quad \phi\partial_{\phi}=-Z_{00},\ \partial_{y^i}= Z_{0i}-(1-|y|^2)^{-1}y_i(Z_{00}-Z_{0j}y^j).
\end{gathered}
\end{equation*}
\end{itemize}
We finish the proof.
\end{proof}

\subsection{Sobolev Spaces and Symbol Spaces}
We define the b-type Sobolev space and conormal symbol space on $X$ in the following. These spaces are used to characterize the regularity and asymptotic properties near boundary for Einstein vacuum solutions. 
Let $m_0$ be the Riemannian metric on $X$ induced from Euclidean metric on $\mathbb{R}^{2+n}$ via (\ref{eq.blowup}) and denote by $dvol_{m_0}$ the corresponding volume form.
\begin{definition}\label{def.1}
Define the b-Sobolev space $H^N_b(X)$ for any $N\in\mathbb{N}_0$  the completion of $C_c^{\infty}(\mathring{X})$ w.r.t. the norm: 
$$
\|v\|_{H^N_b(X)} = \left(\int_{X} \sum_{|I|\leq N}|\tilde{Z}^I v|^2 \frac{dvol_{m_0}}{\tilde{\rho}}\right)^{\frac{1}{2}}.
$$
Define the conormal symbol space by 
$$
\begin{aligned}
\mathscr{A}^{0,0,0}(X)&=\{v\in C^{-\infty}(X): \mathscr{V}_b(X)^lv\in L^{\infty}(X),\ \forall\ l\in \mathbb{N}_0\},\\
\mathscr{A}^{c_1,c_2,c_3}(X)&=\rho_0^{c_0}\rho_1^{c_1}\rho_2^{c_2}\mathscr{A}^{0,0,0}(X) 
\end{aligned}
$$
for $c_1,c_2,c_3\in \mathbb{C}$. 
\end{definition}
The relationship between the b-Sobolev spaces and the conormal symbol spaces are stated in the following lemma. 
\begin{lemma}\label{lem.2.3}
If $\Re{c_0},\Re{c_1},\Re{c_2}>0$ , then
\begin{equation*}
\mathscr{A}^{c_1,c_2,c_3}(X) \
\subset \bigcap_{N=0}^{\infty}H_b^N(X).
\end{equation*}
Moreover, there exists a constant $C_N$ for any $N> \frac{1}{2}\mathrm{dim}(X)$ such that
\begin{equation*}
\|v\|_{L^{\infty}(X)}\leq C_N\|v\|_{H_b^N(X)},
\end{equation*}
which implies that 
\begin{equation*}
\bigcap_{N=0}^{\infty} H_b^N(X) \subset \mathscr{A}^{0,0,0}(X). 
\end{equation*}
\end{lemma}

Definition \ref{def.1} and Lemma \ref{lem.2.3} can be generalize to any manifold with corners, including closed $p$-submanifolds of $X$ with induced Riemannian metric and boundary defining functions.  See \cite{Me} for definition of $p$-submainifold. In this paper, we mainly consider the Cauchy surface $\Sigma_0=\overline{\{t=0\}}=\overline{\mathbb{R}^n}$ and  null infinity $S^{\pm}_1$, which are both $p$-submanifolds of $X$. We study the map from Cauchy data to Characteristic data for Einstein vacuum equations, which lie in some weighted b-Sobolev spaces or cornormal symbol space on $\Sigma_0$ and $S^{\pm}_1$ correspondingly.

\vspace{0.2in}
\section{Initial Data}\label{sec.initialdata}
In this section, we study the spaces of small asymptotically flat initial data for Einstein vacuum equations which satisfy either the constraint equations (\ref{constraint.1}) or the harmonic gauge conditions (\ref{harmonic.1}), and the relation between them. 

Denote by $g=m+h$ a small perturbation of Minkowski metric: 
$g=g_{\mu\nu}dx^{\mu}dx^{\nu}, h=h_{\mu\nu}dx^{\mu}dx^{\nu}$. We can view $g,h$ as functions of $(t,x)$ valued in $(n+1)\times(n+1)$ symmetric matrices.  Then (\ref{constraint.1}) and (\ref{harmonic.1}) can be written in terms of the components of $(h^0, h^1)=(h|_{t=0},\partial h|_{t=0})$, which turn out to be undetermined elliptic systems after suitable decomposition. 

\subsection{Initial Data Subject to the Constraint Equations (\ref{constraint.1})}
Let us first derive the constraint equations (\ref{constraint.1}) in terms of components of $(h^0,h^1)$. Denote by
\begin{equation*}
\begin{aligned}
&\Gamma_{\alpha\beta\mu}=2g_{\mu\nu}\Gamma^{\nu}_{\alpha\beta} =\partial_{\alpha}h_{\mu\beta}+\partial_{\beta}h_{\alpha\mu}-\partial_{\mu}h_{\alpha\beta}, \\
&\Gamma_{\mu}=g_{\mu\nu}g^{\alpha\beta}\Gamma^{\nu}_{\alpha\beta} =g^{\alpha\beta}(\partial_{\alpha}h_{\mu\beta}-\tfrac{1}{2}\partial_{\mu}h_{\alpha\beta}).
\end{aligned}
\end{equation*}
Then the components of Ricci curvature of $g$ can be expressed as
\begin{equation*}
\begin{aligned}
R_{\mu\nu}=&\ \partial_{\lambda}\Gamma^{\lambda}_{\mu\nu}-\partial_{\mu}\Gamma^{\lambda}_{\lambda\nu}+\Gamma^{\lambda}_{\mu\nu}\Gamma^{\delta}_{\lambda\delta} -\Gamma^{\delta}_{\mu\lambda}\Gamma^{\lambda}_{\nu\delta}
\\
=&\ \tfrac{1}{2}g^{\lambda\delta}(-\partial_{\lambda}\partial_{\delta}h_{\mu\nu}+\partial_{\lambda}\partial_{\nu}h_{\mu\delta}+\partial_{\lambda}\partial_{\mu}h_{\delta\nu}-\partial_{\mu}\partial_{\nu}h_{\lambda\delta}) +E_{\mu\nu}\\
=&\ 
-\tfrac{1}{2}g^{\lambda\delta}\partial_{\lambda}\partial_{\delta}h_{\mu\nu} +\tfrac{1}{2}\partial_{\mu}\Gamma_{\nu}+\tfrac{1}{2}\partial_{\nu}\Gamma_{\mu}
+\tfrac{1}{2}F_{\mu\nu}
\end{aligned}
\end{equation*}
where 
\begin{equation*}\label{notation.1}
\begin{aligned}
F_{\mu\nu}=&\ 2E_{\mu\nu}+D_{\mu\nu}+D_{\nu\mu},
\\
E_{\mu\nu}=&\ \Gamma^{\lambda}_{\mu\nu}\Gamma^{\delta}_{\lambda\delta} -\Gamma^{\delta}_{\mu\lambda}\Gamma^{\lambda}_{\nu\delta}
+ \tfrac{1}{2}(\partial_{\lambda}g^{\lambda\delta})\Gamma_{\mu\nu\delta} 
-\tfrac{1}{2}(\partial_{\mu}g^{\lambda\delta})\Gamma_{\lambda\nu\delta},
\\
D_{\mu\nu}=& 
-(\partial_{\mu}g^{\lambda\delta})(\partial_{\lambda}h_{\nu\delta} -\tfrac{1}{2}\partial_{\nu}h_{\lambda\delta}).
\end{aligned}
\end{equation*}
Denote by 
\begin{equation*}
D'_{\mu\nu}=D_{\mu\nu}|_{t=0}, \quad 
E'_{\mu\nu}=E_{\mu\nu}|_{t=0}, \quad 
F'_{\mu\nu}=F_{\mu\nu}|_{t=0}.
\end{equation*}
Then $D'_{\mu\nu}, E'_{\mu\nu}, F'_{\mu\nu}$ are quadratic forms of $(\partial_ih^0,h^1)$ with coefficients analytically depending on $h^0$.

%
%

The Einstein vacuum equations (\ref{eq.1}) imply that
\begin{equation*}
\begin{aligned}
g^{00}R_{00}-g^{ij}R_{ij}=0, \quad g^{00}R_{0k}+g^{0j}R_{kj}=0
\end{aligned}
\end{equation*}
for $k=1,...,n$, which cancel the $\partial_0^2h$ terms and impose $(n+1)$ constraint equations on $(h^0, h^1)$ at $t=0$: 
\begin{equation}\label{constraint.2}
\left\{
\begin{aligned}
&g^{pl}g^{ij}\partial_i(\partial_jh^0_{pl}-\partial_ph^0_{jl})
+g^{pl}g^{0i}(\partial_ih^1_{pl}-\partial_ph^1_{il}-\partial_i\partial_ph^0_{0l}+\partial_p\partial_l h^0_{i0})
+ g^{00}E'_{00}-g^{pl}E'_{pl}=0,
\\
&g^{00}g^{ij}(\partial_ih^1_{kj}-\partial_kh^{1}_{ij}
-\partial_i\partial_jh^0_{k0}+\partial_k\partial_ih^0_{0j})
+g^{0l}g^{0i}(\partial_kh^1_{il}-\partial_ih^{1}_{kl}
+\partial_i\partial_lh^0_{k0}-\partial_k\partial_lh^0_{0i})
\\
&\quad\quad\quad\quad\quad\  \ 
+g^{0l}g^{ij}(-\partial_i\partial_jh^0_{kl}+ \partial_i\partial_kh^0_{jl}+\partial_i\partial_lh^{0}_{kj}
-\partial_k\partial_lh^0_{ij})+
2(g^{00}E'_{0k}+g^{0l}E'_{kl})=0. 
\end{aligned}
\right.
\end{equation}
for $k=1,...,n$. Here (\ref{constraint.2}) is the system of constraint equations (\ref{constraint.1}) in fixed coordinates. We are interested in small solutions to (\ref{constraint.2}) in the following space:
\begin{equation}\label{initialspace}
\rho_0^{\lambda}H_b^{N+1}(\overline{\mathbb{R}^n}: \mathbb{R}^{\frac{(n+1)(n+2)}{2}}) \times \rho_0^{\lambda+1}H_b^{N}(\overline{\mathbb{R}^n}: \mathbb{R}^{\frac{(n+1)(n+2)}{2}}).
\end{equation}
Let
\begin{equation*}
\widetilde{\mathcal{U}}_{\epsilon}^{N, \lambda}
=\{ (h_0,h_1):\ (h_0,h_1)\ \textrm{satisfies (\ref{constraint.2}) and } \|(h^0,h^1)\|_{\rho_0^{\lambda}H_b^{N+1}(\overline{\mathbb{R}^n}) \times \rho_0^{\lambda+1}H_b^{N}(\overline{\mathbb{R}^n})}<\epsilon \}.
\end{equation*}
\begin{proposition}\label{prop.2}
For $n\geq 3$, $\lambda\in (0,n-2)$ and $\epsilon>0$ small enough, $\widetilde{\mathcal{U}}_{\epsilon}^{N,\lambda}$ is a Fr\'{e}chet manifolds.
\end{proposition}
\begin{proof}
We prove the statement by studying the linearization of (\ref{constraint.2}) at $(0,0)$:
\begin{equation}\label{constraint.3}
\left\{
\begin{aligned}
&-\triangle\sum_i h^0_{ii}-\sum_{ij}\partial_i\partial_jh^0_{ij}=0,
\\
&-\triangle h^0_{0k}-\partial_k\sum_i\partial_ih^0_{0i}
+\partial_k\sum_ih^1_{ii}-\sum_i\partial_ih^1_{ki}=0, \quad\textrm{for $k=1,...,n$}
\end{aligned}
\right.
\end{equation}
where $\triangle=-\sum_{j=1}^n \partial_j^2$,  and applying implicit function theorem.
Notice that by \cite{Me2}, for $n\geq 3$ and $\lambda\in(0,n-2)$, 
\begin{equation*}
\triangle: \rho_0^{\lambda}H_b^{N+1}(\overline{\mathbb{R}^n})\longrightarrow \rho_0^{\lambda+2}H_b^{N-1}(\overline{\mathbb{R}^n})
\end{equation*} 
is an isomorphism and hence
\begin{equation*}
\triangle^{\frac{1}{2}}: \rho_0^{\lambda}H_b^{N+1}(\overline{\mathbb{R}^n})\longrightarrow \rho_0^{\lambda+1}H_b^{N}(\overline{\mathbb{R}^n}) 
\quad\mathrm{and}\quad 
\triangle^{\frac{1}{2}}:
\rho_0^{\lambda+1}H_b^{N+1}(\overline{\mathbb{R}^n})\longrightarrow \rho_0^{\lambda+2}H_b^{N}(\overline{\mathbb{R}^n}), 
\end{equation*} 
are both isomorphism. See Proposition \ref{app.1} in Appendix. Assume $(h^0,h^1)$ is in space (\ref{initialspace}) and take a linear transformation of as follows: for $1\leq k\leq n$, $2\leq l\leq n$, $1\leq i\neq j\leq n$,
%
\begin{equation*}
\begin{gathered}
A_1=\tfrac{n-1}{n}\sum_ih^0_{ii},\quad
A_l=\tfrac{1}{n}\sum_ih^0_{ii}-h^0_{ll}, \quad 
A_{ij}=h^0_{ij},
\\
C_1=\tfrac{n-1}{n}\sum_ih^1_{ii},\quad
C_l=\tfrac{1}{n}\sum_i h^1_{ii} - h^1_{ll}, \quad 
C_{ij}=h^1_{ij}.
\end{gathered}
\end{equation*}
Then the first equation in (\ref{constraint.3}) becomes
\begin{equation}\label{eq.4}
\triangle A_1=-\partial_1^2\sum_{j=2}^nA_j+ \sum_{j=2}^n\partial_j^2A_j -\sum_{i\neq j} \partial_i\partial_jA_{ij}. 
\end{equation}
Denote by $B=h^0_{0k}dx^k$ and $D=\triangle^{-1}(\partial_k\sum_ih^1_{ii}-\sum_i\partial_ih^1_{ki})dx^k$, $1$-forms on the Euclidean space, i.e.
\begin{equation*}
\begin{gathered}
B, D\in \rho_0^{\lambda}H_b^{N+1}(\overline{\mathbb{R}^n}:\Lambda^1(\mathbb{R}^n)).
\end{gathered}
\end{equation*}
Then the second equation in (\ref{constraint.3}) can be expressed as
\begin{equation}\label{eq.3}
-\triangle B+ d\delta B+\triangle D=0, 
\end{equation}
Since $\triangle$ has trivial kernel on $\rho_0^{\lambda}H_b^{N+1}(\overline{\mathbb{R}^n}:\Lambda^p(\mathbb{R}^n)) $ for all $0\leq p\leq n$, the Hodge decomposition theorem implies
\begin{equation*}
\rho_0^{\lambda}H_b^{N+1}(\overline{\mathbb{R}^n}:\Lambda^1(\mathbb{R}^n)) 
=\mathcal{E}^{N+1,\lambda}_c\varoplus \mathcal{E}^{N+1,\lambda}_{cc}
\end{equation*}
with $\mathcal{E}^{N+1,\lambda}_c$ and $\mathcal{E}^{N+1,\lambda}_{cc}$ the subspace of closed and coclosed forms respectively.  More explicitly,  let
\begin{equation*}
\begin{aligned}
&\Psi: \rho^{\lambda+1}H_b^{N}(\overline{\mathbb{R}^n})\ni f \longrightarrow 
\triangle^{-1}df\in \rho^{\lambda}H_b^{N+1}(\overline{\mathbb{R}^n}: \Lambda^1(\mathbb{R}^n)),
\\
&\Phi: \rho^{\lambda}H_b^{N+1}(\overline{\mathbb{R}^n}:\Lambda^1(\mathbb{R}^n))\ni u\longrightarrow \delta u\in \rho^{\lambda+1}H_b^{N}(\overline{\mathbb{R}^n}).
\end{aligned}
\end{equation*}
Then $\Phi\Psi=Id$ and 
$$ \Psi\Phi: \rho_0^{\lambda}H_b^{N+1}(\overline{\mathbb{R}^n}:\Lambda^1(\mathbb{R}^n)) \longrightarrow \mathcal{E}^{N+1,\lambda}_c $$
is a projection. Writing $B=B'+B''$ and $D=D'+D''$ with $B', D'$ closed and $B'',D''$ coclosed, i.e.
\begin{equation*}
B'=\Psi\Phi(B)
, \quad
B''=(Id-\Psi\Phi)(B),\quad
D'=\Psi\Phi(D)
, \quad D''=(Id-\Psi\Phi)(D),
\end{equation*}
then (\ref{eq.3}) split into 
\begin{equation}\label{eq.5}
\triangle D'=0,\quad -\triangle B''+\triangle D''=0.
\end{equation}
For the first equation in (\ref{eq.5}), notice that we have the following commutative diagram:
\begin{displaymath}
\xymatrix {
 \rho^{\lambda+1}H_b^{N}(\overline{\mathbb{R}^n}:\Lambda(\mathbb{R}^n)) 
 \ar[r] ^{d,\delta} \ar[d]_{\triangle^{-\frac{1}{2}}}
  & \rho^{\lambda+2}H_b^{N-1}(\overline{\mathbb{R}^n}:\Lambda(\mathbb{R}^n)) 
  \ar[d]^{\triangle^{-\frac{1}{2}}}
   \\
  \rho^{\lambda}H_b^{N+1}(\overline{\mathbb{R}^n}:\Lambda(\mathbb{R}^n)) 
  \ar[r]_{d,\delta} &  \rho^{\lambda+1}H_b^{N}(\overline{\mathbb{R}^n}:\Lambda(\mathbb{R}^n))
  }
\end{displaymath}
where $\Lambda(\mathbb{R}^n)=\sum_{p=0}^n\Lambda^p(\mathbb{R}^n) $.
Since $h^1\in \rho^{\lambda+1}H_b^{N}(\overline{\mathbb{R}^n}:\Lambda^1(\mathbb{R}^n))$, 
$$
\begin{aligned}
\triangle D' &= \triangle \Psi\Phi (D) =d\delta \triangle ^{-1}\big((\partial_k\sum_ih^1_{ii}-\sum_i\partial_ih^1_{ki})dx^k\big)
\\
&=d\triangle^{-\frac{1}{2}}\delta \big((\partial_k\sum_i \triangle^{-\frac{1}{2}}h^1_{ii}-\sum_i\partial_i \triangle^{-\frac{1}{2}}h^1_{ki})dx^k\big)
\\
&=d\triangle^{-\frac{1}{2}} (\triangle\sum_i \triangle^{-\frac{1}{2}}h^1_{ii}+ \sum_{i,k}\partial_k\partial_i \triangle^{-\frac{1}{2}}h^1_{ki})
\end{aligned}
$$
Hence $\triangle D'=0$ is equivalent to  
$$
\triangle\sum_i \triangle^{-\frac{1}{2}}h^1_{ii}+ \sum_{i,k}\partial_k\partial_i \triangle^{-\frac{1}{2}}h^1_{ki}=0.
$$
Substitute $h^1_{ij}$ by $C$,  (\ref{constraint.3}) is equivalent to the elliptic system:
\begin{equation}\label{constraint.4}
\left\{
\begin{aligned}
&\triangle A_1=-\partial_1^2\sum_{j=2}^nA_j +\sum_{j=2}^n\partial_j^2A_j -\sum_{i\neq j} \partial_i\partial_jA_{ij},
\\
&\triangle^{\frac{1}{2}} C_1=-\partial_1^2\sum_{j=2}^n \triangle^{-\frac{1}{2}} C_j +\sum_{j=2}^n\partial_j^2\triangle^{-\frac{1}{2}} C_j -\sum_{i\neq j} \partial_i\partial_j \triangle^{-\frac{1}{2}} C_{ij},
\\
&\triangle B''=\partial_1(C_1-\sum_{l=2}^nC_l)dx^1 +\sum_{k=2}^n\partial_k(C_1+C_l)dx^k
+\sum_{i\neq k}\partial_iC_{ki}dx^k.
\end{aligned}
\right.
\end{equation}
In (\ref{constraint.3}), the components $\{h^0_{00}, h^1_{00}, h^1_{0k}: 1\leq k\leq n\}$ are free. To solve $\{h^0_{0k}, h^0_{ij}, h^1_{ij}: 1\leq i,j,k\leq n\}$ from (\ref{constraint.3}), it is equivalent to solve $\{B_k, A_k, A_{ij}, C_i, C_{ij}:1\leq k, i\neq j\leq n\}$ from (\ref{constraint.4}). 
Given arbitray $h^0_{00},A_k, A_{ij}\in \rho_0^{\lambda}H_b^{N+1}(\overline{\mathbb{R}^n})$ and $h^1_{0\mu},C_k, C_{ij}\in \rho_0^{\lambda+1}H_b^{N}(\overline{\mathbb{R}^n})$ for $0\leq\mu\leq n, 2\leq k\leq n, 1\leq i\neq j\leq n$ and $B'\in \mathcal{E}^{N+1,\lambda}_c$, we can solve (\ref{constraint.4}) uniquely and have
\begin{equation*}
A_1\in \rho_0^{\lambda}H_b^{N+1}(\overline{\mathbb{R}^n}),\quad
C_1\in \rho_0^{\lambda+1}H_b^{N}(\overline{\mathbb{R}^n}), \quad
B''\in \mathcal{E}^{N+1,\lambda}_{cc}. 
\end{equation*}
And hence the solution space of (\ref{constraint.3}), denoted by $\mathcal{U}^{N,\lambda}$, is isomorphism to 
\begin{equation*}
\rho_0^{\lambda}H_b^{N+1}(\overline{\mathbb{R}^n}: \mathbb{R}^{\frac{n(n+1)}{2}}) \times \mathcal{E}^{N+1,\lambda}_{c}\times \rho_0^{\lambda+1}H_b^{N}(\overline{\mathbb{R}^n}: \mathbb{R}^{\frac{(n+1)(n+2)}{2}-1})
\end{equation*}
Due to the ellipticity of (\ref{constraint.4}) and invertibility of $\triangle$ and $\triangle^{\frac{1}{2}}$, when $\epsilon>0$ small enough, 
$\widetilde{\mathcal{U}}^{N,\lambda}_{\epsilon}$ is a submanifold in space (\ref{initialspace}), which has tangent space $\mathcal{U}^{N,\lambda}$ at $(0,0)$. 
\end{proof}

\subsection{Initial Data Subject to Harmonic Gauge Conditions (\ref{harmonic.1})}
We first deriving the equations (\ref{harmonic.1}) in terms of components of $(h^0,h^1)$. 
First $\Gamma_{\mu}|_{t=0}=0$ is equivalent to 
\begin{equation}\label{constraint.6}
\left\{
\begin{aligned}
&\tfrac{1}{2}g^{00}h^1_{00}-\tfrac{1}{2}g^{ij}h^1_{ij}+g^{i\beta}\partial_ih^0_{0\beta}=0,\\
&g^{0\beta}h^1_{k\beta}+g^{i\beta}\partial_ih^0_{k\beta}
-\tfrac{1}{2}g^{\alpha\beta}\partial_kh^0_{\alpha\beta}=0, 
\end{aligned}
\right.
\end{equation}
for $k=1,...,n$. 
Secondly $\partial_t\Gamma_{\mu}|_{t=0}$ is equivalent to
\begin{equation}\label{eq.7}
\left\{
\begin{aligned}
& (\tfrac{1}{2}g^{00}\partial_t^2h_{00}-\tfrac{1}{2}g^{ij}\partial_t^2h_{ij})|_{t=0} 
+g^{i\beta}\partial_ih^1_{0\beta}- D'_{00}=0, 
\\
& (g^{0\beta}\partial_t^2h_{k\beta})|_{t=0} 
+ g^{i\beta}\partial_ih^1_{k\beta}-\tfrac{1}{2} g^{\alpha\beta}\partial_kh^1_{\alpha\beta}
-D'_{0k}=0,
\end{aligned}
\right.
\end{equation}
for $k=1,...,n$. 
When assuming $\Gamma_{\mu}=0$, the Reduced Einstein vacuum equations (\ref{eq.2}) at $t=0$ is 
\begin{equation}\label{eq.0}
g^{00}\partial_t^2h_{\mu\nu}|_{t=0}=-2g^{0i}\partial_ih^1_{\mu\nu} -g^{ij}\partial_i\partial_jh^0_{\mu\nu}+F'_{\mu\nu}.
\end{equation}
Multiplying (\ref{eq.7}) by $g^{00}$ and substituting $(g^{00}\partial^2_th_{\mu\nu})|_{t=0}$ terms by the right hand side of (\ref{eq.0}),
we have
\begin{equation}\label{constraint.7}
\left\{\begin{aligned}
&g^{00}g^{ij}\partial_ih^1_{0j}+g^{pl}g^{0i}\partial_ih^1_{pl} -\tfrac{1}{2}g^{00}g^{ij}\partial_i\partial_jh^0_{00} 
+ \tfrac{1}{2}g^{pl}g^{ij}\partial_i\partial_jh^0_{pl}
+ g^{00}E'_{00}- \tfrac{1}{2}g^{pl}F'_{pl}=0,
\\
&g^{0\beta}(- 2g^{0i}\partial_ih^1_{k\beta}-g^{ij}\partial_i\partial_j h^0 _{k\beta})
+g^{00}(g^{i\beta}\partial_ih^1_{k\beta}-\tfrac{1}{2} g^{\alpha\beta}\partial_kh^1_{\alpha\beta})
+g^{0\beta}F'_{k\beta}-g^{00} D'_{0k}=0,
\end{aligned}
\right.
\end{equation}
for $k=1,...,n$. 

\begin{lemma}\label{lem.1}
The harmonic gauge conditions (\ref{constraint.6}) and (\ref{constraint.7}) 
are equivalent to the constraint conditions (\ref{constraint.2}) combined with (\ref{constraint.6}) in the fixed coordinates, i.e. (\ref{constraint.6}) + (\ref{constraint.7}) =(\ref{constraint.2}) + (\ref{constraint.6}).
\end{lemma}
\begin{proof}
We only need to show under the condition  (\ref{constraint.6}), (\ref{constraint.7}) is equivalent to (\ref{constraint.2}) .  First (\ref{constraint.6}) gives
\begin{equation}\label{eq.00}
\left\{
\begin{aligned}
&g^{00}h^1_{00} =g^{ij}h^1_{ij}-2g^{i\beta}\partial_ih^0_{0\beta},
\\
&g^{00}h^1_{0k}=-g^{0l}h^1_{kl}-g^{i\beta}\partial_ih^0_{k\beta} +\tfrac{1}{2}g^{\alpha\beta}\partial_kh^0_{\alpha\beta},
\end{aligned}
\right.
\end{equation}
for $k=1,...,n$. 
Plugging them into the first equation of (\ref{constraint.7}), we have
\begin{equation*}
\begin{aligned}
0=\ & g^{ij}\partial_i(-g^{0l}h^1_{jl}-g^{l\beta}\partial_lh^0_{j\beta} +\tfrac{1}{2}g^{\alpha\beta}\partial_jh^0_{\alpha\beta})-g^{ij}\partial_ig^{00}h^1_{0j}+g^{pl}g^{0i}\partial_ih^1_{pl} -\tfrac{1}{2}g^{00}g^{ij}\partial_i\partial_jh^0_{00} 
\\ \ &
+ \tfrac{1}{2}g^{pl}g^{ij}\partial_i\partial_jh^0_{pl}
+ g^{00}E'_{00}- \tfrac{1}{2}g^{pl}F'_{pl}
\\
=\ &g^{0i}g^{pl}(\partial_ih^1_{pl}-\partial_ph^1_{il})+ g^{ij}(-g^{l\beta}\partial_i\partial_l h^0_{j\beta}+g^{pl}\partial_i\partial_jh^0_{pl}+g^{0l}\partial_i\partial_jh^0_{0l}) + Q_0
\end{aligned}
\end{equation*}
where $Q_0$ is a quadratic form of $(\partial_ih^0, h^1)$, 
\begin{equation*}
\begin{aligned}
Q_0 =\ & g^{ij}(-\partial_ig^{0l}h^1_{jl}-\partial_ig^{l\beta}\partial_lh^0_{j\beta}+\tfrac{1}{2}\partial_ig^{\alpha\beta}\partial_jh^0_{\alpha\beta})-g^{ij}\partial_ig^{00}h^1_{0j}+ g^{00}E'_{00}- \tfrac{1}{2}g^{pl}F'_{pl}
\\=\ &
g^{ij}(-\partial_ig^{0l}h^1_{jl}-\partial_ig^{l\beta}\partial_lh^0_{j\beta}+\tfrac{1}{2}\partial_ig^{\alpha\beta}\partial_jh^0_{\alpha\beta}-\partial_ig^{00}h^1_{0j}-D'_{ij}) + g^{00}E'_{00}-g^{pl}E'_{pl}
\\=\ &
g^{00}E'_{00}-g^{pl}E'_{pl},
\end{aligned}
\end{equation*}
which gives the first equation of (\ref{constraint.2}). 
Rewrite the second equation of  (\ref{constraint.7}) as follows:
\begin{equation*}
\begin{aligned}
&-\tfrac{1}{2}g^{00}g^{00}\partial_kh^{1}_{00}-g^{00}g^{0i}(\partial_ih^1_{k0}+\partial_kh^1_{0i})
+
g^{00}g^{ij}(-\partial_i\partial_jh^0_{k0}+\partial_ih^1_{kj} -\tfrac{1}{2}\partial_kh^1_{ij})
\\&\quad\quad\quad\quad
-2g^{0i}g^{0l}\partial_ih^1_{kl} -g^{0l}g^{ij}\partial_i\partial_jh^0_{kl} 
+g^{0\beta}F'_{k\beta}-g^{00} E''_{0k}=0.
\end{aligned}
\end{equation*}
Plugging (\ref{eq.00}) into 
the first three terms in above equation, we get
$$
\begin{aligned}
0=\ &-g^{00}\partial_k(\tfrac{1}{2}g^{ij}h^1_{ij}-g^{i\beta}\partial_ih^0_{0\beta}) 
+\tfrac{1}{2}g^{00}\partial_kg^{00}h^1_{00} +g^{0i}\partial_i(g^{0l}h^1_{kl}
+g^{l\beta}\partial_lh^0_{k\beta} -\tfrac{1}{2}g^{\alpha\beta}\partial_kh^0_{\alpha\beta})
\\ \ &
+g^{0i}\partial_ig^{00}h^1_{k0}+g^{0i}\partial_k(g^{0l}h^1_{il}+g^{l\beta}\partial_lh^0_{i\beta} 
-\tfrac{1}{2}g^{\alpha\beta}\partial_i h^0_{\alpha\beta}) 
+g^{0i}\partial_kg^{00}h^1_{0i}
\\ \ & 
+g^{00}g^{ij}(-\partial_i\partial_jh^0_{k0}+\partial_ih^1_{kj} -\tfrac{1}{2}\partial_kh^1_{ij})
-2g^{0i}g^{0l}\partial_ih^1_{kl} -g^{0l}g^{ij}\partial_i\partial_jh^0_{kl} 
+g^{0\beta}F'_{k\beta}-g^{00} E''_{0k}
\\
=\ & g^{00}g^{ij}(\partial_ih^1_{kj}-\partial_kh^{1}_{ij}
-\partial_i\partial_jh^0_{k0}+\partial_k\partial_ih^0_{0j})
+g^{0l}g^{0i}(\partial_kh^1_{il}-\partial_ih^{1}_{kl}
+\partial_i\partial_lh^0_{k0}-\partial_k\partial_lh^0_{0i})
\\ \ &
+g^{0l}g^{ij}(-\partial_i\partial_jh^0_{kl}+ \partial_i\partial_kh^0_{jl}+\partial_i\partial_lh^{0}_{kj}
-\partial_k\partial_lh^0_{ij})+
Q_k 
\end{aligned}
$$
where $Q_k$ is a quadratic form of $(\partial_ih^0, h^1)$, 
$$
\begin{aligned}
Q_k =\ &
g^{00}(-\tfrac{1}{2}\partial_kg^{ij}h^1_{ij}+\partial_kg^{i\beta}\partial_ih^0_{0\beta}+ \tfrac{1}{2}\partial_k g^{00} h^1_{00} )
+ g^{0i}(\partial_ig^{0l}h^1_{kl}+\partial_ig^{l\beta}\partial_lh^0_{k\beta}-\tfrac{1}{2}\partial_ig^{\alpha\beta} \partial_kh^0_{\alpha\beta}
\\ \ & 
+\partial_ig^{00} h^1_{k0} +\partial_kg^{0l}h^1_{il} +\partial_kg^{l\beta}\partial_lh^0_{i\beta}-\tfrac{1}{2}\partial_kg^{\alpha\beta}\partial_ih^0_{\alpha\beta}+\partial_kg^{00} h^1_{0i}) +g^{0\beta}F'_{k\beta}-g^{00} E''_{0k}
\\
=\ & -g^{00}E''_{k0} -g^{0i}(E''_{ik}+E''_{ki})+g^{00}F'_{k0}+g^{0i}F'_{ki}-g^{00} E''_{0k}
\\
=\ & 2(g^{00}E'_{k0}+g^{0i}E'_{ki})
\end{aligned}
$$
which gives the second equation in (\ref{constraint.2}). Since the substituting action is reversible, we finish proving the equivalence of (\ref{constraint.6})+(\ref{constraint.7}) and (\ref{constraint.6})+(\ref{constraint.2}) . 
\end{proof}

%
%

Notice that  (\ref{constraint.6}) is an elliptic system of equations of $h^1_{0\mu}$ for $\mu=0,1,...,n$,  which are free terms when solving the linearization of (\ref{constraint.2}). Hence by a similar decomposition as in the proof of Proposition \ref{prop.2},  the linearization of (\ref{constraint.6})+(\ref{constraint.7}) form an elliptic system 
 w.r.t. $\{A_1,C_1, B'', h^1_{0\mu}: \mu=0,...,n\}$ when considering the solutions in space (\ref{initialspace}) for $n\geq 3$ and $\lambda\in (0,n-2)$. 
Let
\begin{equation}\label{initial.2}
\widetilde{\mathcal{V}}_{\epsilon}^{N, \lambda}
=\{ (h_0,h_1):\ (h_0,h_1)\ \textrm{satisfies(\ref{constraint.6})+(\ref{constraint.7}), } \|(h^0,h^1)\|_{\rho_0^{\lambda}H_b^{N+1}(\overline{\mathbb{R}^n}) \times \rho_0^{\lambda+1}H_b^{N}(\overline{\mathbb{R}^n})}<\epsilon \}.
\end{equation}
Then by the same proof of Proposition \ref{prop.2}, we have 
\begin{proposition}\label{prop.3}
For $n\geq 3$, $\lambda\in (0,n-2)$ and $\epsilon>0$ small enough, $\widetilde{\mathcal{V}}_{\epsilon}^{N,\lambda}$ is a Fr\'{e}chet manifolds.
\end{proposition}
Here $\widetilde{\mathcal{V}}_{\epsilon}^{N,\lambda}$ has tangent space 
$\mathcal{V}^{N,\lambda}$ at $(0,0)$ which is the space of solutions to linearization of (\ref{constraint.6})+(\ref{constraint.7}) and is isomorphic to
\begin{equation*}
\rho_0^{\lambda}H_b^{N+1}(\overline{\mathbb{R}^n}: \mathbb{R}^{\frac{n(n+1)}{2}}) \times \mathcal{E}^{N+1,\lambda}_{c}\times \rho_0^{\lambda+1}H_b^{N}(\overline{\mathbb{R}^n}: \mathbb{R}^{\frac{(n+1)n}{2}-1}).
\end{equation*}

\subsection{Gauge Fixing.} 
%


Suppose $(\mathbb{R}^{1+n}_{\bar{t},\bar{x}}, \bar{g})$ is an Einstein vacuum solution which is close to the Minkowski space-time with Cauchy data $(\bar{h}^0,\bar{h}^{1})$ satisfying the constraint equations (\ref{constraint.2}). Here
\begin{equation*}
\bar{h}^0_{\mu\nu}=\bar{g}_{\mu\nu}-m_{\mu\nu}|_{\bar{t}=0}, \quad \bar{h}^1_{\mu\nu}=\partial_{\bar{t}}\bar{g}_{\mu\nu}|_{\bar{t}=0}
\end{equation*}
in the sense of matrices. 
Let $\Phi: \mathbb{R}^{1+n}_{t,x}\longrightarrow \mathbb{R}^{1+n}_{\bar{t},\bar{x}}$ be a diffeomorphism preserving the Cauchy surface: 
\begin{equation*}
\bar{x}^{\mu}=x^{\mu}+f^{\mu}(t,x) \quad\mathrm{with}\quad  f^{0}(0,x)=0.
\end{equation*}
%
We want to find out all $\Phi$ such that $(t,x)$ are wave coordinates w.r.t. the pull back metric  $g=\Phi^{*}\bar{g}$:
\begin{equation*}
\begin{aligned}
g_{\mu\nu}=\bar{g}_{\mu\nu}+\bar{g}_{\mu\beta}\partial_{\nu}f^{\beta}+ \bar{g}_{\alpha\nu}\partial_{\mu}f^{\alpha}+ \bar{g}_{\alpha\beta}\partial_{\mu}f^{\alpha}\partial_{\nu}f^{\beta},
\end{aligned}
\end{equation*}
Denote by $(h^0,h^1)$ the corresponding Cauchy data in coordinates (t,x) :
\begin{equation}\label{eq.10}
\left\{
 \begin{aligned}
  h^0_{\mu\nu} =&\ \bar{h}^0_{\mu\nu}+(\bar{g}_{\mu\beta}\partial_{\nu}f^{\beta}+ \bar{g}_{\alpha\nu}\partial_{\mu}f^{\alpha}+ \bar{g}_{\alpha\beta}\partial_{\mu}f^{\alpha}\partial_{\nu}f^{\beta})|_{t=0},
\\
  h^1_{\mu\nu} =&\ \bar{h}^1_{\mu\nu} + (\bar{h}^1_{\mu\beta}\partial_{\nu}f^{\beta}+ \bar{h}^1_{\alpha\nu}\partial_{\mu}f^{\alpha}+\bar{h}^1_{\alpha\beta}\partial_{\mu}f^{\alpha}\partial_{\nu}f^{\beta})|_{t=0}
\\
  &\ +[\partial_tf^{\gamma} (\partial_{\bar{x}^{\gamma}}\bar{g}_{\mu\nu} +\partial_{\bar{x}^{\gamma}}\bar{g}_{\mu\beta}\partial_{\nu}f^{\beta}+ \partial_{\bar{x}^{\gamma}}\bar{g}_{\alpha\nu}\partial_{\mu}f^{\alpha}+ \partial_{\bar{x}^{\gamma}}\bar{g}_{\alpha\beta}\partial_{\mu}f^{\alpha}\partial_{\nu}f^{\beta})]|_{t=0}
\\
  &\ +[\bar{g}_{\mu\beta}\partial_{\nu}\partial_tf^{\beta}+ \bar{g}_{\alpha\nu}\partial_{\mu}\partial_tf^{\alpha}+ \bar{g}_{\alpha\beta}(\partial_t\partial_{\mu}f^{\alpha}\partial_{\nu}f^{\beta}+ \partial_{\mu}f^{\alpha}\partial_t\partial_{\nu}f^{\beta})]_{t=0}.
 \end{aligned}\right.
\end{equation}
For simplicity, denote $\partial_{\bar{x}^{\alpha}}=\bar{\partial}_{\alpha}$. It shows that $(h^0, h^1)$ is determined by $(f,\partial_t f,\partial_t^2f)|_{t=0}$.

To make $g$ satisfy the harmonic gauge condition (\ref{harmonic.3}), it is equivalent to require $f$ satisfying the following equations:
\begin{equation}\label{eq.8}
\begin{gathered}
(\bar{g}_{\mu\nu}+\bar{g}_{\delta\nu}\partial_{\mu}f^{\delta}) g^{\alpha\beta}\partial_{\alpha}\partial_{\beta}f^{\nu}+ G_{\mu}(\bar{g},\bar{\partial} \bar{g};f,\partial f)=0
\end{gathered}
\end{equation}
for $\mu=0,...,n$, where $G_{\mu}$ are analytic functions of $(\bar{g},\bar{\partial}\bar{g};f,\partial f)$ for $(f,\partial f)$ small: 

\begin{equation*}
\begin{aligned}
G_{\mu}(\bar{g},\bar{\partial} \bar{g};f,\partial f)=
&\  g^{\alpha\beta}(\partial_{\alpha}\bar{g}_{\mu\beta}+\partial_{\alpha}\bar{g}_{\mu\gamma}\partial_{\beta}f^{\gamma}+ \partial_{\alpha}\bar{g}_{\delta\beta}\partial_{\mu}f^{\delta}+\bar{g}_{\delta\gamma}\partial_{\mu}f^{\delta}\partial_{\beta}f^{\gamma}) 
\\ &\ 
-\tfrac{1}{2}g^{\alpha\beta}(\partial_{\mu}\bar{g}_{\alpha\beta}+\partial_{\mu}\bar{g}_{\alpha\gamma}\partial_{\beta}f^{\gamma}+\partial_{\mu}\bar{g}_{\delta\beta}\partial_{\alpha}f^{\delta}+
\partial_{\mu}\bar{g}_{\delta\gamma}\partial_{\alpha}f^{\delta}\partial_{\beta}f^{\gamma}).
\end{aligned}
\end{equation*}
Notice that $\partial_{\alpha}\bar{g}_{\mu\nu}=\bar{\partial}_{\alpha}\bar{g}_{\mu\nu}+\partial_{\alpha}f^{\beta}\bar{\partial}_{\beta}\bar{g}_{\mu\nu}$. In particular, 
$$
G_{\mu}(\bar{g},\bar{\partial} \bar{g};0,0)=\ \bar{g}^{\alpha\beta}(\bar{\partial}_{\alpha}\bar{g}_{\mu\beta} -\tfrac{1}{2}\bar{\partial}_{\mu}\bar{g}_{\alpha\beta}).
$$
Here (\ref{eq.8}) is a system of quasilinear wave equations for $f$.  Given any $f_0=f|_{t=0}$ and $f_1=\partial_tf|_{t=0}$ with enough regularity, we can solve (\ref{eq.8}) in a neighborhood of Cauchy surface. Furthermore, if $(f_0,f_1)$ are also small enough, $f$ is determined globally. Since we assume $f_0^0=0$, $(f_0,f_1)$ have $2n+1$ degree of freedom. 

However, we want to choose the gauge before we find out the global Einstein solution $\bar{g}$ for the Cauchy problem.  By  Theorem \ref{thm.lr}, this can be done by requiring $(h^0,h^1)$ in the new coordinates $(t,x)$ to satisfy the harmonic gauge conditions (\ref{harmonic.1}), or equivalently, (\ref{constraint.6})+(\ref{constraint.7}). First  (\ref{constraint.6}) give $(n+1)$ linear equations of $\partial_t^2f|_{t=0}$ and allow us to write $ \partial_t^2f|_{t=0}$ 
in terms of $(f_0,f_1)$. More explicitly 
\begin{equation}\label{eq.50}
\partial^2_{t}f^{\mu}|_{t=0}= f^{\mu}_2(\bar{h}^0,\bar{h}^1;f_0,f_1)
\end{equation}
where $f_2^{\mu}$ are analytic functions of $(\bar{h}^0,\bar{h}^1;f_0,f_1)$. 
%
%
%
Then (\ref{constraint.7}) is satisfied automatically by Lemma \ref{lem.1} since the constraint equations (\ref{constraint.1}), or equivalently (\ref{constraint.2}), is invariant under deffeomorphism. This shows again that the global diffeomorphism and the global harmonic coordinates set are determined by $(f_0,f_1)$.

\subsection{Group Action and Fibration.} 
From (\ref{eq.10}),  the set of $(f_0,f_1)=(f|_{t=0},\partial_tf|_{t=0})$ gives a group action  on the space of  $(h^0,h^1)$ which satisfies the harmonic gauge conditions (\ref{constraint.6})+(\ref{constraint.7}). We investigate this group action in more details in the following. 

For $\lambda>0$, let $\mathcal{G}^{N,\lambda}$ be the nonabelian group generated by
$$
\{(f_0,f_1)\in \rho_0^{\lambda-1}H_b^{N+2}(\overline{\mathbb{R}^n}:\mathbb{R}^{n+1})\times \rho_0^{\lambda}H_b^{N+1}(\overline{\mathbb{R}^n}:\mathbb{R}^{n+1}):f^0_0=0,\ \|(f_0,f_1)\|_{}<\epsilon'\},
$$
for some $\epsilon'>0$ small with group operation
$$
(\tilde{f}_0,\tilde{f}_1)\star(f_0,f_1)=\big(\tilde{f}_0+f_0\circ (Id+\tilde{f}_0), \tilde{f}_1+(1+\tilde{f}_1^0)f_1\circ (Id+\tilde{f}_0)+\tilde{f}_1^i \partial_if_0 \circ (Id+\tilde{f}_0) \big).
$$
and identity $(0,0)$. Here $\circ$ denotes the composition of functions. It is clearly that $\mathcal{G}^{N,\lambda}$ does not depend on $\epsilon'>0$ small. Then $\mathcal{G}^{N,\lambda}$ has two subgroups generated by $(f_0,0)$ and $(0,f_1)$ respectively, which transversely intersect at $(0,0)$. 
Denote
\begin{equation*}
\begin{aligned}
&\mathcal{G}_0^{N,\lambda}= \{f_0: (f_0,0)\in\mathcal{G}^{N,\lambda}\},\quad
\mathcal{G}_1^{N,\lambda}= \{f_1: (0,f_1)\in\mathcal{G}^{N,\lambda}\}.
\end{aligned}
\end{equation*}
\begin{lemma}\label{lem.13}
With above notation,
$$
\mathcal{G}^{N,\lambda} = \mathcal{G}_1^{N,\lambda}\star\mathcal{G}_0^{N,\lambda}=\mathcal{G}_0^{N,\lambda}\star\mathcal{G}_1^{N,\lambda}.
$$
\end{lemma}

Denote by $\widetilde{\mathcal{V}}^{N,\lambda} = \mathcal{G}^{N,\lambda}\cdot \widetilde{\mathcal{V}}_{\epsilon}^{N,\lambda}\subset \widetilde{\mathcal{V}}^{N,\lambda}_{\infty}$ where $\widetilde{\mathcal{V}}_{\epsilon}^{N,\lambda} $ is given in Proposition \ref{prop.3}. By Lemma \ref{lem.13}, we can study the action of $\mathcal{G}_1^{N,\lambda}$ and $\mathcal{G}_0^{N,\lambda}$ on $\widetilde{\mathcal{V}}_{\epsilon}^{N,\lambda}$ separately.   First,  $\mathcal{G}_1^{N,\lambda}$ forms an nonabelian group with operation
\begin{equation}
(0,\tilde{f}_1)\star (0,f_1) = (0,f_1+\tilde{f}_1+ \tilde{f}^0_1f_1),\quad \forall\ f_1,\tilde{f}_1\in \mathcal{G}_1^{N,\lambda}.
\end{equation}
By (\ref{eq.10}), given $f_1\in \mathcal{G}_1^{N,\lambda}$ and $(\bar{h}^0,\bar{h}^1)\in \widetilde{\mathcal{V}}^{N,\lambda}$, then $(h^0,h^1)=(0,f_1)\cdot(\bar{h}^0,\bar{h}^1)$ is defined as follow: 
\begin{equation}
\left\{
\begin{aligned}
h^{0}_{00}=&\ \bar{h}^0_{00}+2\bar{g}_{0\alpha}f_1^{\alpha} 
+\bar{g}_{\alpha\beta}f_1^{\alpha}f_1^{\beta},
\\
h^0_{0i}=&\ \bar{h}^0_{0i}+\bar{g}_{i\beta}f_1^{\beta},
\\
h^0_{ij}=&\ \bar{h}^0_{ij}
\\
h^{1}_{00}=&\ (\bar{h}^1_{00}+\bar{h}^1_{00}f_1^0+\partial_i\bar{h}^0_{00}f_1^i)+ 2(\bar{h}^1_{0\beta}+\bar{h}^1_{0\beta}f_1^0+\partial_i\bar{h}^0_{0\beta}f_1^i)f_1^{\beta} 
\\
&\ +(\bar{h}^1_{\alpha\beta}+\bar{h}^1_{\alpha\beta}f_1^0+\partial_i\bar{h}^0_{\alpha\beta}f_1^i) f_1^{\alpha}f_1^{\beta}
 +2\bar{g}_{0\beta}f^{\beta}_2+ 2\bar{g}_{\alpha\beta}f_1^{\alpha}f_2^{\beta},
\\
h^1_{0i}=&\ (\bar{h}^1_{0i}+\bar{h}^1_{0i}f_1^0+\partial_k\bar{h}^0_{0i}f_1^k)+ (\bar{h}^1_{i\beta}+\bar{h}^1_{i\beta}f_1^0+\partial_k\bar{h}^0_{i\beta}f_1^k)f_1^{\beta}
\\
&\  +\bar{g}_{i\beta}f_2^{\beta}+\bar{g}_{0\beta}\partial_if_1^{\beta}+ \bar{g}_{\alpha\beta} f_1^{\alpha}\partial_if_1^{\beta},
\\
h^1_{ij}=&\ \bar{h}^1_{ij}+ \bar{h}^1_{ij}f_1^0+\partial_k\bar{h}^0_{ij}f_1^k+ \bar{g}_{i\beta}\partial_jf_1^{\beta}+\bar{g}_{\alpha j}\partial_if_1^{\alpha}.
\end{aligned}\right.
\end{equation}
Here $f_2=f_2(\bar{h}^0,\bar{h}^1;0,f_1)$ is defined in (\ref{eq.50}).
%
%
Obviously $\mathcal{G}_1^{N,\lambda}$ preserves the coordinates on the Cauchy surface
and hence preserves the restricting metric and second fundamental form on Cauchy surface. Given initial data $(h_0,h_1)$, let $g_0$ be the induced metric on Cauchy surface and $k_0$ the corresponding second fundamental form. This defines the map $\pi: (h^0,h^1)\longrightarrow (g_0,k_0)$:
$$
\begin{aligned}
 \quad [g_0]_{ij}= \delta_{ij}+h^0_{ij},
\quad
[k_0]_{ij}=-(1-h^0_{00}+g_0^{kl}h^0_{0l}h^0_{0k})^{\frac{1}{2}}
(\partial_ih^0_{0j}+\partial_jh^0_{0i}-h^1_{ij}).
\end{aligned}
$$
In our setting, 
$$
[g_0]_{ij}-\delta_{ij}\in \rho_0^{\lambda}H_b^{N+1}(\overline{\mathbb{R}^n}),\quad  [k_0]_{ij}\in \rho_0^{\lambda+1}H_b^{N}(\overline{\mathbb{R}^n}).
$$
Let 
$\widetilde{C}^{N,\lambda}_{\epsilon}=\pi(\mathcal{G}_1^{N,\lambda}\cdot \widetilde{\mathcal{V}}_{\epsilon}^{N,\lambda})$. Obviously, 
$$
\xymatrix{\mathcal{G}_1^{N,\lambda} \ar@{-}[r]& \mathcal{G}_1^{N,\lambda}\cdot \widetilde{\mathcal{V}}_{\epsilon}^{N,\lambda} \ar[d]^{\pi} \\ & \widetilde{C}^{N,\lambda}_{\epsilon}}
$$
provides a fibration structure. 

Secondly, $\mathcal{G}_0^{N,\lambda}$ is obvious equivalent to the diffeomorphism group on Cauchy surface with the corresponding regularity and decay at infinity. The group operation is as follows: 
$$
\tilde{f}_0 \star f_0 =\tilde{f}_0+f_0\circ(Id+ \tilde{f}_0),\quad \forall \ \tilde{f}_0,f_0\in \mathcal{G}_0^{N,\lambda}. 
$$
The action of $\mathcal{G}_0^{N,\lambda}$ on $\widetilde{C}^{N,\lambda}_{\epsilon}$ is just a pull-back. Denote by $\widetilde{C}^{N,\lambda}=\mathcal{G}_0^{N,\lambda}\cdot \widetilde{C}^{N,\lambda}_{\epsilon}$. Then for any $f_0\in \mathcal{G}_0^{N,\lambda}$, we have the commutativity
$$
\pi\circ f_0 =f_0 \circ \pi
$$
which implies that
$$
\xymatrix{\mathcal{G}_1^{N,\lambda} \ar@{-}[r]& \widetilde{\mathcal{V}}^{N,\lambda} \ar[d]^{\pi} \\ & \widetilde{C}^{N,\lambda}}
$$
is also a fibration. 
Let $\widetilde{\mathcal{D}}^{N,\lambda}=\widetilde{C}^{N,\lambda}/\mathcal{G}_0^{N,\lambda}$ and $\pi'$ the quotient map. Denote by $m_e$ be the Euclidean metric and $S^2(T^*\mathbb{R}^n)$ be the symmetric 2-tensor bundle on $\mathbb{R}^n$. Summarizing above discussion, we have
\begin{proposition}
For $n\geq 3$ and $\lambda\in(0,n-2)$, 
the quotient space $\widetilde{\mathcal{D}}^{N,\lambda}$ is an open neighborhood of $(m_e,0)$ in the solution space of constraint equation (\ref{constraint.1}) in
$$
\big(m_e+\rho_0^{\lambda}H_b^{N+1}(\overline{\mathbb{R}^n}; S^2(T^*\mathbb{R}^n)) \big) \times \rho_0^{\lambda+1}H_b^{N}(\overline{\mathbb{R}^n}; S^2(T^*\mathbb{R}^n)) .
$$
with the fibration structure: 
$$
\xymatrix{
\mathcal{G}_1^{N,\lambda} \ar@{-}[r]& \widetilde{\mathcal{V}}^{N,\lambda} \ar[d]^{\pi} \\ 
\mathcal{G}_0^{N,\lambda}\ar@{-}[r] & \widetilde{C}^{N,\lambda} \ar[d]^{\pi'}\\
& \widetilde{D}^{N,\lambda} .
}
$$

\end{proposition}
%


\vspace{0.2in}
\section{Conformal Transformation}\label{sec.conf}

To study the asymptotic behavior of Einstein vacuum solutions which are close to Minkowski space-time, we will work on their conformal transformation on $X$, the compactification of $\mathbb{R}^{1+n}_{t,x}$ defined in Section \ref{sec.geosetting}, and perform energy estimates for a conformal transformation of the reduced Einstein vacuum equation (\ref{eq.2}) with Cauchy data satisfying the harmonic gauge conditions (\ref{constraint.6})+(\ref{constraint.7}). We clarify the conformal transformation of metrics and equations in this section by describing them in the local coordinates  specified in Section \ref{sec.geosetting}.

\subsection{Perturbed Lorentzian Metric}
Let $m=-dt^2+\sum_{x=1}^n(dx^i)^2$ be the Minkowski metric and $g=m+h$ a small perturbation of it:  $h=h_{00}d^2t+2h_{0i}dtdx^i+h_{ij}dx^{i}dx^{j}$. In the fixed coordinates $(t,x)=(x^0,x^1,...,x^n)$, it is equivalent to say that $h$ is a function on $\mathbb{R}^{1+n}_{t,x}$ valued in $(n+1)\times(n+1)$ symmetric matrix space $SM(n+1,\mathbb{R})\simeq \mathbb{R}^{\frac{(n+2)(n+1)}{2}}$. In this sense, denote the inverse of $g$ by
\begin{equation*}
g^{\mu\nu}=m^{\mu\nu}+H^{\mu\nu}.
\end{equation*}
\begin{definition}
For any $k\in \mathbb{N}_0$, denote by $\Theta_k(h)$ a real analytic function of $h$ with coefficients in $\mathscr{A}^{0,0,0}(X)$ such that 
\begin{equation*}
\Theta_k(h)=O(|h|^k) \quad\mathrm{as}\quad |h|\rightarrow 0.
\end{equation*}
For any $k\in \mathbb{N}_0,j,l\in \mathbb{N}$, denote by $\Theta_k(h)(v^1,...,v^l)$ an $l$-form in $(v^1,...,v^l)$ with coefficients $\Theta_k(h)$, and  $\Theta_k(h)(v^1,...,v^l)(u^1,...,u^j)$ an $j$-form in $(u^1,...,u^j)$ with coefficients $\Theta_k(h)(v^1,...,v^l)$. 
\end{definition}

\begin{lemma} 
For any $k,l\in \mathbb{N}_0$
\begin{equation*}
\begin{gathered}
\Theta_k(h) +\Theta_l(h) =\Theta_{\min\{k,l\}}(h), \quad \Theta_k(h)\Theta_l(h)= \Theta_{k+l}(h).
\end{gathered}
\end{equation*}
\end{lemma}
\begin{lemma}
With above notation 
\begin{equation*}
H^{00}=-h_{00}+\Theta_2(h), \quad H^{0i}=h_{0i}+\Theta_2(h), \quad H^{ij}=-h_{ij}+\Theta_2(h).
\end{equation*}
\end{lemma}

For the purpose of studying the geometry of $g$ near null infinity of $(\mathbb{R}^{1+n}_{t,x},g)$, we also denote by 
\begin{equation*}
\begin{gathered}
h_{0\rho}=h_{\rho 0}=h_{0\mu}\theta^{\mu},\quad
h_{\rho\rho}=h_{\mu\nu}\theta^{\mu}\theta^{\nu}, \quad h_{i\rho}=h_{\rho i}=h_{i\mu }\theta^{\mu}.
\end{gathered}
\end{equation*}
With this notations, the harmonic gauge condition says that some components decay faster than the others near null infinity. 
\begin{lemma}\label{lem.4}
Suppose $(t,x)$ are harmonic coordinates w.r.t. $g$. We define a vector field $D\in\mathscr{V}_b(\Omega_i)$ for each $1\leq i\leq 5$ as follows:
$$
D=\begin{cases}
0 &\text{in $\Omega_1$ }\\
b\partial b-a\partial a &\text{in $\Omega_2$}\\
\partial_{\tau} &\text{in $\Omega_3$}\\
\bar{a}\partial_{\bar{a}}-\bar{b}\partial_{\bar{b}} &\text{in $\Omega_4$}\\
0 &\text{in $\Omega_5$}
\end{cases}
$$
Then in each domain $\Omega_i$ for $1\leq i\leq 5$, with the choice of $\tilde{\rho}$ and $\rho_1$ specified in Section \ref{sec.geosetting}, $D$ satisfies
$$
[D,\tilde{\rho}]=[D,\theta]=0,
$$
and the harmonic gauge condition implies
$$
D(2h_{\mu\rho}-\theta_{\mu}\mathrm{tr}_mh)=\rho_1 \tilde{\partial}h+\Theta_1(h)(\tilde{\partial}h) 
$$
for $\mu=0,1,...,n$. Here $\tilde{\partial} \in \mathscr{V}_b(X)$.
\end{lemma}
\begin{proof}
We only need to prove the statement in $\Omega_2,\Omega_3,\Omega_4$. By writing out the harmonic gauge condition in local coordinates, we have
\begin{itemize}
\item In $\Omega_2$, $a=1-\frac{t}{r}$, $b=\frac{1}{r-t}$, $\tilde{\rho}=ab$, $\rho_1=a$,
\begin{equation*}
\begin{gathered}
(b\partial_{b}-a\partial_a) (2h_{0\rho}+\mathrm{tr}_{m}h) =a\tilde{\partial}h +\Theta_1(h)(\tilde{\partial}h),
\\
(b\partial_{b}-a\partial_a) (2h_{i\rho}-\theta_i\mathrm{tr}_{m}h) =a\tilde{\partial}h +\Theta_1(h)(\tilde{\partial}h),
\\
\Longrightarrow\quad
Dh_{\rho\rho}=a\tilde{\partial}h +\Theta_1(h)(\tilde{\partial}h);
\end{gathered}
\end{equation*}
\item In $\Omega_3$, $\rho=\frac{1}{r}$, $\tau=t-r$, $\tilde{\rho}=\rho_1=\rho$,
\begin{equation*}
\begin{gathered}
\partial_{\tau}(2h_{0\rho}+\mathrm{tr}_{m}h) =\rho\tilde{\partial}h+\Theta_1(h)(\tilde{\partial}h),
\\
\partial_{\tau}(2h_{i\rho}-\theta_i\mathrm{tr}_{m}h) =\rho\tilde{\partial}h +\Theta_1(h)(\tilde{\partial}h),
\\
\Longrightarrow \quad
Dh_{\rho\rho}=\rho\tilde{\partial}h+\Theta_1(h)(\tilde{\partial}h);
\end{gathered}
\end{equation*}
\item In $\Omega_4$, $\bar{a}=1-\frac{r}{t}$, $\bar{b}=\frac{1}{t-r}$, $\tilde{\rho}=\bar{a}\bar{b}$, $\rho_1=\bar{a}$,
\begin{equation*}
\begin{gathered}
(\bar{a}\partial_{\bar{a}}-\bar{b}\partial_{\bar{b}}) (2h_{0\rho}+\mathrm{tr}_{m}h)=\bar{a}\tilde{\partial}h +\Theta_1(h)(\tilde{\partial}h),
\\
(\bar{a}\partial_{\bar{a}}-\bar{b}\partial_{\bar{b}}) (2h_{i\rho }-\theta_i\mathrm{tr}_{m}h)=\ \bar{a}\tilde{\partial}h +\Theta_1(h)(\tilde{\partial}h),
\\
\Longrightarrow\quad
Dh_{\rho\rho}= \bar{a}\tilde{\partial}h +\Theta_1(h)(\tilde{\partial}h).
\end{gathered}
\end{equation*}
\end{itemize} 
Here $\mathrm{tr}_{m}h=m^{\alpha\beta}h_{\alpha\beta}=-h_{00}+\sum_{i=1}^nh_{ii}$.  We finish the proof.
\end{proof}

\subsection{Conformal Transformation of the Lorentzian Metric}
Recall that $\tilde{\rho}$ denotes a total boundary defining function of $X$ and $\rho_i$ denotes the defining function for boundary hypersurfaces $S_i$ for $i=0,1,2$. Set 
\begin{equation}\label{eq.11}
\begin{gathered}
\tilde{m}=\tilde{\rho}^2m,\quad \tilde{h}=\tilde{\rho}^{\frac{1-n}{2}}h,\quad
\tilde{g}=\tilde{\rho}^2g=\tilde{m}+\tilde{\rho}^2h=\tilde{m}+\tilde{\rho}^{\frac{n+3}{2}}\tilde{h}.
\end{gathered}
\end{equation}
Similarly, denote
\begin{equation*}
\begin{gathered}
\tilde{h}_{0\rho}=\tilde{h}_{\rho 0}=\tilde{h}_{0\mu}\theta^{\mu},\quad 
\tilde{h}_{\rho\rho}=\tilde{h}_{\mu\nu}\theta^{\mu}\theta^{\nu},\quad
\tilde{h}_{i\rho}=\tilde{h}_{\rho i}=\tilde{h}_{\mu i}\theta^{\mu}.
\end{gathered}
\end{equation*}
\begin{lemma}
 Under the assumption and notation in Lemma \ref{lem.4}, for $\mu=0,1,...,n$,
\begin{equation*} 
D(2\tilde{h}_{\mu\rho}-\theta_{\mu}\mathrm{tr}_m\tilde{h})=\rho_1\tilde{\partial}\tilde{h}+
\rho_1 \Theta_1(\tilde{h}) +\Theta_1(\tilde{\rho}^{\frac{n-1}{2}}\tilde{h})(\tilde{\partial}\tilde{h})+\Theta_1(\tilde{\rho}^{\frac{n-1}{2}}\tilde{h})(\tilde{h}).
\end{equation*}
\end{lemma}

\begin{proposition}\label{prop.4}
If $h\in C^1(X: SM(n+1,\mathbb{R}))$ such that for some $\delta>0$ and all $1\leq i\leq n$
\begin{equation*}
\begin{gathered}
\rho_1^{-1}h_{\rho i}\in C^1(X),\quad\rho_1^{-1-\delta}h_{\rho\rho}\in C^1(X),
\\
\sum_{\mu,\nu}\|h_{\mu\nu}\|_{\infty}+ \sum_{1\leq i\leq n}\|\rho_1^{-1}h_{\rho i}\|_{\infty}+\|\rho_1^{-1-\delta}h_{\rho\rho}\|_{\infty}<\epsilon,
\end{gathered}
\end{equation*}
for some $\epsilon>0$ small, then $\tilde{g}$ extends to a $C^0(X^2)$ Lorentzian b-metric of signature $(n,1)$ on $X^2$ with $S^{\pm}$ being its characteristic surfaces. 
\end{proposition}
\begin{proof}
The statement is independent of choice of boundary defining functions. Hence we only need to write the metric $\tilde{g}=\tilde{m}+\tilde{\rho}^2h$ in local coordinates near $\partial X$ with particular choice of $\tilde{\rho}$ and $\rho_i$ as stated in Section \ref{sec.geosetting}. 
\begin{itemize}
\item In $\Omega_1$, $s=\frac{t}{r}$, $\rho=\frac{1}{r}$, $\tilde{\rho}=\rho$,
\begin{equation*}
\begin{aligned}
\tilde{m}=&\ -d^2s+2sds\frac{d\rho}{\rho}+(1-s^2)(\frac{d\rho}{\rho})^2+d^2\theta,
\\
\tilde{\rho}^2h =&\  
h_{00}d^2s+(h_{00}s^2+2h_{0i}\theta^is +h_{ij}\theta^i\theta^j)(\frac{d\rho}{\rho})^2 -2(h_{00}s+h_{0i})ds\frac{d\rho}{\rho}\\
& +h_{ij}d\theta^id\theta^j +2h_{0i}dsd\theta^i -2(h_{0i}s+h_{ij}\theta^j)d\theta^j\frac{d\rho}{\rho}.
\end{aligned}
\end{equation*}
\item In $\Omega_2$, $a=1-\frac{t}{r}$, $b=\frac{1}{r-t}$, $\tilde{\rho}=ab$, $\rho_1=a$,
\begin{equation*}
\begin{aligned}
\tilde{m}=&\ 2da\frac{db}{b}+a(2-a)(\frac{db}{b})^2+d^2\theta,\quad 
\tilde{m}|_{\rho_1=0}= d^2\theta,
\\
\tilde{\rho}^2h =&\ h_{\rho\rho}(\frac{da}{a})^2 + (h_{\rho\rho}-2ah_{0\rho}+a^2h_{00}) (\frac{db}{b})^2 +2(h_{\rho\rho}-ah_{0\rho}) \frac{da}{a}\frac{db}{b} 
\\
&+h_{ij}d\theta^id\theta^j  
+2h_{\rho i}\frac{da}{a}d\theta^i+2(-h_{\rho i}+ah_{0 i})\frac{db}{b}d\theta^i.
\end{aligned}
\end{equation*}
\item In $\Omega_3$, $\rho=\frac{1}{r}$, $\tau=t-r$, $\tilde{\rho}=\rho_1=\rho$,
\begin{equation*}
\begin{aligned}
\tilde{m}=&\ 2d\tau d\rho-\rho^2 d^2\tau+d^2\theta,\quad
\tilde{m}|_{\rho_1=0}=d^2\theta,
\\
\tilde{\rho}^2h =&\  h_{\rho\rho}(\frac{d\rho}{\rho})^2+\rho^2h_{00}d^2\tau -2h_{0\rho}d\tau d\rho+ h_{ij}d\theta^id\theta^j+2\rho h_{0 i}d\theta^id\tau- 2h_{\rho i}d\theta^i \frac{d\rho}{\rho}.
\end{aligned}
\end{equation*}
\item In $\Omega_4$, $\bar{a}=1-\frac{r}{t}$, $\bar{b}=\frac{1}{t-r}$, $\tilde{\rho}=\bar{a}\bar{b}$, $\rho_1=\bar{a}$,
\begin{equation*}
\begin{aligned}
\tilde{m} =&\ -2d\bar{a}\frac{d\bar{b}}{\bar{b}} -\bar{a}(2-\bar{a})(\frac{d\bar{b}}{d\bar{b}})+(1-\bar{a})^2d^2\theta,
\quad \tilde{m}|_{\rho_1=0}= d^2\theta,
\\
\tilde{\rho}^2h =&\ h_{\rho\rho}(\frac{d\bar{a}}{\bar{a}})^2 + (h_{\rho\rho}-2\bar{a}h_{i\rho}\theta^i+\bar{a}^2h_{ij}\theta^i\theta^j) (\frac{d\bar{b}}{\bar{b}})^2 +2(h_{\rho\rho}-\bar{a}h_{\rho i}\theta^i)\frac{d\bar{a}}{\bar{a}} \frac{d\bar{b}}{\bar{b}}
\\
&+(1-\bar{a})^2h_{ij}d\theta^id\theta^j - 2(1-\bar{a})h_{\rho i} d\theta^i\frac{d\bar{a}}{\bar{a}} + 2(1-\bar{a})(-h_{\rho i}+\bar{a}h_{ij}\theta^j) d\theta^i\frac{d\bar{b}}{\bar{b}}.
\end{aligned}
\end{equation*}
\item In $\Omega_5$, $\phi=\frac{1}{t}$, $y=\frac{x}{t}$, $\tilde{\rho}=\phi$,
\begin{equation*}
\begin{aligned}
\tilde{m}=&\ -(1-|y|^2)(\frac{d\phi}{\phi})^2+\sum_{i=1}^nd^2y^i-2y^idy^i(\frac{d\phi}{\phi}),
\\
\tilde{\rho}^2h =&\  (h_{00}+2h_{0i}y^i+h_{ij}y^iy^j)(\frac{d\phi}{\phi})^2+h_{ij}dy^idy^j-2(h_{0i} +h_{ij}y^j)dy^i\frac{d\phi}{\phi}.
\end{aligned}
\end{equation*}
\end{itemize}
Here in $\Omega_i$ for $1\leq i\leq 4$ we use the polar coordinates. Notice here $\theta$ is restricted on $\mathbb{S}^{n-1}$, i.e. $|\theta|=1$ and $\theta_id\theta^i=0$. Then the statement is obviously true if $\delta\geq 1$. For $\delta\in (0,1)$, let
\begin{equation*}
f(\rho_1)=\tfrac{1}{2}\int_0^{\rho_1}{\rho'_1}^{-2}h_{\rho\rho}d\rho'_1\in \rho_1^{\delta}C^0(X).
\end{equation*}
and change the coordinates near $S_1$ as follows:
\begin{itemize}
\item In $\Omega_2$, set $b=b'e^{-f(a)}$;
\item In $\Omega_3$, set $\tau=\tau'-f(\rho)$; 
\item In $\Omega_4$, set $\bar{b}= \bar{b}'e^{f(\bar{a})}$.
\end{itemize}
In the new coordinates the metric components is uniformly bounded and $\{\rho_1=0\}$ is its characteristic surface. Since the coordinates changing preserves the boundary hypersurface $S^{\pm}_1$, we prove the statement. Notice that the coordinates change extends to a $C^{0,\delta}$ homeomorphism of $X\longrightarrow X$. 
\end{proof}


\subsection{Wave Operator and Commutators.}
Denote by $\Box_g$ and $\Box_{\tilde{g}}$ the Laplace-Beltrami operators w.r.t. $g$ and $\tilde{g}=\tilde{\rho}^2g$. The relation between $\Box_g$ and $\Box_{\tilde{g}}$ is stated as follows.
\begin{lemma}
For any $u\in C^2(\mathbb{R}^{1+n}_{t,x})$, 
\begin{equation}
\Box_g (\tilde{\rho}^{\frac{n-1}{2}}u)= \tilde{\rho}^{\frac{n+3}{2}}(\Box_{\tilde{g}}+\gamma)u,
\quad \gamma=-\tilde{\rho}^{\frac{n-1}{2}}\Box_{\tilde{g}}\tilde{\rho}^{\frac{1-n}{2}}.
\end{equation}
\end{lemma}
\begin{proof}
The formula can be proved by direct computation as Friedlander did in \cite{Fr}.
\begin{equation}
\begin{aligned}
\Box_g (\tilde{\rho}^{\frac{n-1}{2}}u) =&\  |g|^{-\frac{1}{2}}\partial_{\mu} (|g|^{\frac{1}{2}}g^{\mu\nu}\partial_{\nu}(\tilde{\rho}^{\frac{n-1}{2}}u))
\\
=&\ \tilde{\rho}^{n+1}|\tilde{g}|^{-\frac{1}{2}}\partial_{\mu} (\tilde{\rho}^{1-n}|\tilde{g}|^{\frac{1}{2}}g^{\mu\nu}\partial_{\nu}(\tilde{\rho}^{\frac{n-1}{2}}u))
\\
=&\ \tilde{\rho}^{n+1}|\tilde{g}|^{-\frac{1}{2}}\partial_{\mu} (|\tilde{g}|^{\frac{1}{2}}g^{\mu\nu}(\tilde{\rho}^{\frac{1-n}{2}}\partial_{\nu}u -u\partial_{\nu}\tilde{\rho}^{\frac{1-n}{2}}))
\\
=&\ \tilde{\rho}^{\frac{n+3}{2}}(\Box_{\tilde{g}}u+\gamma u)
\end{aligned}
\end{equation}
\end{proof}

\begin{lemma}\label{lem.2}
Suppose $(t,x)$ are harmonic coordinates w.r.t. $g$. Then in each domain $\Omega_i$ for $1\leq i\leq5$, with the choice of $\tilde{\rho}, \rho_0,\rho_1,\rho_2$ specified in Section \ref{sec.geosetting}, we have
\begin{equation*}
\begin{gathered}
\Box_{\tilde{g}} =  \Box_{\tilde{m}}
+(\rho_0\rho_2)^{\frac{n-1}{2}} [-\rho_1^{\frac{n-5}{2}}\tilde{h}_{\rho\rho}(D^2+(3-i)D) + \rho_1^{\frac{n-3}{2}}\Theta_1(\tilde{h})((\tilde{\partial}^2+\tilde{\partial}))],\quad
\\
\Box_{\tilde{m}}=\ 2\partial_{\rho_1}D+ \tilde{\partial}^2+\tilde{\partial}, 
\quad
\gamma= \gamma_0+\tilde{\rho}^{\frac{n-1}{2}}\Theta_1(\tilde{h}),
\end{gathered}
\end{equation*}
where $D\in\mathscr{V}_b(\Omega_i)$ is defined in Lemma \ref{lem.4} and $\tilde{\partial}\in\mathscr{V}_b(X), \tilde{\partial}^2\in\mathscr{V}^2_b(X)$. 
Here $\gamma_0=-\tilde{\rho}^{\frac{n-1}{2}}\Box_{\tilde{m}}\tilde{\rho}^{\frac{1-n}{2}}$ is a constant in each domain.
\end{lemma}
\begin{proof}
If $(t,x)$ are harmonic coordinates w.r.t. $g$, then $\Box_{g}=g^{\mu\nu}\partial_{\mu}\partial_{\nu}=\Box_m+H^{\mu\nu}\partial_{\mu}\partial_{\nu}$, which implies that
\begin{equation*}
\begin{gathered}
\Box_{\tilde{g}}+\gamma=\Box_{\tilde{m}}+\gamma_0+ \tilde{\rho}^{-\frac{n+3}{2}} H^{\mu\nu}\partial_{\mu}\partial_{\nu} \tilde{\rho}^{\frac{n-1}{2}}.
\end{gathered}
\end{equation*}
We write $\Box_{\tilde{m}}$,  $\tilde{\rho}^{-\frac{n+3}{2}} H^{\mu\nu}\partial_{\mu}\partial_{\nu} \tilde{\rho}^{\frac{n-1}{2}}$,  $\gamma_0$ and $\gamma$ in local coordinates in each domain $\Omega_i$ for $1\leq i\leq 5$.
\begin{itemize}
\item In $\Omega_1$, $s=\frac{t}{r}$, $\rho=\frac{1}{r}$, $\tilde{\rho}=\rho$,
\begin{equation*}
\begin{aligned}
\Box_{\tilde{m}} =&\ -(1-s^2)\partial_s^2 +2s\partial_s\rho\partial_{\rho}+2s\partial_s +(\rho\partial_{\rho})^2 +\rho\partial_{\rho} +\triangle_{\theta},
\\
\Box_{\tilde{g}} =&\ \Box_{\tilde{m}}+[H^{00}-2sH^{0i}\theta_i+s^2H^{ij}\theta_i\theta_j]\partial^2_s +[H^{ij}\theta_i\theta_j](\rho\partial_{\rho})^2 +H^{ij}\slashpar_i\slashpar_j 
\\
& +2[-H^{0i}\theta_i+sH^{ij}\theta_i\theta_j]\partial_s(\rho\partial_{\rho}) +2[H^{0i}-sH^{ij}\theta_j]\partial_s\slashpar_i 
\\
& +2[-H^{ij}\theta_j](\rho\partial_{\rho})\slashpar_i  +[-(n+1)H^{0i}\theta_i+sH^{ij}(-\delta_{ij}+(n+2)\theta_i\theta_j)]\partial_s 
\\
& +[H^{ij}(-\delta_{ij}+(n+1)\theta_i\theta_j)]\rho\partial_{\rho}
+[-nH^{ij}\theta_j]\slashpar_i
\\
=&\ \Box_{\tilde{m}}+\rho^{\frac{n-1}{2}}\Theta_1(\tilde{h}) (\tilde{\partial}^2+\tilde{\partial}),  
\\
\gamma =&\ \gamma_0+[H^{ij}(-\tfrac{n-1}{2}\delta_{ij}+\tfrac{(n-1)(n+3)}{4}\theta_i\theta_j)], \quad
\gamma_0=-\tfrac{(n-1)(n-3)}{4}.
\end{aligned}
\end{equation*}
\item In $\Omega_2$, $a=1-\frac{t}{r}$, $b=\frac{1}{r-t}$, $\tilde{\rho}=ab$,
\begin{equation*}
\begin{aligned}
\Box_{\widetilde{m}} =&\  2\partial_ab\partial_b -a(2-a)\partial_a^2 -2(1-a)\partial_a +\triangle_{\theta}
\\
=&\ 2\partial_a(b\partial_b-a\partial_a)+(a\partial_a)^2+a\partial_a+\triangle_{\theta}
\\
\Box_{\tilde{g}} =&\ \Box_{\tilde{m}}+
a^{-2}[H^{00}-2H^{i0}\theta_i+H^{ij}\theta_{i}\theta_{j}](b\partial_{b}-a\partial_{a})^2 +[H^{ij}\theta_i\theta_j](a\partial_{a})^2 
\\
& +2a^{-1}[-H^{0i}\theta_i+H^{ij}\theta_i\theta_j] (a\partial_{a})(b\partial_{b}-a\partial_{a})+H^{ij}\slashpar_i\slashpar_j 
\\
& +2a^{-1}[H^{0i}-H^{ij}\theta_j]\slashpar_i(b\partial_{b}-a\partial_{a})
+2[-H^{ij}\theta_j]\slashpar_i(a\partial_{a})
\\
& +a^{-2}[H^{00}-2H^{i0}\theta_i+H^{ij}\theta_{i}\theta_{j}
+aH^{ij}(-\delta_{ij}+n\theta_i\theta_j) -(n-1)aH^{0i}\theta_i](b\partial_{b}-a\partial_{a})
\\
&
+[H^{ij}(-\delta_{ij}+(n+1))\theta_i\theta_j]a\partial_{a} +[-nH^{ij}\theta_j]\slashpar_i,
\\
=&\  \Box_{\tilde{m}}-a^{\frac{n-5}{2}}b^{\frac{n-1}{2}}\tilde{h}_{\rho\rho} [(b\partial_{b}-a\partial_{a})^2+(b\partial_{b}-a\partial_{a})]+
 a^{\frac{n-3}{2}} b^{\frac{n-1}{2}} \Theta_1(\tilde{h})(\tilde{\partial}^2+\tilde{\partial}),  
\\
\gamma =&\ \gamma_0+ [H^{ij}(-\tfrac{n-1}{2}\delta_{ij}+\tfrac{(n-1)(n+3)}{4}\theta_i\theta_j)], \quad \gamma_0=-\tfrac{(n-1)(n-3)}{4}.
\end{aligned}
\end{equation*}
\item In $\Omega_3$, $\rho=\frac{1}{r}$, $\tau=t-r$, $\tilde{\rho}=\rho$, 
\begin{equation*}
\begin{aligned}
\Box_{\tilde{m}}=&\ 2\partial_{\rho}\partial_{\tau} +(\rho\partial_{\rho})^2 +\rho\partial_{\rho} +\triangle_{\theta},
\\
\Box_{\tilde{g}} =&\  \Box_{\tilde{m}}+\rho^{-2}[H^{00}-2H^{i0}\theta_i+H^{ij}\theta_{i}\theta_{j}]\partial^2_{\tau} +[H^{ij}\theta_i\theta_j](\rho\partial_{\rho})^2 +H^{ij}\slashpar_i\slashpar_j 
\\
& +2[-H^{ij}\theta_j](\rho\partial_{\rho})\slashpar_i +2\rho^{-1}[-H^{i0}\theta_i+H^{ij}\theta_i\theta_j]\partial_{\tau}(\rho\partial_{\rho})
\\
& +2\rho^{-1}[H^{i0}-H^{ij}\theta_{j}]\partial_{\tau}\slashpar_i
+[H^{ij}(-\delta_{ij}+(n+1)\theta_i\theta_j)]\rho\partial_{\rho}
\\
& +\rho^{-1}[-(n-1)H^{0i}\theta_i+H^{ij}(-\delta_{ij}+n\theta_i\theta_j)]\partial_{\tau}
+n[-H^{ij}\theta_j]\slashpar_i  
\\
=&\ \Box_{\tilde{m}}-\rho^{\frac{n-5}{2}}\tilde{h}_{\rho\rho}\partial^2_{\tau} +\rho^{\frac{n-3}{2}}\Theta_1(\tilde{h})(\tilde{\partial}^2+\tilde{\partial}),
\\
\gamma
=&\ \gamma_0+[H^{ij}(-\tfrac{n-1}{2}\delta_{ij}+\tfrac{(n-1)(n+3)}{4})\theta_i\theta_j], 
\quad \gamma_0=-\tfrac{(n-1)(n-3)}{4}.
\end{aligned}
\end{equation*}
\item In $\Omega_4$, $\bar{a}=1-\frac{1}{t}$, $\bar{b}=\frac{1}{t-r}$, $\tilde{\rho}=\bar{a}\bar{b}$,
\begin{equation*}
\begin{aligned}
\Box_{\tilde{m}}=&\ -2\partial_{\bar{a}} \bar{b}\partial_{\bar{b}} +\bar{a}(2-\bar{a}) \partial_{\bar{a}}^2 +2\partial_{\bar{a}}-(n+1)\bar{a}\partial_{\bar{a}} +\tfrac{1}{(1-\bar{a})^{2}}\triangle_{\theta}
-\tfrac{n-1}{1-\bar{a}}(\bar{a}\partial_{\bar{a}} -\bar{b}\partial_{\bar{b}})
\\
=&\ 2\partial_{\bar{a}}(\bar{a}\partial_{\bar{a}} -\bar{b}\partial_{\bar{b}}) -(\bar{a}\partial_{\bar{a}})^2
-n\bar{a}\partial_{\bar{a}} +\tfrac{1}{(1-\bar{a})^{2}}\triangle_{\theta}
-\tfrac{n-1}{1-\bar{a}}(\bar{a}\partial_{\bar{a}} -\bar{b}\partial_{\bar{b}})
\\
\Box_{\tilde{g}} =&\ \Box_{\tilde{m}}+ 
\bar{a}^{-2}[H^{00}-2H^{i0}\theta_i+H^{ij}\theta_i\theta_j] (\bar{b}\partial_{\bar{b}}-\bar{a}\partial_{\bar{a}})^2 +H^{00}(\bar{a}\partial_{\bar{a}})^2 
\\
&
+ (1-\bar{a})^{-2}H^{ij}\slashpar_i\slashpar_j
+2\bar{a}^{-1}[H^{00}-H^{0i}\theta_i](\bar{a}\partial_{\bar{a}}) (\bar{b}\partial_{\bar{b}}-\bar{a}\partial_{\bar{a}})
\\
& -2(1-\bar{a})^{-1}H^{0i}(\bar{a}\partial_{\bar{a}})\slashpar_i
+2(\bar{a}(1-\bar{a}))^{-1}[-H^{i0}+H^{ij}\theta_j](\bar{b}\partial_{\bar{b}} -\bar{a}\partial_{\bar{a}})\slashpar_i 
\\
&+(1-\bar{a})^{-2}[-H^{ij}\theta_j-(n-1)H^{0i}]\slashpar_i +nH^{00}\bar{a}\partial_{\bar{a}} 
+ \bar{a}^{-2}[H^{00}-2H^{i0}\theta_i+H^{ij}\theta_i\theta_j 
 \\
&+(n-1)\bar{a}(H^{00}-H^{0i}\theta_i)
-\bar{a}(1-\bar{a})^{-1}H^{ij}(-\delta_{ij} +\theta_i\theta_j)](\bar{b}\partial_{\bar{b}} -\bar{a}\partial_{\bar{a}})
\\
=&\ \Box_{\tilde{m}}-\bar{a}^{\frac{n-5}{2}}\bar{b}^{\frac{n-1}{2}}\tilde{h}_{\rho\rho} [(\bar{b}\partial_{\bar{b}} -\bar{a}\partial_{\bar{a}})^2+(\bar{b}\partial_{\bar{b}} -\bar{a}\partial_{\bar{a}})]+  \bar{a}^{\frac{n-3}{2}} \bar{b}^{\frac{n-1}{2}} \Theta_1(\tilde{h})(\tilde{\partial}^2+\tilde{\partial}),  
\\
\gamma=&\ \gamma_0+\tfrac{n^2-1}{4}H^{00},\quad \gamma_0=-\tfrac{n^2-1}{4}.
\end{aligned}
\end{equation*}
\item In $\Omega_5$, $\phi=\frac{1}{t}$, $y=\frac{x}{t}$, $\tilde{\rho}=\phi$,
\begin{equation*}
\begin{aligned}
\Box_{\tilde{m}}=&\ -(\phi\partial_{\phi}+\sum_iy^i\partial_{y^i})^2 -n(\phi\partial_{\phi} +\sum_iy^i\partial_{y^i}) +\triangle_{y},
\\
\Box_{\tilde{g}} =&\ \Box_{\tilde{m}}+
H^{00}(\phi\partial_{\phi}+y^k\partial_{y^k})^2 +H^{ij}\partial_{y^i}\partial_{y^j} -2H^{0i}(\phi\partial_{\phi}+y^k\partial_{y^k})\partial_{y^i}
\\
& +nH^{00}(\phi\partial_{\phi}+y^i\partial_{y^i})-(n+1)H^{0i}\partial_{y^i},
\\
=&\ \Box_{\tilde{m}}+\phi^{\frac{n-1}{2}}\Theta_1(\tilde{h}) (\tilde{\partial}^2+\tilde{\partial}),  
\\
\gamma=&\ \gamma_0+\tfrac{n^2-1}{4}H^{00},\quad \gamma_0=-\tfrac{n^2-1}{4}.
\end{aligned}
\end{equation*}
\end{itemize}
We finish the proof.
\end{proof}

Next, we consider the commutator of $\Box_{\tilde{m}}$ and a basis of $\mathscr{V}_b(X)$. By Lemma \ref{lem.0}, we can choose a basis  $\mathscr{B}_j$ for $\mathscr{V}_b(\Omega_j)$ for $0\leq j\leq 5$ as follows:
\begin{equation}\label{def.basis}
\left\{
\begin{aligned}
&\mathscr{B}_0=\{\partial_{\mu}: \mu=0,1,...,n\},\\
&\mathscr{B}_j=\{Z_{\mu\nu}: \mu,\nu=0,1,...,n\},\ \textrm{for $j=1,2,4,5$,}\\
&\mathscr{B}_3=\{\partial_{\tau},\rho\partial_{\rho}, Z_{ij}: i,j=1,...,n\}.
\end{aligned}\right.
\end{equation}

\begin{lemma}\label{lem.4.7}
Suppose $(t,x)$ are harmonic coordinates w.r.t. $g$. Then in each domain $\Omega_j$ for $1\leq j\leq 5$,  with the choice of $\tilde{\rho}$ specified in Section \ref{sec.geosetting}, 
$$
[\Box_{\tilde{m}},Z_{00}]=[\Box_{\tilde{m}},Z_{ij}]=0, \quad
[\Box_{\tilde{m}},Z_{0i}]=c_i(\Box_{\tilde{m}}+\gamma_0) +\tilde{\partial}+c'_i, 
$$
for $1\leq i\leq n$, where $c_i,c'_i\in C^{\infty}(\overline{\Omega_j})$ and $\tilde{\partial}\in\mathscr{V}_b(\Omega_j)$.  In particular, in 
$\Omega_3$, 
$$
[\Box_{\tilde{m}},\partial_{\tau}]=0, 
\quad [\Box_{\tilde{m}},\rho\partial_{\rho}] =\Box_{\tilde{m}}-(\rho\partial_{\rho})^2-\rho\partial_{\rho}-\triangle_{\theta}
.$$
\end{lemma}
\begin{proof}
For any $u\in C^{\infty}(X)$, 
$$
\begin{aligned}
\ [\Box_{\tilde{m}},Z_{\mu\nu}] u =&\  (\Box_{\tilde{m}}+\gamma_0)Z_{\mu\nu}u -Z_{\mu\nu}(\Box_{\tilde{m}}+\gamma_0)u 
\\
=&\ \tilde{\rho}^{-\frac{n+3}{2}}\Box_m\tilde{\rho}^{\frac{n-1}{2}}Z_{\mu\nu}u - Z_{\mu\nu}\tilde{\rho}^{-\frac{n+3}{2}}\Box_{m}\tilde{\rho}^{\frac{n-1}{2}}u
\\
=&\ \tilde{\rho}^{-\frac{n+3}{2}}[\Box_m,Z_{\mu\nu}]\tilde{\rho}^{\frac{n-1}{2}}u 
+\tilde{\rho}^{-\frac{n+3}{2}}\Box_m[\tilde{\rho}^{\frac{n-1}{2}},Z_{\mu\nu}]u 
-[Z_{\mu\nu},\tilde{\rho}^{-\frac{n+3}{2}}]\Box_{m}\tilde{\rho}^{\frac{n-1}{2}}u
\\
=&\  c_{\mu\nu}(\Box_{\tilde{m}}+\gamma_0)u 
 -\tfrac{n-1}{2}(\Box_{\tilde{m}}+\gamma_0)\tilde{\rho}^{-1}Z_{\mu\nu}(\tilde{\rho})u
 +\tfrac{n+3}{2}\tilde{\rho}^{-1}Z_{\mu\nu}(\tilde{\rho})(\Box_{\tilde{m}}+\gamma_0)u
 \\
 =&\ (c_{\mu\nu}+2\tilde{\rho}^{-1}Z_{\mu\nu}(\tilde{\rho})) (\Box_{\tilde{m}}+\gamma_0)u 
 -\tfrac{n-1}{2}[\Box_{\tilde{m}},  \tilde{\rho}^{-1}Z_{\mu\nu}(\tilde{\rho})]u
\end{aligned}
$$
In the following, we study the term $\tilde{\rho}^{-1}Z_{\mu\nu}(\tilde{\rho})$ and its commutator with $\Box_{\tilde{m}}$. 
Recall that $[\Box_{m},Z_{\mu\nu}]=c_{\mu\nu}\Box_m$ with $c_{0i}=c_{ij}=0, c_{00}=2$ and
$$
\begin{gathered}
\rho^{-1} Z_{00}(\rho)=\phi^{-1} Z_{00}(\phi)=-1, \quad
\rho^{-1} Z_{ij}(\rho)=\phi^{-1} Z_{ij}(\phi)=0.
\end{gathered}
$$
Hence $[\Box_{\tilde{m}},Z_{00}]=[\Box_{\tilde{m}},Z_{ij}]=0$.  
Notice that we choose $\tilde{\rho}=\rho$ in $\Omega_1,\Omega_2,\Omega_3$ and $\tilde{\rho}=\phi$ in $\Omega_4,\Omega_5$ and
$$
\begin{gathered}
\rho^{-1} Z_{0i}(\rho) = -\frac{t\theta_i}{r}=\begin{cases}
-s\theta_i & \textrm{in $\Omega_1$} \\
-(1-a)\theta_i & \textrm{in $\Omega_2$}\\
-(1+\rho\tau)\theta_i & \textrm{in $\Omega_3$}
\end{cases}, \quad
\phi^{-1} Z_{0i}(\phi)=-\frac{x^i}{t}=\begin{cases}
-(1-\bar{a})\theta_i& \textrm{in $\Omega_4$}\\
-y^i& \textrm{in $\Omega_5$}
\end{cases}.
\end{gathered}
$$
Hence $ \tilde{\rho}^{-1}Z_{0i}(\tilde{\rho})\in C^{\infty}(\overline{\Omega_j})$ for all $1\leq i\leq n$ and $1\leq j\leq 5$. Since
$$
[\mathscr{V}_b(\Omega_j), C^{\infty}(\overline{\Omega_j})]\subset C^{\infty}(\overline{\Omega_j}),
$$
it is obviously for $j=1$ or $j=5$, 
$$
[\Box_{\tilde{m}},  \tilde{\rho}^{-1}Z_{0i}(\tilde{\rho})] \in \mathscr{V}_b(\Omega_j)+C^{\infty}(\overline{\Omega_j}).
$$
For $j=2,3,4$, by Lemma \ref{lem.2}, we only need to consider 
$$
[2\partial_{\rho_1}D,  \tilde{\rho}^{-1}Z_{0i}(\tilde{\rho})] =2\partial_{\rho_1}D(  \tilde{\rho}^{-1}Z_{0i}(\tilde{\rho})) +2D(  \tilde{\rho}^{-1}Z_{0i}(\tilde{\rho})) \partial_{\rho_1}+ 2\partial_{\rho_1}(  \tilde{\rho}^{-1}Z_{0i}(\tilde{\rho})) D
$$
where $2\partial_{\rho_1}D(  \tilde{\rho}^{-1}Z_{0i}(\tilde{\rho})),  2\partial_{\rho_1}(  \tilde{\rho}^{-1}Z_{0i}(\tilde{\rho})) \in C^{\infty}(\overline{\Omega_j})$ and
$D(  \tilde{\rho}^{-1}Z_{0i}(\tilde{\rho}))\in \rho_1 C^{\infty}(\overline{\Omega_j})$. Hence
$$
[2\partial_{\rho_1}D,  \tilde{\rho}^{-1}Z_{0i}(\tilde{\rho})]  \in \mathscr{V}_b(\Omega_j)+C^{\infty}(\overline{\Omega_j}).
$$
We finish the proof.
\end{proof}

\subsection{Conformal Transformation of Reduced Einstein Equations}\label{sec.conftransform}
With the conformal transformation (\ref{eq.11}),
the reduced Einstein vacuum equations (\ref{eq.2}) is equivalent to the following system:
\begin{equation}\label{eq.12}
\begin{gathered}
(\Box_{\tilde{g}}+\gamma)\tilde{h}_{\mu\nu} = (\rho_0\rho_2)^{\frac{n-1}{2}} \rho_1^{\frac{n-5}{2}} \tilde{F}_{\mu\nu}(\tilde{h},\tilde{\partial}\tilde{h}),
\end{gathered}
\end{equation}
where $\gamma=\gamma_0+\tilde{\rho}^{\frac{n-1}{2}}\Theta_1(\tilde{h})$ is given in Lemma \ref{lem.2} and 
\begin{equation*}
\begin{aligned}
\tilde{F}_{\mu\nu}=&\ (\rho_0\rho_2)^{-(n+1)}\rho_1^{-(n-1)}F_{\mu\nu}(\tilde{\rho}^{\frac{n-1}{2}}\tilde{h})(\partial(\tilde{\rho}^{\frac{n-1}{2}}\tilde{h}), \partial(\tilde{\rho}^{\frac{n-1}{2}}\tilde{h})).
\end{aligned}
\end{equation*}
Using the explicit formula of $F_{\mu\nu}$ given in \cite{LR1}, we have
\begin{lemma} \label{lem.4.8}
For $n\geq 3$, 
$$\tilde{F}_{\mu\nu}=\Theta_0(\tilde{h})(\tilde{\partial}\tilde{h},\tilde{\partial}\tilde{h})+\rho_1\Theta_1(\tilde{h})(\tilde{\partial}\tilde{h})+\rho_1\Theta_2(\tilde{h}).
$$
In particular, in domains $\Omega_2,\Omega_3,\Omega_4$, 
$$
\tilde{F}_{\mu\nu}=\theta_{\mu}\theta_{\nu}\big(\tfrac{1}{4}(D\mathrm{tr}_m\tilde{h})^2-\tfrac{1}{2}D\tilde{h}^{\alpha}_{\beta}D\tilde{h}^{\beta}_{\alpha}\big) +\rho_1\big(\Theta_0(\tilde{h})(\tilde{\partial}\tilde{h},\tilde{\partial}\tilde{h})+\Theta_1(\tilde{h})(\tilde{\partial}\tilde{h})\big)+\rho_1^2\Theta_2(\tilde{h})
$$
where $D\in \mathscr{V}_b(\Omega_i)$ is defined in Lemma \ref{lem.4}. 
\end{lemma}
\begin{proof}
By \cite{LR1}, $F_{\mu\nu}(\partial h,\partial h)=P_{\mu\nu}(\partial h,\partial h)+Q_{\mu\nu}(\partial h,\partial h)+G(h)_{\mu\nu}(\partial h,\partial h)$, where
$$
\begin{aligned}
P_{\mu\nu}(\partial h,\partial h) = \tfrac{1}{4}\partial_{\mu}(\mathrm{tr}_mh)\partial_{\nu}(\mathrm{tr}_mh)-\tfrac{1}{2}\partial_{\mu}h_{\alpha}^{\beta}\partial_{\nu}h^{\alpha}_{\beta},
\end{aligned}
$$
$Q_{\mu\nu}(\partial h,\partial h)$ satisfies the null condition and 
$$
G(h)_{\mu\nu}(\partial h,\partial h)=\Theta_1(h)(\partial h, \partial h).
$$
First notice that $[\partial ,\tilde{\rho}^{\frac{n-1}{2}}]\in\tilde{\rho}^{\frac{n+1}{2}}\mathscr{A}^{0,0,0}(X)$. 
Hence by the conformal transformation (\ref{eq.11}), 
$$
\begin{aligned}
P_{\mu\nu}(\partial h,\partial h) =&\ \tilde{\rho}^{n-1}\big((\rho_0\rho_2)^2\Theta_0(\tilde{h})(\tilde{\partial}\tilde{h},\tilde{\partial}\tilde{h})+\rho_0\rho_2\tilde{\rho}\Theta_1(\tilde{h})(\tilde{\partial}\tilde{h})+ \tilde{\rho}^{2}\Theta_2(\tilde{h})\big),
\\
Q_{\mu\nu}(\partial h,\partial h) =&\  \tilde{\rho}^{n}\rho_0\rho_2\big( \Theta_0(\tilde{h})(\tilde{\partial}\tilde{h},\tilde{\partial}\tilde{h})+\Theta_1(\tilde{h})(\tilde{\partial}\tilde{h})\big )+ \tilde{\rho}^{n+1}\Theta_2(\tilde{h}),
\\
G(h)_{\mu\nu}(\partial h,\partial h)=&\ \tilde{\rho}^{\frac{3n-3}{2}} \big( 
(\rho_0\rho_2)^2\Theta_1(\tilde{h})(\tilde{\partial}\tilde{h},\tilde{\partial}\tilde{h}) + \rho_0\rho_2\tilde{\rho}\Theta_2(\tilde{h})(\tilde{\partial}\tilde{h}) +\tilde{\rho}^2\Theta_3(\tilde{h})
\big).
\end{aligned}
$$
In particular, in domain $\Omega_2,\Omega_3,\Omega_4$, 
$$
\begin{aligned}
P_{\mu\nu}(\partial h,\partial h) =&\ \tilde{\rho}^{n-1}
(\rho_0\rho_2)^2 \theta_{\mu}\theta_{\nu}\big(\tfrac{1}{4}(D\mathrm{tr}_m\tilde{h})^2-\tfrac{1}{2}D\tilde{h}^{\alpha}_{\beta}D\tilde{h}^{\beta}_{\alpha}\big)
\\
&\ +\tilde{\rho}^n\rho_0\rho_2\big(\Theta_0(\tilde{h})(\tilde{\partial}\tilde{h},\tilde{\partial}\tilde{h})+\Theta_1(\tilde{h})(\tilde{\partial}\tilde{h})\big)+\tilde{\rho}^{n+1}\Theta_2(\tilde{h}).
\end{aligned}
$$
We finish the proof.
\end{proof}

Suppose $h$ is a solution of the equation (\ref{eq.12}). Then we can insert vector fields which are tangent to the boundary of $X$ and have the following lemma.
\begin{lemma}\label{lem.6}
Let $\tilde{D}$ be any vector field in $\mathscr{B}_j$ for $1\leq j\leq 5$ and $I$ a multi-index. Then in each domain $\Omega_j$ for $1\leq j\leq 5$, the equations (\ref{eq.12}) implies that
$$
(\Box_{\tilde{g}}+\gamma)\tilde{D}^I\tilde{h}_{\mu\nu}= \sum_{0\leq i\leq |I|+1}\tilde{\partial}^i \tilde{h}_{\mu\nu}+ \tilde{f}^I_{\mu\nu}(\tilde{h}),
$$
where in $\Omega_1,\Omega_5$,
$$
\tilde{f}^I_{\mu\nu}(\tilde{h}) =\tilde{\rho}^{\frac{n-1}{2}}\left(\sum _{\alpha_1+\cdots\alpha_l\leq |I|+2,\ 2\leq l\leq |I|+2,\ 0\leq \alpha_i\leq |I|+1} \Theta_0(\tilde{h})(\tilde{\partial}^{\alpha_1}\tilde{h},\cdots,\tilde{\partial}^{\alpha_l}\tilde{h})\right);
$$
 in $\Omega_2, \Omega_3, \Omega_4$,
$$
\begin{aligned}
\tilde{f}^I_{\mu\nu}(\tilde{h}) =\ & (\rho_0\rho_2)^{\frac{n-1}{2}}\rho_1^{\frac{n-5}{2}}\left(\sum_{\alpha_1+\alpha_2\leq |I|}\Theta_0(\tilde{h})(\tilde{D}^{\alpha_1}D\tilde{h},\tilde{D}^{\alpha_2}D\tilde{h})+\sum_{\alpha_1+\alpha_2\leq |I|+2}\Theta_0(\tilde{h})(\tilde{D}^{\alpha_1}\tilde{h}_{\rho\rho},\tilde{D}^{\alpha_2}\tilde{h})\right)
\\&+(\rho_0\rho_2)^{\frac{n-1}{2}}\rho_1^{\frac{n-3}{2}}
\left(\sum _{l\geq 2,\ \alpha_1+\cdots\alpha_l\leq |I|+2,\ \alpha_i\leq |I|+1} \Theta_0(\tilde{h})(\tilde{D}^{\alpha_1}\tilde{h},\cdots,
\tilde{D}^{\alpha_l}\tilde{h})\right).
\end{aligned}
$$
\end{lemma}
\begin{proof}
It is obviously true if $|I|=0$ by Lemma \ref{lem.4.8}. For $|I|\geq 1$, 
$$
(\Box_{\tilde{g}}+\gamma)\tilde{D}^I\tilde{h}_{\mu\nu}
= [\Box_{\tilde{g}}-\Box_{\tilde{m}} +\gamma-\gamma_0,\tilde{D}^I] \tilde{h}_{\mu\nu}+\tilde{D}^I\tilde{f}^0_{\mu\nu}(\tilde{h}) +[\Box_{\tilde{m}},\tilde{D}^I]\tilde{h}_{\mu\nu}.
$$
It is obvious the first two terms on the right hand side are nonlinear and can be expressed as above. For the third term, from Lemma \ref{lem.4.7}, we have
$$
[\Box_{\tilde{m}},\tilde{D}]=c\Box_{\tilde{m}}+\tilde{\partial}^2+\tilde{\partial}+c'
$$
for some $c,c'\in C^{\infty}(\overline{\Omega}_i)$. Then for $|I|>0$
$$
\begin{aligned}
\  [\Box_{\tilde{m}},\tilde{D}^I]\tilde{h}_{\mu\nu}
 =&\ \sum_{|K|+|J|=|I|-1}C_{K,J}\tilde{D}^{K}[\Box_{\tilde{m}},\tilde{D}]\tilde{D}^{J}\tilde{h}_{\mu\nu}
 \\
 =&\ \sum_{|K|+|J|=|I|-1}C_{K,J}\tilde{D}^{K}\big(c(\Box_{\tilde{g}}+\gamma)+c(\Box_{\tilde{m}}-\Box_{\tilde{g}}+\gamma_0-\gamma)+\tilde{\partial}^2+\tilde{\partial}+c'-c\gamma_0\big)\tilde{D}^{J}\tilde{h}_{\mu\nu}
 \end{aligned}
$$
where $C_{K,J}$ are constants. By induction on $I$ and Lemma \ref{lem.2} , $ [\Box_{\tilde{m}},\tilde{D}^I]\tilde{h}_{\mu\nu}$ is equal to $\sum_{0\leq i\leq |I|+1}\tilde{\partial}^i \tilde{h}_{\mu\nu} $ plus a nonlinear term which can be expressed as above. 
We finish the proof. 
\end{proof}

\vspace{0.2in}
\section{Time-like Functions.} \label{sec.timelikefun}
Time-like function is the most important concept associated to a Lorentzian metric: we use it to define
space-like hypersurfaces, positive quadratic forms as well as energy norms. In this section, we consider the time-like functions w.r.t. $\tilde{m}=\tilde{\rho}^2m$ and $\tilde{g}=\tilde{\rho}^2g=\tilde{m}+\tilde{\rho}^{\frac{n+3}{2}}\tilde{h}$ with $\tilde{h}$ small in the interior of $X$.  

\begin{definition}
Suppose $\tilde{g}$ extends to a Lorentz b-metric on $[X]^2$ of signature $(n,1)$, $T\in C^1(\mathrm{Int}([X]^2))$ and $p\in \mathrm{Int}([X]^2)$. We say $T$  is \textit{time-like} w.r.t. $\tilde{g}$ at $p$ iff
\begin{equation*}
\langle \nabla T,\nabla T\rangle_{\tilde{g}}<0 \quad\textrm{at $p$}, 
\end{equation*}
\textit{and null} w.r.t. $\tilde{g}$ at $p$ iff
\begin{equation*}
\langle \nabla T,\nabla T\rangle_{\tilde{g}}=0 \quad\textrm{at $p$}\ .
\end{equation*}
A hypersurface $\Sigma\subset X$ is called \textit{space-like} (resp. \textit{null}) w.r.t. $\tilde{g}$ iff $\Sigma$ has a defining function $T\in C^1(\mathrm{Int}([X]^2))$ such that $T$ is time-like (resp. null) on $\Sigma$. 
\end{definition}
\begin{definition}\label{def.4}
For a time like function $T$ w.r.t. $\tilde{g}$, define a vector field $\mathcal{F}_{\tilde{g}}(T,v)$ associate to $T$ and quadratic in $v$ by
\begin{equation*}
\mathcal{F}_{\tilde{g}}(T,v)= \langle \nabla T,\nabla v\rangle_{\tilde{g}} \nabla v -\tfrac{1}{2} \langle \nabla v,\nabla v\rangle_{\tilde{g}} \nabla T +\tfrac{1}{2}\gamma_0 v^2\nabla T . 
\end{equation*}
Here $\nabla$ is the covariant derivative w.r.t. $\tilde{g}$ and $\gamma_0<0$ is some constant. 
\end{definition}
\begin{lemma}
Suppose $T,T'$ are two time-like functions w.r.t. a Lorentzian metric $\tilde{g}$ at $p\in \mathrm{Int}([X]^2)$ such that $\langle \nabla T,\nabla T'\rangle_{\tilde{g}}<0$ at $p$, then $\langle \mathcal{F}_{\tilde{g}}(T,v),\nabla T'\rangle_{\tilde{g}}$ is a strict positive quadratic form of $(\nabla v, v)$ at $p$. 
\end{lemma}
See Chapter 6 in \cite{Ho} for proof in details. Moreover, we have the following formulas from the divergence theorem directly, which implies the energy estimates if choosing $T$ and $\Omega$ properly.

\begin{lemma}\label{lem.3}
Suppose that the domain $\Omega\subset \mathring{X}$ has piecewise smooth boundary $\partial\Omega$ with defining function $T'$ and $T$ is a $C^2$ time-like function w.r.t. $\tilde{g}$ on $\overline{\Omega}\cap \mathring{X}$. Assume $T''\in C^1(\overline{\Omega}\cap \mathring{X})$. Then 
\begin{equation*}
\int_{\Omega}\mathrm{div}_{\tilde{g}} (e^{-2T''}\mathcal{F}_{\tilde{g}}(T,v))dvol_{\tilde{g}} = \int_{\partial\Omega} e^{-2T''}\langle \mathcal{F}_{\tilde{g}}(T,v), \nabla T'\rangle_{\tilde{g}} d\mu_{\tilde{g}}^{T'}.
\end{equation*}
Here $d\mu_{\tilde{g}}^{T'}$ is the volume form on $\partial\Omega$  such that $dT'\wedge d\mu_{\tilde{g}}^{T'}= dvol_{\tilde{g}}$ and $\partial\Omega$ is oriented by $d\mu_{\tilde{g}}^{T'}>0$.  Notice that  $\nabla$ is the covariant derivative w.r.t. $\tilde{g}$ and 
\begin{equation}\label{eq.div}
\begin{gathered}
\mathrm{div}_{\tilde{g}}(e^{-2T''}\mathcal{F}_{\tilde{g}}(T,v))=e^{-2T''}\big( -2\langle \mathcal{F}_{\tilde{g}}(T,v), \nabla T''\rangle_{\tilde{g}}+\langle\nabla v, \nabla T\rangle_{\tilde{g}} (\Box_{\tilde{g}}+\gamma_0)v + \mathcal{Q}_{\tilde{g}}(T,v)\big),
\\
\mathrm{where}\quad
\mathcal{Q}_{\tilde{g}}(T,v)= \tfrac{1}{2}\Box_{\tilde{g}}T(\gamma_0v^2-\langle\nabla v,\nabla v\rangle_{\tilde{g}})+\nabla^2T(dv,dv).
\end{gathered}
\end{equation}
\end{lemma}
\begin{figure}[htp]
\centering
\includegraphics[totalheight=2.5in,width=3.75in]{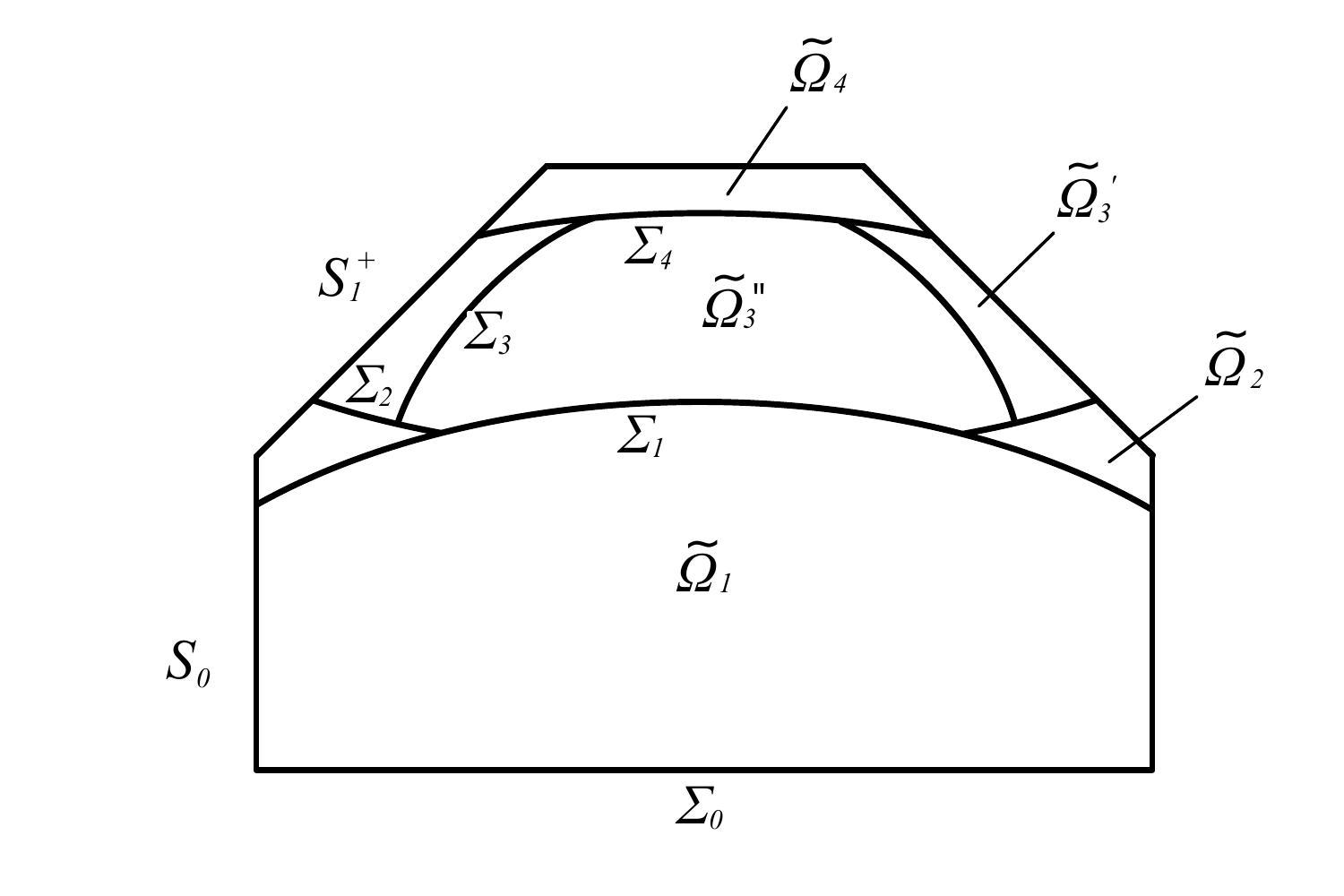}
\caption{Division of $X_{t\geq 0}$ by space like hypersurfaces and null hypersurfaces.}\label{fig.2}
\end{figure}
For the purpose of doing energy estimates, we divide $X_{t\geq 0}$ into four domains $\{\widetilde{\Omega}_i\}_{1\leq k\leq 4}$ by space-like hypersurfaces and null hypersurfaces w.r.t. $\tilde{m}$, or equivalently w.r.t. $\tilde{g}$. See Figure \ref{fig.2}. We list the time-like functions used to define those surfaces, which are also used to define the energy norms in next section, domain by domain.  For $1\leq i\leq 4$, we use $T_i,T_i'$ to denote time-like functions in $\widetilde{\Omega}_i\cap \mathring{X}$ such that $T_i'$  extends to a null function on $S_1^{+}\cap \widetilde{\Omega}_i$. 
We use $T_i''$ to denote the weight functions if necessary. We also compute $\mathcal{F}_{\tilde{g}}(T_i,v)$,  $\langle \mathcal{F}_{\tilde{g}}(T_i,v),\nabla T_i'\rangle_{\tilde{g}}$,  $\langle \mathcal{F}_{\tilde{g}}(T_i,v),\nabla T_i''\rangle_{\tilde{g}}$ and $\mathcal{Q}_{\tilde{g}}(T_i,v)$ for $1\leq i\leq 4$.  

\textbf{Assumption.} Through out this section, we assume the following:
\begin{itemize}
\item [(A1)]
 $(t,x)$ are harmonic coordinates w.r.t. $g=m+\tilde{\rho}^{\frac{n-1}{2}}\tilde{h}$.
 \item [(A2)]
 $|(\rho_0\rho_2)^{\frac{n-1}{2}}\tilde{h}|$, $|(\rho_0\rho_2)^{\frac{n-1}{2}}\tilde{\partial}\tilde{h}|$, $|(\rho_0\rho_2)^{\frac{n-1}{2}}\rho_1^{-\delta}\tilde{h}_{\rho\rho}|$, $|(\rho_0\rho_2)^{\frac{n-1}{2}}\rho_1^{-\delta}\tilde{\partial}\tilde{h}_{\rho\rho}|$ are small enough for some 
 $0<\delta<\tfrac{1}{2}$. This means $\tilde{h}_{\rho\rho}$ is bounded by $C\rho_1^{\delta}$ when approaching the null infinity $S_1^{\pm}$. 
 \item [(A3)]  $\delta' \in (0,\delta)$ satisfies $\frac{1}{2}-\delta' < 1-2\delta$ and $\alpha >1$. 
\end{itemize}

In Section 6, we will prove given small initial data $(h^0, h^1)=(O(\rho_0^{\frac{n-1}{2}+\delta}),O( \rho_0^{\frac{n+1}{2}+\delta}))$ with enough regularity the Einstein vacuum solutions will satisfy (A1)  and (A2). And with the choice of $\delta'$ satisfy (A3), the time-like functions we choose in the following will always be time-like.

\subsection{In $\widetilde{\Omega}_1$.}
In $\Omega_1\cup \Omega_0$, define
\begin{equation*}
\begin{gathered}
T_1=\frac{t}{\psi_1(r)};
\\
\widetilde{\Omega}_1=\{0\leq T_1\leq \tfrac{3}{4}\} ,\quad \Sigma_0=\{T_1=0\},\quad
\Sigma_1=\{T_1=\tfrac{3}{4}\}.
\end{gathered}
\end{equation*}
where $\psi_1\in C^2(\widetilde{\Omega}_1)$ is defined by 
$$
\psi_1(r)= \begin{cases}\tfrac{3}{8}+\tfrac{3}{4}r^2-\tfrac{1}{8}r^4 & \textrm{for $r\leq 1$} \\ r &\textrm{for $r>1$}\end{cases}
$$
Hence $0\leq \psi_1'\leq 1$. 
Then in $\widetilde{\Omega}_1$,
\begin{equation*}
\begin{gathered}
\langle \nabla T_1,\nabla T_1\rangle_{\tilde{m}} = (\tilde{\rho}\psi_1(r))^{-2}(-1+T^2(\psi'_1(r))^2)
\leq -\tfrac{7}{16}(\tilde{\rho}\psi_1(r))^{-2}<-C 
\end{gathered}
\end{equation*}
for some constant $C>0$. Hence $T_1$ is a regular time-like function w.r.t. $\tilde{m}$ all over $\widetilde{\Omega}_1$ and bounded. We choose $\tilde{\rho}=\rho_0=\psi_1^{-1}$ and $\rho_1=\rho_2=1$ in $\widetilde{\Omega}_1$. 
Here $\widetilde{\Omega}_1\cap \Omega_0$ is bounded in the interior of $X$. 
Near the boundary of $\widetilde{\Omega}_1$, i.e. in $\widetilde{\Omega}_1\cap \Omega_1$, recall that we use the following coordinates and boundary defining functions:
$$
(s,\rho,\theta)=(\frac{t}{r},\frac{1}{r},\frac{x}{r}); \quad
\tilde{\rho}=\rho_0=\rho,\ \rho_1=\rho_2=1.
$$
Let $\gamma_0<0$ be a constant and we have the following two lemmas. 
\begin{lemma}\label{lem.5.1}
In $\widetilde{\Omega}_1\cap \Omega_1$, 
\begin{equation*} 
\begin{aligned}
\langle \mathcal{F}_{\tilde{m}}(T_1,v),\nabla T_1\rangle_{\tilde{m}}
=&\ \tfrac{1}{2}|(1-s^2)\partial_sv-\rho\partial_{\rho}v|^2 +\tfrac{1}{2}|\rho\partial_{\rho}v|^2+\tfrac{1}{2}(1-s^2)(|\slashpar_{\theta}v|^2-\gamma_0v^2), \\
\mathcal{Q}_{\tilde{m}}(T_1,v)=&\  -s(|\slashpar_{\theta}v|^2-\gamma_0 v^2).
\end{aligned}
\end{equation*}
Here $\nabla$ is the covariant derivative w.r.t. $\tilde{m}$.
\end{lemma}
\begin{proof}
In $\widetilde{\Omega}_1\cap \Omega_1$, change variable to $\xi=-\ln\rho\in (0,\infty)$ and hence $\frac{d\rho}{\rho}=-d\xi$ and $\rho\partial_{\rho}=-\partial_{\xi}$. Then in coordinates $\{s,\xi,\theta\}$, 
\begin{equation*}
\tilde{m}=-d^2s-2sdsd\xi+(1-s^2)d^2\xi+d^2\theta.
\end{equation*}
It is a product type metric. WLOG, at any fixed point $p=(s,\xi,\theta)$, let $\theta$ denote the normal coordinate w.r.t. standard spherical metric $d^2\theta$ on $\mathbb{S}^{n-1}$. Then at $p$, we can write the metric components as follows:
\begin{equation*}
\begin{gathered}
\tilde{m}=\left[\begin{array}{ccc}-1 & -s & 0\\ -s & 1-s^2& 0\\0&0&\mathbbm{1}_{n-1}\end{array}\right], \quad 
\tilde{m}^{-1}=\left[\begin{array}{ccc}-1+s^2 & -s & 0\\ -s & 1 &0\\ 0&0& \mathbbm{1}_{n-1}\end{array}\right].
\end{gathered}
\end{equation*}
For the quadratic form, notice that
\begin{equation*}
\begin{gathered}
\langle \mathcal{F}_{\tilde{m}}(T_1,v),\nabla T_1\rangle_{\tilde{m}} =  
 \langle\nabla T_1,\nabla v\rangle^2_{\tilde{m}}-\tfrac{1}{2} \langle\nabla v,\nabla v\rangle_{\tilde{m}} \langle\nabla T_1,\nabla T_1\rangle_{\tilde{m}}+\tfrac{1}{2}\gamma_0 \langle\nabla T_1,\nabla T_1\rangle_{\tilde{m}}v^2,
\end{gathered}
\end{equation*}
where
\begin{equation*}
\begin{aligned}
 \langle\nabla T_1,\nabla T_1\rangle_{\tilde{m}}=&\ -1+s^2,
\\
 \langle\nabla T_1,\nabla v\rangle_{\tilde{m}} =&\
(-1+s^2)\partial_sv-s\partial_{\xi}v,
\\
\langle\nabla v,\nabla v\rangle_{\tilde{m}}
=&\ (-1+s^2)|\partial_sv|^2-2s\partial_sv\partial_{\xi}v+|\partial_{\xi}v|^2+|\slashpar_{\theta}v|^2.
\end{aligned}
\end{equation*} 
Then the formula for $\langle F(T_1,v),\nabla T_1\rangle_{\tilde{m}}$ follows directly.
For $\mathcal{Q}_{\tilde{m}}(T_1,v)$,  by Lemma \ref{lem.3} we only  need to compute $\Box_{\tilde{m}} T_1$ and $\nabla^2 T_1(dv,dv)$. Using to the formula of $\Box_{\tilde{m}}$ in the proof of Lemma \ref{lem.2}, we have
\begin{equation*}
\begin{gathered}
\Box_{\tilde{m}} T_1=2s,\\
\begin{aligned}
 \nabla^2T_1(dv,dv)
= (\partial_{I}\partial_{J}T_1-\Gamma^{K}_{IJ}\partial_{K}T_1)\nabla^{I}v\nabla^{J}v.
\end{aligned}
\end{gathered}
\end{equation*}
Here $I,J,K\in\{s,\xi,\theta\}$. 
At $p$, the non-zero connection components are 
\begin{equation*}
\begin{gathered}
\Gamma_{ss}^s(\tilde{m})=s,\quad \Gamma_{s\xi}^s(\tilde{m})=s^2,\quad 
\Gamma_{\xi\xi}^s(\tilde{m}) =-s(1-s^2),
\\
\Gamma_{ss}^{\xi}(\tilde{m})=-1,\quad \Gamma_{s\xi}^{\xi}(\tilde{m})=-s, \quad \Gamma_{\xi\xi}^{\xi}(\tilde{m})=-s^2.
\end{gathered}
\end{equation*}
Hence
\begin{equation*}
\begin{aligned}
 \nabla^2T_1(dv,dv)
= -s(1-s^2)|\partial_sv|^2-2s^2\partial_sv\partial_{\xi}v+s|\partial_{\xi}v|^2
\end{aligned}
\end{equation*}
%
%
%
%
and $\mathcal{Q}_{\tilde{m}} (T_1,v)$ follows. 
\end{proof}

Those two formulae do not change much  in the following sense when we take the background metric to be $\tilde{g}$.
\begin{lemma}\label{lem.5.2}
In $\widetilde{\Omega}_1$, 
\begin{equation*} 
\begin{gathered}
dvol_{\tilde{g}}=(1+\rho_0^{\frac{n-1}{2}}\Theta_1(\tilde{h})) dvol_{\tilde{m}},
\\
\langle  \mathcal{F}_{\tilde{g}}(T_1,v),\nabla T_1\rangle_{\tilde{g}}-\langle  \mathcal{F}_{\tilde{m}}(T_1,v),\nabla T_1\rangle_{\tilde{m}}= 
\rho_0^{\frac{n-1}{2}}\Theta_1(\tilde{h})(v, \tilde{\partial}v),
\\
\mathcal{Q}_{\tilde{g}} (T_1,v)-\mathcal{Q}_{\tilde{m}} (T_1,v)
= \rho_0^{\frac{n-1}{2}}\Theta_1(\tilde{h})(v,\tilde{\partial}v)+\rho_0^{\frac{n-1}{2}}\Theta_0(\tilde{h})(\tilde{\partial}\tilde{h})(\tilde{\partial}v,\tilde{\partial}v).
\end{gathered}
\end{equation*}
\end{lemma}
\begin{proof}
We only need to concern the region near boundary.  In $\widetilde{\Omega}_1\cap\Omega_1$, let us do the similar computation as in the proof of Lemma \ref{lem.5.1} by using the same coordinate $(s,\rho,\theta)$.  First, by the proof of Proposition \ref{prop.4}, write 
\begin{equation*}
\tilde{g}= \tilde{m}+\Theta_1(h)= \tilde{m}+\rho_0^{\frac{n-1}{2}}\Theta_1(\tilde{h}),\quad \tilde{g}^{-1}=\tilde{m}^{-1}+\Theta_1(h)= \tilde{m}^{-1}+\rho_0^{\frac{n-1}{2}}\Theta_1(\tilde{h}).
\end{equation*}
Hence the first formula follows directly by
\begin{equation*}
 \begin{aligned}
  \langle\nabla T_1,\nabla T_1\rangle_{\tilde{g}}-\langle\nabla T_1,\nabla T_1\rangle_{\tilde{m}}
  =&\ \rho_0^{\frac{n-1}{2}}\Theta_1(\tilde{h}),
\\
 \langle\nabla T_1,\nabla v\rangle_{\tilde{g}}-\langle\nabla T_1,\nabla v\rangle_{\tilde{m}}=&\ 
  \rho_0^{\frac{n-1}{2}}\Theta_1(\tilde{h}) (\tilde{\partial}v),
\\
  \langle\nabla v,\nabla v\rangle_{\tilde{g}}-\langle\nabla v,\nabla v\rangle_{\tilde{m}}=&\ 
   \rho_0^{\frac{n-1}{2}}\Theta_1(\tilde{h}) (\tilde{\partial}v,\tilde{\partial}v).
 \end{aligned}
\end{equation*}
According to the formula of $\Box_{\tilde{g}}$ in the proof of Lemma \ref{lem.2}, 
\begin{equation*}
\begin{gathered}
 \Box_{\tilde{g}}T_1-\Box_{\tilde{m}}T_1=\rho_0^{\frac{n-1}{2}}\Theta_1(\tilde{h})(\tilde{\partial}^2+\tilde{\partial})T_1=\rho_0^{\frac{n-1}{2}}\Theta_1(\tilde{h});
 \\
 \Gamma^{K}_{IJ}(\tilde{g})=\Gamma^{K}_{IJ}(\tilde{m})
+\rho_0^{\frac{n-1}{2}}(\Theta_1(\tilde{h}) +\Theta_0(\tilde{h})(\tilde{\partial}\tilde{h})), 
\\
\tilde{g}^{IJ}\partial_{I}v =\tilde{m}^{IJ}\partial_{I}v +  \rho_0^{\frac{n-1}{2}}\Theta_1(\tilde{h})(\tilde{\partial}v),
\end{gathered}
\end{equation*}
which implies the second formula. Here $I,J,K\in\{s,\xi,\theta\}$.
\end{proof}


 \subsection{In $\widetilde{\Omega}_2$.}
In $\Omega_2$, recall that the coordinates and boundary defining functions are as follows:
$$
(a,b,\theta)=(1-\frac{t}{r}, \frac{1}{r-t},\frac{x}{r}); \quad \rho_0=b,\ \rho_1=a,\ \rho_2=1,\ \tilde{\rho}=ab.$$
Define for $\tau_0>8$
\begin{equation*}
\begin{gathered}
T_2=-\tfrac{2}{1+2\delta'}a^{\delta'+\frac{1}{2}}+\log b, 
\quad T_2'=-a, \quad T_2''=\log b;
\\
\widetilde{\Omega}_2=\{ -\tfrac{1}{4}\leq T_2'\leq 0,\ -\infty< T_2\leq -\ln\tau_0\},\quad 
\Sigma_2=\{T_2=-\ln\tau_0\}\cap \widetilde{\Omega}_2.
\end{gathered}
\end{equation*}
Here $T_2'$ is a natural extension of $T_1-1$ into $\Omega_2$.
Then in $\widetilde{\Omega}_2$,
\begin{equation*}
\begin{aligned}
&\langle \nabla T_2,\nabla T_2\rangle_{\tilde{m}}=-a^{\delta'-\frac{1}{2}}(2+a^{\delta'+\frac{1}{2}}(2-a)),
\\
&\langle \nabla T'_2,\nabla T'_2\rangle_{\tilde{m}}= -a(2-a). 
\end{aligned}
\end{equation*}
Hence $T_2, T_2'$ are time-like functions all over $\widetilde{\Omega}_2\cap \mathring{X}$ and $T'_2$ is null on $\widetilde{\Omega}_2\cap S_1$ w.r.t. $\tilde{m}$.  In $\widetilde{\Omega}_2$, with $\gamma_0<0$ a constant, we have the following two lemmas. 

\begin{lemma}\label{lem.5.3}
In $\widetilde{\Omega}_2$,
\begin{equation*}
\begin{aligned}
\langle  \mathcal{F}_{\tilde{m}}(T_2,v),\nabla T_2' \rangle_{\tilde{m}} =&\ 
\tfrac{1}{2} a^{\delta'-\frac{1}{2}}(|b\partial_bv|^2+|b\partial_bv-a(2-a)\partial_av|^2)
+\tfrac{1}{2}a(2-a)|\partial_av|^2 
\\
&\ +\tfrac{1}{2}(1+a^{\delta'+\frac{1}{2}}(2-a))(|\slashpar_{\theta}v|^2-\gamma_0v^2),
\\
\langle  \mathcal{F}_{\tilde{m}}(T_2,v),\nabla T_2'' \rangle_{\tilde{m}}=&\ 
|\partial_a v|^2+\tfrac{1}{2}a^{\delta'-\frac{1}{2}}(a(2-a)|\partial_a v|^2+|\slashpar_{\theta}v|^2-\gamma_0v^2), 
\\
\mathcal{Q}_{\tilde{m}} (T_2,v)=&\ 
\tfrac{1}{2}(\tfrac{1}{2}-\delta')a^{\delta'-\frac{3}{2}}
(|b\partial_bv|^2+|b\partial_bv-a(2-a)\partial_av|^2)
+(1-a)|\partial_av|^2
 \\
 &\ - \tfrac{1}{2}a^{\delta'-\frac{1}{2}}
(1+2\delta'-(\tfrac{3}{2}+\delta')a)(|\slashpar_{\theta}v|^2-\gamma_0v^2).
\end{aligned}
\end{equation*}
\end{lemma}
\begin{proof}
In $\widetilde{\Omega}_2$, change variable to $\xi=-\ln b\in (\ln\tau_0,\infty)$ and hence $\frac{db}{b}=-d\xi, b\partial_b=-\partial_{\xi}$. Then in coordinate $(a,b,\theta)$, 
$$
\tilde{m}=-2dad\xi+a(2-a)d^2\xi+d^2\theta. 
$$
 With $\theta$ the normal coordinates w.r.t. standard spherical metric $d^2\theta$ at $p\in \widetilde{\Omega}_2\cap\mathrm{Int}(X^2)$,  we
 can write the metric components at $p$ as follows:
$$
\tilde{m}=\left[\begin{array}{ccc}0 & -1 & 0\\ -1& a(2-a) &0\\ 0&0&\mathbbm{1}_{n-1}\end{array}\right], \quad 
\tilde{m}^{-1}=\left[\begin{array}{ccc}-a(2-a) &-1&0 \\ -1 & 0&0\\0&0&\mathbbm{1}_{n-1}\end{array}\right].
$$
Then $\langle  \mathcal{F}_{\tilde{m}}(T_2,v),\nabla T_2' \rangle_{\tilde{m}}$ and $\langle  \mathcal{F}_{\tilde{m}}(T_2,v),\nabla T_2'' \rangle_{\tilde{m}}$ follow from the following computation:
$$ 
\begin{aligned}
&\langle \nabla T_2,\nabla T_2'\rangle_{\tilde{m}} =-1-a^{\delta'+\frac{1}{2}}(2-a), 
\\
&\langle \nabla T_2,\nabla T_2''\rangle_{\tilde{m}} =-a^{\delta'-\frac{1}{2}},
\\
&\langle \nabla T_2,\nabla v\rangle_{\tilde{m}} = a^{\delta'-\frac{1}{2}}(a(2-a)\partial_av+\partial_{\xi}v)+\partial_av,
\\
&\langle \nabla T'_2,\nabla v\rangle_{\tilde{m}}=a(2-a)\partial_av+\partial_{\xi}v,
\\
&\langle \nabla T''_2,\nabla v\rangle_{\tilde{m}}=\partial_av, 
\\
&\langle \nabla v,\nabla v\rangle_{\tilde{m}} =-2\partial_{a}v\partial_{\xi}v-a(2-a)|\partial_{a}v|^2+|\slashpar_{\theta}v|^2. 
\end{aligned}
$$
For $\mathcal{Q}_{\tilde{m}}(T_2,v)$,  we have
$$
\Box_{\tilde{m}}T_2=a^{\delta'-\frac{1}{2}} (1+2\delta'-a(\tfrac{3}{2}+\delta'))
$$
and the non-zero connection components at $p$ are 
\begin{equation*}
\begin{gathered}
\Gamma^a_{a\xi}(\tilde{m})=-(1-a),\quad
\Gamma^a_{\xi\xi}(\tilde{m})=a(1-a)(2-a),
\quad\Gamma^{\xi}_{\xi\xi}(\tilde{m})=1-a,
\end{gathered}
\end{equation*}
which implies that  
$$
\begin{aligned}
\nabla^2T_2(dv,dv)=&(\tfrac{1}{2}-\delta')a^{\delta'-\frac{3}{2}}(a(2-a)\partial_av+\partial_{\xi}v)^2\\
&\quad+(1-a)[(\partial_av)^2 -a^{\delta'-\frac{1}{2}}(2\partial_av\partial_{\xi}v+a(2-a)(\partial_av)^2)].
\end{aligned}
$$
Then $\mathcal{Q}_{\tilde{m}}(T_2,v)$ follows. 
\end{proof}

\begin{lemma}\label{lem.5.4}
For $n\geq 4$ and in $\widetilde{\Omega}_2$, 
$$
\begin{gathered}
dvol_{\tilde{g}}=(1+a^{\frac{n-1}{2}}b^{\frac{n-1}{2}}\Theta_1(\tilde{h})) dvol_{\tilde{m}},
\\
\begin{aligned}
\langle  \mathcal{F}_{\tilde{g}}(T_2,v),\nabla T'_2 \rangle_{\tilde{g}} -\langle  \mathcal{F}_{\tilde{m}}(T_2,v),\nabla T'_2 \rangle_{\tilde{m}}
=&\ a^{\frac{n-5}{2}}b^{\frac{n-1}{2}}\big(\Theta_1(\tilde{h}_{\rho\rho})(a\partial_av,\partial_{\xi}v)+a\Theta_1(\tilde{h})(\tilde{\partial}v,v)\big)
\\&\ 
+a^{n-4}b^{n-1}\big(\Theta_2(\tilde{h}_{\rho\rho})(a\partial_av,\partial_{\xi}v)+a\Theta_2(\tilde{h})(\tilde{\partial}v,\tilde{\partial}v)\big),
\\
\langle  \mathcal{F}_{\tilde{g}}(T_2,v),\nabla T_2'' \rangle_{\tilde{g}}-\langle  \mathcal{F}_{\tilde{m}}(T_2,v),\nabla T''_2 \rangle_{\tilde{m}}
=&\ a^{\frac{n-7}{2}}b^{\frac{n-1}{2}}\big(\Theta_1(\tilde{h}_{\rho\rho})(a\partial_av,\partial_{\xi}v)+a\Theta_1(\tilde{h})(\tilde{\partial}v,v)\big)
\\&\ 
+a^{n-5}b^{n-1}\big(\Theta_2(\tilde{h}_{\rho\rho})(a\partial_av,\partial_{\xi}v)+a\Theta_2(\tilde{h})(\tilde{\partial}v,\tilde{\partial}v)\big),
\end{aligned}
\\
\begin{aligned}
\mathcal{Q}_{\tilde{g}} (T_2,v)-\mathcal{Q}_{\tilde{m}} (T_2,v)
= &\ a^{\frac{n-7}{2}}b^{\frac{n-1}{2}}\big(\Theta_1(\tilde{h}_{\rho\rho})(a\partial_av,\partial_{\xi}v)+ 
\Theta_1(\tilde{h}_{\rho\rho})(\tilde{\partial}\tilde{h}_{\rho\rho})(a\partial_av,\partial_{\xi}v)
\\&\
+a\Theta_1(\tilde{h})(\tilde{\partial}v,v)
 +a\Theta_0(\tilde{h})(\tilde{\partial}\tilde{h})(\tilde{\partial}v,v)\big).
\end{aligned}
\end{gathered}
$$
\end{lemma}
\begin{proof}
We can do the similar computation as in the proof of Lemma \ref{lem.5.1} by writing out the metric components in local coordinates $(a,\xi,\theta)$ with $\xi=-\log b$ and $\theta$ is normal w.r.t. the standard spherical metric $d^2\theta$ at $p\in \widetilde{\Omega}_2\cap \mathrm{Int}(X^2)$: 
$$
\begin{gathered}
\tilde{g}=\tilde{m}+a^{\frac{n-5}{2}}b^{\frac{n-1}{2}}
\left[\begin{array}{ccc}
\tilde{h}_{\rho\rho} 
& -a(\tilde{h}_{\rho\rho}-a\tilde{h}_{0\rho})
& a[\Theta_1(\tilde{h})]_{1\times (n-1)}
 \\
 -a(\tilde{h}_{\rho\rho}-a\tilde{h}_{0\rho})
 &  a^{2}(\tilde{h}_{\rho\rho} -
2a\tilde{h}_{0\rho} +a^2\tilde{h}_{00})
&a^{2}[\Theta_1(\tilde{h})]_{1\times (n-1)}
\\ 
a[\Theta_1(\tilde{h})]_{(n-1)\times 1}
&a^{2}[\Theta_1(\tilde{h})]_{(n-1)\times 1}
&a^{2}[\Theta_1(\tilde{h})]_{(n-1)\times (n-1)}
\end{array}\right].
\\
\tilde{g}^{-1}=\tilde{m}^{-1}+ a^{\frac{n-5}{2}}b^{\frac{n-1}{2}}
\left[\begin{array}{ccc}
-a^{2}(\tilde{h}_{\rho\rho} +a\Theta_1(\tilde{h}))
& -a(\tilde{h}_{\rho\rho}+a\Theta_1(\tilde{h}))
& a^{2}[\Theta_1(\tilde{h})]_{1\times (n-1)}
 \\
 -a(\tilde{h}_{\rho\rho}+a\Theta_1(\tilde{h}))
 &  -(\tilde{h}_{\rho\rho} +a^{\frac{n-1}{2}} \Theta_2(\tilde{h}))
&a[\Theta_1(\tilde{h})]_{1\times (n-1)}
\\ 
a^{2}[\Theta_1(\tilde{h})]_{(n-1)\times 1}
&a[\Theta_1(\tilde{h})]_{(n-1)\times 1}
&a^{2}[\Theta_1(\tilde{h})]_{(n-1)\times (n-1)}
\end{array}\right]. 
\end{gathered}
$$
Hence 
$$ 
\begin{aligned}
&\langle \nabla T_2,\nabla T_2\rangle_{\tilde{g}} -\langle \nabla T_2,\nabla T_2\rangle_{\tilde{m}} =-a^{\frac{n-5}{2}} b^{\frac{n-1}{2}}(\tilde{h}_{\rho\rho}(1+a^{\delta'+\frac{1}{2}})^2+a\Theta_1(\tilde{h})),
\\
&\langle \nabla T'_2,\nabla T'_2\rangle_{\tilde{g}} -\langle \nabla T'_2,\nabla T'_2\rangle_{\tilde{m}} 
=-a^{\frac{n-1}{2}} b^{\frac{n-1}{2}}(\tilde{h}_{\rho\rho}+a\Theta_1(\tilde{h})),
\\
&\langle \nabla T_2,\nabla T'_2\rangle_{\tilde{g}} -\langle \nabla T_2,\nabla T'_2\rangle_{\tilde{m}} 
=-a^{\frac{n-3}{2}} b^{\frac{n-1}{2}}(\tilde{h}_{\rho\rho}(1+a^{\delta'+\frac{1}{2}})+a\Theta_1(\tilde{h})), 
\\
&\langle \nabla T_2,\nabla T_2''\rangle_{\tilde{g}}- \langle \nabla T_2,\nabla T_2''\rangle_{\tilde{m}} =-a^{\frac{n-5}{2}} b^{\frac{n-1}{2}}(\tilde{h}_{\rho\rho}(1+a^{\delta'+\frac{1}{2}})+a\Theta_1(\tilde{h})),
\\
&\langle \nabla T_2,\nabla v\rangle_{\tilde{g}} -\langle \nabla T_2,\nabla v\rangle_{\tilde{m}} 
=a^{\frac{n-5}{2}}b^{\frac{n-1}{2}}\big(\tilde{h}_{\rho\rho}(1+a^{\delta'+\frac{1}{2}})(a\partial_av+\partial_{\xi}v)+a\Theta_1(\tilde{h})(\tilde{\partial}v)\big) , 
\\
&\langle \nabla T'_2,\nabla v\rangle_{\tilde{g}} -\langle \nabla T'_2,\nabla v\rangle_{\tilde{m}}=
a^{\frac{n-3}{2}}b^{\frac{n-1}{2}}\big(\tilde{h}_{\rho\rho}(a\partial_av+\partial_{\xi}v)+a\Theta_1(\tilde{h})(\tilde{\partial}v)\big) ,
\\
&\langle \nabla T''_2,\nabla v\rangle_{\tilde{g}} -\langle \nabla T''_2,\nabla v\rangle_{\tilde{m}}=
a^{\frac{n-5}{2}}b^{\frac{n-1}{2}}\big(\tilde{h}_{\rho\rho}(a\partial_av+\partial_{\xi}v)+a\Theta_1(\tilde{h})(\tilde{\partial}v)\big) ,
\\
&\langle \nabla v,\nabla v\rangle_{\tilde{g}} -\langle \nabla v,\nabla v\rangle_{\tilde{m}} =
-a^{\frac{n-5}{2}}b^{\frac{n-1}{2}}\big(\tilde{h}_{\rho\rho}(a\partial_av+\partial_{\xi}v)^2+a\Theta_1(\tilde{h})(\tilde{\partial}v, \tilde{\partial}v)\big). 
\end{aligned}
$$
Then the estimates of $\langle  \mathcal{F}_{\tilde{g}}(T_2,v),\nabla T_2' \rangle_{\tilde{g}}$ and $\langle  \mathcal{F}_{\tilde{g}}(T_2,v),\nabla T_2'' \rangle_{\tilde{g}}$ follow directly. For $\mathcal{Q}_{g}(T_2,v)$, first by the formula of $\Box_{\tilde{g}}$ in the proof of Proposition \ref{prop.4}, we have
$$
\Box_{\tilde{g}}T_2-\Box_{\tilde{m}}T_2 =-a^{\frac{n-5}{2}}b^{\frac{n-1}{2}}\tilde{h}_{\rho\rho}(1+(\tfrac{1}{2}-\delta')a^{\delta'+\frac{1}{2}})+ a^{\frac{n-3}{2}}b^{\frac{n-1}{2}} \Theta_1(\tilde{h}).
$$
Hence
$$
\begin{aligned}
&(\gamma_0v^2-\langle \nabla v,\nabla v\rangle_{\tilde{g}})\Box_{\tilde{g}}T_2
-(\gamma_0v^2-\langle \nabla v,\nabla v\rangle_{\tilde{m}})\Box_{\tilde{m}}T_2
\\
=&\ 
a^{\frac{n-7}{2}}b^{\frac{n-1}{2}}\big(\Theta_1(\tilde{h}_{\rho\rho})(a\partial_av,\partial_{\xi}v)+a\Theta_1(\tilde{h})(\tilde{\partial}v,v)\big).
\end{aligned}
$$
And moreover,
$$
\begin{aligned}
\Gamma_{aa}^a(\tilde{g})-\Gamma_{aa}^a(\tilde{m}) 
=&\  a^{\frac{n-5}{2}}b^{\frac{n-1}{2}}\big(
\Theta_1(\tilde{h}_{\rho\rho})+\Theta_1(\tilde{\partial}\tilde{h}_{\rho\rho})+a\Theta_1(\tilde{h})+a\Theta_0(\tilde{h})(\tilde{\partial}\tilde{h})\big)
\\
&\ +a^{n-4}b^{n-1}\big( \Theta_2(\tilde{h}_{\rho\rho})+\Theta_1(\tilde{h}_{\rho\rho})(\tilde{\partial}\tilde{h}_{\rho\rho})+a\Theta_2(\tilde{h})+a\Theta_1(\tilde{h})(\tilde{\partial}\tilde{h})\big), 
\\
\Gamma_{a\xi}^a(\tilde{g})-\Gamma_{a\xi}^a(\tilde{m}) 
=&\  a^{\frac{n-3}{2}}b^{\frac{n-1}{2}}\big(
\Theta_1(\tilde{h}_{\rho\rho})+\Theta_1(\tilde{\partial}\tilde{h}_{\rho\rho})+a\Theta_1(\tilde{h})+a\Theta_0(\tilde{h})(\tilde{\partial}\tilde{h})\big)
\\
&\ +a^{n-3}b^{n-1}\big( \Theta_2(\tilde{h}_{\rho\rho})+\Theta_1(\tilde{h}_{\rho\rho})(\tilde{\partial}\tilde{h}_{\rho\rho})+a\Theta_2(\tilde{h})+a\Theta_1(\tilde{h})(\tilde{\partial}\tilde{h})\big), 
\\
\Gamma_{\xi\xi}^a(\tilde{g})-\Gamma_{\xi\xi}^a(\tilde{m}) 
=&\  a^{\frac{n-1}{2}}b^{\frac{n-1}{2}}\big(
\Theta_1(\tilde{h}_{\rho\rho})+\Theta_1(\tilde{\partial}\tilde{h}_{\rho\rho})+a\Theta_1(\tilde{h})+a\Theta_0(\tilde{h})(\tilde{\partial}\tilde{h})\big)
\\
&\ +a^{n-2}b^{n-1}\big( \Theta_2(\tilde{h}_{\rho\rho})+\Theta_1(\tilde{h}_{\rho\rho})(\tilde{\partial}\tilde{h}_{\rho\rho})+a\Theta_2(\tilde{h})+a\Theta_1(\tilde{h})(\tilde{\partial}\tilde{h})\big), 
\\
\Gamma_{a\theta}^a(\tilde{g})-\Gamma_{a\theta}^{a}(\tilde{m}) 
=&\  a^{\frac{n-3}{2}}b^{\frac{n-1}{2}}\big(
\Theta_1(\tilde{h})+\Theta_0(\tilde{h})(\tilde{\partial}\tilde{h})\big)
 +a^{n-3}b^{n-1}\big(\Theta_2(\tilde{h})+\Theta_1(\tilde{h})(\tilde{\partial}\tilde{h})\big), 
\\
\Gamma_{\xi\theta}^{a}(\tilde{g})-\Gamma_{\xi\theta}^{a}(\tilde{m}) 
=&\  a^{\frac{n-1}{2}}b^{\frac{n-1}{2}}\big(
\Theta_1(\tilde{h})+\Theta_0(\tilde{h})(\tilde{\partial}\tilde{h})\big)
 +a^{n-2}b^{n-1}\big(\Theta_2(\tilde{h})+\Theta_1(\tilde{h})(\tilde{\partial}\tilde{h})\big), 
\\
\Gamma_{\theta\theta}^{a}(\tilde{g})-\Gamma_{\theta\theta}^{a}(\tilde{m}) 
=&\  a^{\frac{n-1}{2}}b^{\frac{n-1}{2}}\big(
\Theta_1(\tilde{h})+\Theta_0(\tilde{h})(\tilde{\partial}\tilde{h})\big)
 +a^{n-2}b^{n-1}\big(\Theta_2(\tilde{h})+\Theta_1(\tilde{h})(\tilde{\partial}\tilde{h})\big), 
\\
\Gamma_{aa}^{\xi}(\tilde{g})-\Gamma_{aa}^{\xi}(\tilde{m}) 
=&\  a^{\frac{n-7}{2}}b^{\frac{n-1}{2}}\big(
\Theta_1(\tilde{h}_{\rho\rho})+\Theta_1(\tilde{\partial}\tilde{h}_{\rho\rho})+a\Theta_1(\tilde{h})+a\Theta_0(\tilde{h})(\tilde{\partial}\tilde{h})\big)
\\
&\ +a^{n-5}b^{n-1}\big( \Theta_2(\tilde{h}_{\rho\rho})+\Theta_1(\tilde{h}_{\rho\rho})(\tilde{\partial}\tilde{h}_{\rho\rho})+a\Theta_2(\tilde{h})+a\Theta_1(\tilde{h})(\tilde{\partial}\tilde{h})\big), 
\\
\Gamma_{a\xi}^{\xi}(\tilde{g})-\Gamma_{a\xi}^{\xi}(\tilde{m}) 
=&\  a^{\frac{n-5}{2}}b^{\frac{n-1}{2}}\big(
\Theta_1(\tilde{h}_{\rho\rho})+\Theta_1(\tilde{\partial}\tilde{h}_{\rho\rho})+a\Theta_1(\tilde{h})+a\Theta_0(\tilde{h})(\tilde{\partial}\tilde{h})\big)
\\
&\ +a^{n-4}b^{n-1}\big( \Theta_2(\tilde{h}_{\rho\rho})+\Theta_1(\tilde{h}_{\rho\rho})(\tilde{\partial}\tilde{h}_{\rho\rho})+a\Theta_2(\tilde{h})+a\Theta_1(\tilde{h})(\tilde{\partial}\tilde{h})\big), 
\\
\Gamma_{\xi\xi}^{\xi}(\tilde{g})-\Gamma_{\xi\xi}^{\xi}(\tilde{m}) 
=&\  a^{\frac{n-3}{2}}b^{\frac{n-1}{2}}\big(
\Theta_1(\tilde{h}_{\rho\rho})+\Theta_1(\tilde{\partial}\tilde{h}_{\rho\rho})+a\Theta_1(\tilde{h})+a\Theta_0(\tilde{h})(\tilde{\partial}\tilde{h})\big)
\\
&\ +a^{n-3}b^{n-1}\big( \Theta_2(\tilde{h}_{\rho\rho})+\Theta_1(\tilde{h}_{\rho\rho})(\tilde{\partial}\tilde{h}_{\rho\rho})+a\Theta_2(\tilde{h})+a\Theta_1(\tilde{h})(\tilde{\partial}\tilde{h})\big), 
\\
\Gamma_{a\theta}^{\xi}(\tilde{g})-\Gamma_{a\theta}^{\xi}(\tilde{m}) 
=&\  a^{\frac{n-5}{2}}b^{\frac{n-1}{2}}\big(
\Theta_1(\tilde{h})+\Theta_0(\tilde{h})(\tilde{\partial}\tilde{h})\big)
 +a^{n-4}b^{n-1}\big(\Theta_2(\tilde{h})+\Theta_1(\tilde{h})(\tilde{\partial}\tilde{h})\big), 
\\
\Gamma_{\xi\theta}^{\xi}(\tilde{g})-\Gamma_{\xi\theta}^{\xi}(\tilde{m}) 
=&\  a^{\frac{n-3}{2}}b^{\frac{n-1}{2}}\big(
\Theta_1(\tilde{h})+\Theta_0(\tilde{h})(\tilde{\partial}\tilde{h})\big)
 +a^{n-3}b^{n-1}\big(\Theta_2(\tilde{h})+\Theta_1(\tilde{h})(\tilde{\partial}\tilde{h})\big), 
\\
\Gamma_{\theta\theta}^{\xi}(\tilde{g})-\Gamma_{\theta\theta}^{\xi}(\tilde{m}) 
=&\  a^{\frac{n-3}{2}}b^{\frac{n-1}{2}}\big(
\Theta_1(\tilde{h})+\Theta_0(\tilde{h})(\tilde{\partial}\tilde{h})\big)
 +a^{n-3}b^{n-1}\big(\Theta_2(\tilde{h})+\Theta_1(\tilde{h})(\tilde{\partial}\tilde{h})\big).
\end{aligned}
$$
This implies that
$$
\begin{aligned}
&\nabla^2T_2(dv,dv)_{\tilde{g}} - \nabla^2T_2(dv,dv)_{\tilde{m}} 
\\
=&\ a^{\frac{n-7}{2}}b^{\frac{n-1}{2}}\big((\Theta_1(\tilde{h}_{\rho\rho})+ 
\Theta_1(\tilde{h}_{\rho\rho})(\tilde{\partial}\tilde{h}_{\rho\rho}))(a\partial_av,\partial_{\xi}v)+a(\Theta_1(\tilde{h})+\Theta_0(\tilde{h})(\tilde{\partial}\tilde{h}))(\tilde{\partial}v,\tilde{\partial}v)\big)
\\ &\ 
+a^{n-5}b^{n-1}\big((\Theta_2(\tilde{h}_{\rho\rho})+ 
\Theta_1(\tilde{h}_{\rho\rho})(\tilde{\partial}\tilde{h}_{\rho\rho}))(a\partial_av,\partial_{\xi}v)+a(\Theta_2(\tilde{h})+\Theta_1(\tilde{h})(\tilde{\partial}\tilde{h}))(\tilde{\partial}v,\tilde{\partial}v)\big).
\end{aligned}
$$
We prove the last formula. 
\end{proof}

\subsection{In $\widetilde{\Omega}_3$.}
In $\Omega_3\cup \Omega_0$, define for $\tau_0\geq 8$
\begin{equation*}
\begin{gathered}
T_3=t-\psi_2(r), \quad
T_3'=\frac{t-2\tau_0}{\psi_3(r)}-1;
\\
\widetilde{\Omega}_3=\overline{\{ T_3\leq \tau_0\}\backslash (\widetilde{\Omega}_1\cup\widetilde{\Omega}_2)},\quad
\Sigma_3=\{T_3'=-1\}\cap \widetilde{\Omega}_3,
\quad \Sigma_4=\{T_3=\tau_0\},
\end{gathered}
\end{equation*}
where $\psi_2,\psi_3\in C^{2}(\widetilde{\Omega}_3)$ are defined by
$$
\begin{aligned}
&\psi_2(r)=\ \begin{cases} r+ \frac{2}{2\delta'+1}r^{-\delta'-\frac{1}{2}} &\quad \textrm{if $r>1$}\\
1+\frac{2}{2\delta'+1}+(\delta'+\tfrac{3}{2})(\tfrac{1}{4}-\tfrac{1}{4}r^2+\tfrac{1}{8}r^4) &\quad \textrm{if $r\leq 1$}\end{cases},
\\
&\psi_3(r)=\ \begin{cases} r &\quad \textrm{if $r>1$}\\
\tfrac{3}{8}+\tfrac{3}{4}r^2-\tfrac{1}{8}r^4&\quad \textrm{if $r\leq 1$}\end{cases}.
\end{aligned}
$$
Hence $-1\leq\psi_2'\leq 1, 0\leq\psi_3'\leq 1$. 
Then in $\widetilde{\Omega}_3$ 
$$
\begin{aligned}
\langle\nabla T_3,\nabla T_3\rangle_{\tilde{m}} =&\ 
\tilde{\rho}^{-2}(-1+(\psi'_2(r))^2)<0.
\end{aligned}
$$
We divide $\widetilde{\Omega}_3$ into two parts: 
$$
\widetilde{\Omega}'_3 =\widetilde{\Omega}_3\cap \{-1\leq T'_3\leq 0\}, \quad
 \widetilde{\Omega}''_3 =\widetilde{\Omega}_3\cap \{T'_3\leq -1\}.
$$
where $\widetilde{\Omega}''_3\subset \mathring{X}$ is compact and $\widetilde{\Omega}'_3\subset \Omega_3$.  Recall that in $\Omega_3$ we use the coordinates 
$$
(\tau,\rho,\theta)=(t-r, \frac{1}{r}, \frac{x}{r}), \quad \tilde{\rho}=\rho_1=\rho,\  \rho_0=\rho_2=1.
$$
Hence in $\widetilde{\Omega}'_3$, 
$$
\begin{gathered}
T_3=\tau- \tfrac{2}{2\delta'+1}\rho^{\delta'+\frac{1}{2}}, \quad T_3'= -\rho(2\tau_0-\tau);
\\
\begin{aligned}
&\langle\nabla T_3,\nabla T_3\rangle_{\tilde{m}} =-\rho^{\delta'-\frac{1}{2}}(2-\rho^{\delta'+\frac{3}{2}}),
\\
&\langle\nabla T'_3,\nabla T'_3\rangle_{\tilde{m}} =-\rho(2\tau_0-\tau)(2-\rho(2\tau_0-\tau)).
\end{aligned}
\end{gathered}
$$
Here $T_3$ is time-like functions all over $\widetilde{\Omega}_3 \cap \mathring{X}$; $T_3'$ is time-like in the interior of $\widetilde{\Omega}_3'$ and null on $\widetilde{\Omega}_3\cap S_1$. 
With $\gamma_0<0$ a constant, we have the following two lemmas. 
\begin{lemma}\label{lem.5.5}
In $ \widetilde{\Omega}'_3$,
$$
\begin{gathered}
\begin{aligned}
\langle  \mathcal{F}_{\tilde{m}}(T_3,v),\nabla T_3\rangle_{\tilde{m}}  =&\  
\tfrac{1}{2}\rho^{2\delta'-1}(|\partial_{\tau}v|^2+|\partial_{\tau}v+\rho^2\partial_{\rho}v|^2) +(1-\rho^{\delta'+\frac{3}{2}})|\partial_{\rho}v|^2  
\\&\ 
+\rho^{\delta'-\frac{1}{2}}(1-\tfrac{1}{2}\rho^{\delta'+\frac{3}{2}})(|\partial_{\theta}v|^2 -\gamma_0 |v|^2),
\\
\langle  \mathcal{F}_{\tilde{m}}(T_3,v),\nabla T'_3\rangle_{\tilde{m}}  =&\  
\tfrac{1}{2}\rho^{\delta'-\frac{1}{2}}(2\tau_0-\tau)(|\partial_{\tau}v|^2+|\partial_{\tau}v+\rho^2\partial_{\rho}v|^2) 
\\ &\
+(1-\tfrac{1}{2}\rho(2\tau_0-\tau)-\tfrac{1}{2}\rho^{\delta'+\frac{3}{2}})\rho|\partial_{\rho}v|^2
\\ &\ 
+\tfrac{1}{2}\big(2\tau_0-\tau+\rho^{\delta'+\frac{1}{2}}(1-\rho(2\tau_0-\tau))\big)(|\partial_{\theta}v|^2 -\gamma_0 |v|^2),
 \\
 \mathcal{Q}_{\tilde{m}}(T_3,v) =&\
 \tfrac{1}{2}(\tfrac{1}{2}-  \delta')\rho^{\delta'-\frac{3}{2}}(|\partial_{\tau}v|^2+|\partial_{\tau}v+\rho^2\partial_{\rho}v|^2) 
 \\&\ 
 -\rho|\partial_{\rho}v|^2 +\tfrac{1}{2}(\delta'+\tfrac{3}{2})\rho^{\delta'+\tfrac{1}{2}}(|\partial_{\theta}v|^2 -\gamma_0 |v|^2).
  \end{aligned}
\end{gathered}
$$
\end{lemma}
\begin{proof}
In $\widetilde{\Omega}'_3$ under coordinates $(\tau,\rho,\theta)$ with $\theta$ normal w.r.t. standard spherical metric at $p\in\widetilde{\Omega}'_3 $, 
we can write the metric components of $\tilde{m}=\rho^2m$ as follows: 
$$
\tilde{m}=\left[\begin{array}{ccc} -\rho^2 &1 &0 \\1&0&0\\0&0&\mathbbm{1}_{n-1}\end{array}\right], \quad
\tilde{m}^{-1}=\left[\begin{array}{ccc} 0 &1 &0 \\1&\rho^2&0\\0&0&\mathbbm{1}_{n-1}\end{array}\right].
$$
Then $\langle  \mathcal{F}_{\tilde{m}}(T_3,v),\nabla T_3\rangle_{\tilde{m}}$ and $\langle  \mathcal{F}_{\tilde{m}}(T_3,v),\nabla T'_3\rangle_{\tilde{m}}$ follow from
$$
\begin{aligned}
&\langle \nabla T_3,\nabla T'_3\rangle_{\tilde{m}} = -(2\tau_0-\tau)-\rho^{\delta'+\frac{1}{2}}(1-\rho(2\tau_0-\tau)),
\\
&\langle \nabla T_3,\nabla v\rangle_{\tilde{m}} = \partial_{\rho}v -\rho^{\delta'-\frac{1}{2}}(\partial_{\tau}v+\rho^2\partial_{\rho}v),\\
&\langle \nabla T'_3,\nabla v\rangle_{\tilde{m}}=\rho\partial_{\rho}v-(2\tau_0-\tau)(\partial_{\tau}v+\rho^2\partial_{\rho}v),\\
&\langle \nabla v,\nabla v\rangle_{\tilde{m}} =2\partial_{\tau}v\partial_{\rho}v+|\rho\partial_{\rho}v|^2+|\slashpar_{\theta}v|^2.
\end{aligned}
$$
For $\mathcal{Q}_{\tilde{m}}(T_3,v)$, we have
$$
\Box_{\tilde{m}} T_3=-(\tfrac{3}{2}+\delta')\rho^{\delta'+\frac{1}{2}}, 
$$
and the non-zero connection components at $p$ are
$$
\Gamma^{\tau}_{\tau\tau}(\tilde{m}) =\rho, \quad
\Gamma^{\rho}_{\rho\tau}(\tilde{m}) =-\rho,\quad \Gamma^{\rho}_{\tau\tau}(\tilde{m}) =\rho^3,
$$
which implies that
$$
\nabla^2T(dv,dv)=(\tfrac{1}{2}-\delta')\rho^{\delta'-\frac{3}{2}}|\partial_{\tau}v+\rho^2\partial_{\rho}v|^2-\rho|\partial_{\rho}v|^2-\rho^{\delta'+\frac{1}{2}}(2\partial_{\tau}v\partial_{\rho}v+\rho^2|\partial_{\rho}v|^2).
$$
Then $\mathcal{Q}_{\tilde{m}}(T_3,v)$ follows. 
\end{proof}

\begin{lemma}\label{lem.5.6}
For $n\geq 4$ and in $\widetilde{\Omega}'_3$, 
$$
\begin{gathered}
dvol_{\tilde{g}}=(1+\rho^{\frac{n-1}{2}}\Theta_1(\tilde{h})) dvol_{\tilde{m}},
\\
\begin{aligned}
&\langle \mathcal{F}_{\tilde{g}}(T_3,v),\nabla T'_3\rangle_{\tilde{g}} -\langle \mathcal{F}_{\tilde{m}}(T_3,v),\nabla T'_3\rangle_{\tilde{m}} =
\rho^{\frac{n-5}{2}}\big(\Theta_1(\tilde{h}_{\rho\rho})(\partial_{\tau}v,\rho\partial_{\rho}v)+
\rho\Theta_1(\tilde{h})(\tilde{\partial}v,v)\big),
\\
& \mathcal{Q}_{\tilde{g}}(T_3,v)- \mathcal{Q}_{\tilde{m}}(T_3,v)=
 \rho^{\frac{n-7}{2}}\big( (\Theta_1(\tilde{h}_{\rho\rho})+\Theta_1(\tilde{\partial}\tilde{h}_{\rho\rho}))(\partial_{\tau}v,\rho\partial_{\rho}v)+\rho(\Theta_1(\tilde{h})+\Theta_0(\tilde{h})(\tilde{\partial}\tilde{h}))(\tilde{\partial}v,v) \big) .
\end{aligned}
\end{gathered}
$$
\end{lemma}

\begin{proof}
We can do the similar computation as in the proof of Lemma \ref{lem.5.1} by writing out the metric components in local coordinates $(\tau,\rho,\theta)$ with $\theta$ normal w.r.t. the standard spherical metric at $p\in \widetilde{\Omega}'_3$: 
$$
\begin{gathered}
\tilde{g}=\tilde{m}+\rho^{\frac{n-5}{2}}
\left[\begin{array}{ccc}
\rho^4\Theta_1(\tilde{h}) & \rho^2\Theta_1(\tilde{h}) &\rho^3[\Theta_1(\tilde{h})]_{1\times(n-1)}\\
\rho^2\Theta_1(\tilde{h}) & \tilde{h}_{\rho\rho} &\rho[\Theta_1(\tilde{h})]_{1\times (n-1)} \\
\rho^3[\Theta_1(\tilde{h})]_{(n-1)\times 1} & \rho[\Theta_1(\tilde{h})]_{(n-1)\times (n-1)} & \rho^2[\Theta_1(\tilde{h})]_{(n-1)\times (n-1)}
\end{array}\right].
\\
\tilde{g}^{-1}=\tilde{m}^{-1}+\rho^{\frac{n-5}{2}}
\left[\begin{array}{ccc}
-\tilde{h}_{\rho\rho}+\rho^{\frac{n-1}{2}}\Theta_2(\tilde{h}) & \rho^2\Theta_1(\tilde{h}) &\rho[\Theta_1(\tilde{h})]_{1\times(n-1)}\\
\rho^2\Theta_1(\tilde{h}) &\rho^4\Theta_1(\tilde{h}) &\rho^3[\Theta_1(\tilde{h})]_{1\times (n-1)} \\
\rho[\Theta_1(\tilde{h})]_{(n-1)\times 1} & \rho^3[\Theta_1(\tilde{h})]_{(n-1)\times (n-1)} & \rho^2[\Theta_1(\tilde{h})]_{(n-1)\times (n-1)}
\end{array}\right]. 
\end{gathered}
$$
Hence $\langle  \mathcal{F}_{\tilde{g}}(T_3,v),\nabla T'_3\rangle_{\tilde{g}}$ can be estimated by the following
$$
\begin{aligned}
&\langle \nabla T_3,\nabla T'_3\rangle_{\tilde{g}} -\langle \nabla T_3, \nabla T'_3\rangle_{\tilde{m}} 
=-\rho^{\frac{n-3}{2}}\tilde{h}_{\rho\rho}+\rho^{\frac{n-1}{2}}\Theta_1(\tilde{h}),
\\
&\langle \nabla T_3,\nabla v\rangle_{\tilde{g}} -\langle \nabla T_3, \nabla v\rangle_{\tilde{m}} 
=-\rho^{\frac{n-5}{2}}\tilde{h}_{\rho\rho}\partial_{\tau}v+\rho^{\frac{n-3}{2}}\Theta_1(\tilde{h})(\tilde{\partial}v),
\\
&\langle \nabla T'_3,\nabla v\rangle_{\tilde{g}} -\langle \nabla T'_3, \nabla v\rangle_{\tilde{m}} 
=-\rho^{\frac{n-3}{2}}\tilde{h}_{\rho\rho}\partial_{\tau}v+\rho^{\frac{n-1}{2}}\Theta_1(\tilde{h})(\tilde{\partial}v),
\\
&\langle \nabla v,\nabla v\rangle_{\tilde{g}} -\langle \nabla T_3, \nabla v\rangle_{\tilde{m}} 
=-\rho^{\frac{n-5}{2}}\tilde{h}_{\rho\rho}|\partial_{\tau}v|^2+\rho^{\frac{n-3}{2}}\Theta_1(\tilde{h})(\tilde{\partial}v,\tilde{\partial}v),
\end{aligned}
$$
For $\mathcal{Q}_{\tilde{g}}(T_3,v)-\mathcal{Q}_{\tilde{m}}(T_3,v)$, we first have
$$
\Box_{\tilde{g}}T_3-\Box_{\tilde{m}}T_3 = \rho^{\frac{n-3}{2}}\Theta_1(\tilde{h}).
$$
Hence
$$
\begin{aligned}
&(\gamma_0v^2-\langle \nabla v,\nabla v\rangle_{\tilde{g}})\Box_{\tilde{g}}T_3 -
(\gamma_0v^2-\langle \nabla v,\nabla v\rangle_{\tilde{m}})\Box_{\tilde{m}}T_3
\\
=&\ 
\rho^{\frac{n-3}{2}}\Theta_1(\tilde{h})(\partial_{\tau}v)(\partial_{\rho}v)+ \rho^{\frac{n-4}{2}+\delta}\Theta_1(\tilde{h})(\tilde{\partial}v, v)+ \rho^{n-4}\Theta_2(\tilde{h})(\tilde{\partial}v, v). 
\end{aligned}
$$
Moreover, 
$$
\begin{aligned}
&\Gamma^{\rho}_{\rho\rho}(\tilde{g})-\Gamma^{\rho}_{\rho\rho}(\tilde{m}) = 
\rho^{\frac{n-5}{2}}\big(\Theta_1(\tilde{h}_{\rho\rho})+\Theta_1(\tilde{\partial}\tilde{h}_{\rho\rho})+\rho\Theta_1(\tilde{h})+\rho\Theta_0(\tilde{h})(\tilde{\partial}\tilde{h})  \big),
\\
&\Gamma^{\rho}_{\rho\tau}(\tilde{g})-\Gamma^{\tau}_{\rho\tau}(\tilde{m})=
\rho^{\frac{n-1}{2}}\big(\Theta_1(\tilde{h}_{\rho\rho})+\Theta_1(\tilde{\partial}\tilde{h}_{\rho\rho})+\rho\Theta_1(\tilde{h})+\rho\Theta_0(\tilde{h})(\tilde{\partial}\tilde{h})  \big),
\\
&\Gamma^{\rho}_{\tau\tau}(\tilde{g})-\Gamma^{\tau}_{\tau\tau}(\tilde{m})=
\rho^{\frac{n+3}{2}}\big(\Theta_1(\tilde{h})+\Theta_0(\tilde{h})(\tilde{\partial}\tilde{h})  \big),
\\
&\Gamma^{\rho}_{\rho\theta}(\tilde{g})-\Gamma^{\tau}_{\rho\theta}(\tilde{m})=
\rho^{\frac{n-3}{2}}\big(\Theta_1(\tilde{h})+\Theta_0(\tilde{h})(\tilde{\partial}\tilde{h})  \big),
\\
&\Gamma^{\rho}_{\tau\theta}(\tilde{g})-\Gamma^{\tau}_{\tau\theta}(\tilde{m})=
\rho^{\frac{n+1}{2}}\big(\Theta_1(\tilde{h})+\Theta_0(\tilde{h})(\tilde{\partial}\tilde{h})  \big),
\\
&\Gamma^{\rho}_{\theta\theta}(\tilde{g})-\Gamma^{\tau}_{\theta\theta}(\tilde{m})=
\rho^{\frac{n-1}{2}}\big(\Theta_1(\tilde{h})+\Theta_0(\tilde{h})(\tilde{\partial}\tilde{h})  \big),
\\
&\Gamma^{\tau}_{\rho\rho}(\tilde{g})-\Gamma^{\tau}_{\rho\rho}(\tilde{m}) = 
\rho^{\frac{n-7}{2}}\big(\Theta_1(\tilde{h}_{\rho\rho})+\Theta_1(\tilde{\partial}\tilde{h}_{\rho\rho})+\rho\Theta_1(\tilde{h})+\rho\Theta_0(\tilde{h})(\tilde{\partial}\tilde{h})  \big),
\\
&\Gamma^{\tau}_{\rho\tau}(\tilde{g})-\Gamma^{\tau}_{\rho\tau}(\tilde{m})=
\rho^{\frac{n-5}{2}}\big(\Theta_1(\tilde{h}_{\rho\rho})+\Theta_1(\tilde{\partial}\tilde{h}_{\rho\rho})+\rho\Theta_1(\tilde{h})+\rho\Theta_0(\tilde{h})(\tilde{\partial}\tilde{h})  \big),
\\
&\Gamma^{\tau}_{\tau\tau}(\tilde{g})-\Gamma^{\tau}_{\tau\tau}(\tilde{m})=
\rho^{\frac{n-1}{2}}\big(\Theta_1(\tilde{h})+\Theta_0(\tilde{h})(\tilde{\partial}\tilde{h})  \big),
\\
&\Gamma^{\tau}_{\rho\theta}(\tilde{g})-\Gamma^{\tau}_{\rho\theta}(\tilde{m})=
\rho^{\frac{n-5}{2}}\big(\Theta_1(\tilde{h})+\Theta_0(\tilde{h})(\tilde{\partial}\tilde{h})  \big),
\\
&\Gamma^{\tau}_{\tau\theta}(\tilde{g})-\Gamma^{\tau}_{\tau\theta}(\tilde{m})=
\rho^{\frac{n-3}{2}}\big(\Theta_1(\tilde{h})+\Theta_0(\tilde{h})(\tilde{\partial}\tilde{h})  \big),
\\
&\Gamma^{\tau}_{\theta\theta}(\tilde{g})-\Gamma^{\tau}_{\theta\theta}(\tilde{m})=
\rho^{\frac{n-3}{2}}\big(\Theta_1(\tilde{h})+\Theta_0(\tilde{h})(\tilde{\partial}\tilde{h})  \big),
\end{aligned}
$$
Hence 
$$
\begin{aligned}
&\nabla^2T_3(dv,dv)_{\tilde{g}} - \nabla^2T_3(dv,dv)_{\tilde{m}} 
\\ 
=&\ 
\rho^{\frac{n-7}{2}}\big( (\Theta_1(\tilde{h}_{\rho\rho})+\Theta_1(\tilde{\partial}\tilde{h}_{\rho\rho}))(\partial_{\tau}v,\rho\partial_{\rho}v)+\rho(\Theta_1(\tilde{h})+\Theta_0(\tilde{h})(\tilde{\partial}\tilde{h}))(\tilde{\partial}v,v) \big) 
\\ &\ 
+ \rho^{n-5}\big( (\Theta_2(\tilde{h}_{\rho\rho})+\Theta_1(\tilde{h}_{\rho\rho})(\tilde{\partial}\tilde{h}_{\rho\rho}))(\partial_{\tau}v,\rho\partial_{\rho}v)+\rho(\Theta_2(\tilde{h})+\Theta_1(\tilde{h})(\tilde{\partial}\tilde{h}))(\tilde{\partial}v,v) \big).
\end{aligned}
$$
Then $\mathcal{Q}_{\tilde{g}}(T_3,v)-\mathcal{Q}_{\tilde{m}}(T_3,v)$ follows. 
\end{proof}

\subsection{In $\widetilde{\Omega}_4$.}
In $\Omega_4$,  recall that we have the coordinates and boundary defining functions
$$
(\bar{a},\bar{b},\theta)=(1-\frac{r}{t}, \frac{1}{t-r},\frac{x}{r});\quad \rho_0=1,\ \rho_1=\bar{a},\ \rho_2=\bar{b},\ \tilde{\rho}=\bar{a}\bar{b};
$$
and in $\Omega_5$, we have
$$
(\phi,y)=(\frac{1}{t}, \frac{x}{t}),\ R=|y|; \quad \rho_0=\rho_1=1, \rho_2=\tilde{\rho}=\phi.
$$
Both sets of coordinates can be extended to $\Omega_4\cup\Omega_5$ with the transition map
$$
\phi=\bar{a}\bar{b}, \ y=(1-\bar{a})\theta.
$$
Define
$$
\begin{gathered}
\begin{aligned}
T_4=&\ -\tfrac{2}{2\delta'+1}(\tfrac{\bar{a}(2-\bar{a})}{2})^{\delta'+\frac{1}{2}}
+\ln(2-\bar{a})-\ln \bar{b}\\
=&\  -\tfrac{2}{2\delta'+1}(\tfrac{1-|y|^2}{2})^{\delta'+\frac{1}{2}}+\ln(1-|y|^2)-\ln \phi,
\\
T'_4=&\ -(\tfrac{2-\bar{a}}{2})^{1-\alpha}\bar{a}\bar{b}^{\alpha}=-(\tfrac{1-|y|^2}{2})^{1-\alpha}\phi^{\alpha},
\\
\end{aligned}
\end{gathered}
$$
and let 
$$
\widetilde{\Omega}_4=
\overline{(\Omega_4\cup\Omega_5)\backslash\widetilde{\Omega}_3}.
$$
Then in  $\widetilde{\Omega}_4\cap \mathring{X}$, 
$$
\begin{aligned}
&\langle \nabla T_4,\nabla T_4\rangle_{\tilde{m}}
= -1-2(1-\bar{a})^2(\tfrac{\bar{a}(2-\bar{a})}{2})^{\delta'-\frac{1}{2}}[1-(\tfrac{\bar{a}(2-\bar{a}}{2}))^{\delta'+\frac{1}{2}}]<0,
\\
&\langle \nabla T'_4,\nabla T'_4\rangle_{\tilde{m}} =
-2\bar{a}(\tfrac{2-\bar{a}}{2})^{2-2\alpha}\bar{b}^{2\alpha}
(\tfrac{2-\bar{a}}{2}\alpha-1+\bar{a})(1+\tfrac{\alpha-1}{2-\bar{a}}\bar{a})< 0, 
\\
&\langle \nabla T'_4,\nabla T_4\rangle_{\tilde{m}} =
-(\tfrac{2-\bar{a}}{2})^{1-\alpha}\bar{b}^{\alpha}\big(
\alpha\bar{a}+\tfrac{2(1-\bar{a})^2}{2-\bar{a}}(1-
(2-\alpha)(\tfrac{\bar{a}(2-\bar{a})}{2})^{\delta'+\frac{1}{2}})
\big)<0.
\end{aligned}
$$
Hence $T_4, T_4'$ are  time-like in $\widetilde{\Omega}_4\cap \mathring{X}$ and $T'_4$ is null on $S_1^+\cap \widetilde{\Omega}_4$. 
With $\gamma_0<0$ a constant, we have the following two lemmas.
\begin{lemma}\label{lem.5.7}
In $\widetilde{\Omega}_4$, 
$$
\begin{aligned}
\langle \mathcal{F}_{\tilde{m}}(T_4,v),\nabla T'_4\rangle_{\tilde{m}} 
=&\   (\tfrac{2-\bar{a}}{2})^{1-\alpha}\bar{b}^{\alpha}\big(
C_1
(|\bar{a}(2-\bar{a})\partial_{\bar{a}}v-\bar{b}\partial_{\bar{b}}v|^2+|\bar{b}\partial_{\bar{b}}v|^2)
+C_2|\partial_{\bar{a}}v|^2
\\ &\ 
+C_3(\tfrac{|\slashpar_{\theta}v|^2}{(1-\bar{a})^2}-\gamma_0v^2)\big),
\\
\bar{a}\mathcal{Q}_{\tilde{m}}(T_4,v)=&\ 
D_1
(|\bar{a}(2-\bar{a})\partial_{\bar{a}}v-\bar{b}\partial_{\bar{b}}v|^2+ |\bar{b}\partial_{\bar{b}}v|^2)
+D_2|\partial_{\bar{a}}v|^2
\\&\ 
+D_3(|\bar{a}(2-\bar{a})\partial_{\bar{a}}v-\bar{b}\partial_{\bar{b}}v|^2-| \bar{b}\partial_{\bar{b}}v|^2)
+D_4\tfrac{|\slashpar_{\theta}v|^2}{(1-\bar{a})^2}+D_5(-\gamma_0v^2),
\end{aligned}
$$
where
$$
\begin{aligned}
&C_1=\tfrac{(1-\bar{a})^2}{2-\bar{a}}(\tfrac{\bar{a}(2-\bar{a})}{2})^{\delta'-\frac{1}{2}}
+\tfrac{1-\bar{a}}{(2-\bar{a})^2}
+ \tfrac{\alpha}{(2-\bar{a})^2}(\tfrac{\bar{a}(2-\bar{a})}{2})^{\delta'+\frac{1}{2}}
+\tfrac{\alpha\bar{a}}{(2-\bar{a})^2}(\tfrac{1}{2}-(\tfrac{\bar{a}(2-\bar{a})}{2})^{\delta'+\frac{1}{2}}),
\\
&C_2=\alpha\bar{a}(1-\bar{a})(1-(\tfrac{\bar{a}(2-\bar{a})}{2})^{\delta'+\frac{1}{2}})-\bar{a}(1-\bar{a}),
\\
&C_3=\tfrac{\alpha\bar{a}}{2} +\tfrac{(1-\bar{a})^2}{2-\bar{a}}+ \tfrac{(\alpha-2)(1-\bar{a})^2}{2-\bar{a}} (\tfrac{\bar{a}(2-\bar{a})}{2})^{\delta'+\frac{1}{2}},
\\
&D_1=(\tfrac{1}{2}-\delta')\tfrac{(1-\bar{a})^2}{2-\bar{a}}(\tfrac{\bar{a}(2-\bar{a})}{2})^{\delta'-\frac{1}{2}}-\tfrac{\bar{a}}{2(2-\bar{a})^2}+\tfrac{1}{2-\bar{a}}(\tfrac{\bar{a}(2-\bar{a})}{2})^{\delta'+\frac{1}{2}},
\\
&D_2=-\bar{a}(1-\bar{a}),
\\
&D_3=\tfrac{n-1}{2-\bar{a}}\big(\tfrac{1}{2}-(\tfrac{\bar{a}(2-\bar{a})}{2})^{\delta'+\frac{1}{2}}\big),
\\
&D_4=(1+2\delta')\tfrac{(1-\bar{a})^2}{2-\bar{a}}(\tfrac{\bar{a}(2-\bar{a})}{2})^{\delta'+\frac{1}{2}}
+(n-2)\bar{a}\big(\tfrac{1}{2}-(\tfrac{\bar{a}(2-\bar{a})}{2})^{\delta'+\frac{1}{2}}\big),
\\
&D_5=(1+2\delta')\tfrac{(1-\bar{a})^2}{2-\bar{a}}(\tfrac{\bar{a}(2-\bar{a})}{2})^{\delta'+\frac{1}{2}}
+n\bar{a}\big(\tfrac{1}{2}-(\tfrac{\bar{a}(2-\bar{a})}{2})^{\delta'+\frac{1}{2}}\big).
\end{aligned}
$$
\end{lemma}
\begin{proof}
Under coordinates $(\bar{a},\xi,\theta)$ with $\xi=-\log \bar{b}$ and $\theta $ normal w.r.t. the standard spherical metric $d^2\theta$ at $p\in \widetilde{\Omega}_4\cap \mathrm{Int}(X^2)$, the metric components are:
$$
\tilde{m}=
\left[\begin{array}{ccc}
0&1&0\\1& -\bar{a}(2-\bar{a}) &0\\ 0&0& (1-\bar{a})^2\mathbbm{1}_{n-1}
\end{array}\right],\quad
\tilde{m}^{-1}=
\left[\begin{array}{ccc}
\bar{a}(2-\bar{a})&1&0\\1& 0 &0\\ 0&0& \frac{1}{(1-\bar{a})^{2}}\mathbbm{1}_{n-1}
\end{array}\right].
$$
Then $\langle \mathcal{F}_{\tilde{m}}(T_4,v),\nabla T'_4\rangle_{\tilde{m}}$  follows from
$$
\begin{aligned}
\langle \nabla T_4,\nabla T'_4\rangle_{\tilde{m}} =&\ 
-\bar{b}^{\alpha}(\tfrac{2-\bar{a}}{2})^{-\alpha}
\big(1-(2-\alpha)[\tfrac{\bar{a}(2-\bar{a})}{2}+(1-\bar{a})^2(\tfrac{\bar{a}(2-\bar{a})}{2})^{\delta'+\frac{1}{2}}]\big),
\\
\langle \nabla T_4,\nabla T''_4\rangle_{\tilde{m}} =&\ 
-1-(1-\bar{a})^2(\tfrac{\bar{a}(2-\bar{a})}{2})^{\delta'-\frac{1}{2}},
\\
\langle \nabla T_4,\nabla v\rangle_{\tilde{m}} =&\ 
-((\tfrac{\bar{a}(2-\bar{a})}{2})^{\delta'-\frac{1}{2}}+\tfrac{1}{2-\bar{a}})\partial_{\xi}v+
(1-\bar{a}-2(\tfrac{\bar{a}(2-\bar{a})}{2})^{\delta'+\frac{1}{2}})\partial_{\bar{a}}v,
\\
\langle \nabla T'_4,\nabla v\rangle_{\tilde{m}} =&\ 
\bar{b}^{\alpha}(\tfrac{2-\bar{a}}{2})^{-\alpha}\big(
-(\tfrac{2-\bar{a}}{2}+\tfrac{\alpha-1}{2}\bar{a})\partial_{\xi}v
+\tfrac{\alpha-2}{2}(1-\bar{a})(2-\bar{a})\bar{a}\partial_{\bar{a}}v\big),
\\
\langle \nabla T''_4,\nabla v\rangle_{\tilde{m}} =&\ 
-\tfrac{1}{2-\bar{a}}\partial_{\xi}v+(1-\bar{a})\partial_{\bar{a}}v,
\\
\langle \nabla v,\nabla v\rangle_{\tilde{m}} =&\ 
2\partial_{\bar{a}}v\partial_{\xi}v+\bar{a}(2-\bar{a})\partial_{\bar{a}}v+\tfrac{|\slashpar_{\theta}v|^2}{(1-\bar{a})^2}.
\end{aligned}
$$
For $\mathcal{Q}_{\tilde{m}}(T_4,v)$, we have
$$
\Box_{\tilde{m}}T_4 =
-(1+2\delta')(1-\bar{a})^2 (\tfrac{\bar{a}(2-\bar{a})}{2})^{\delta'-\frac{1}{2}}-n+2n(\tfrac{\bar{a}(2-\bar{a})}{2})^{\delta'+\frac{1}{2}} 
$$
and the non-zero connection components at $p$ are
$$
\begin{gathered}
\Gamma^{\bar{a}}_{\bar{a}\xi}=-(1-\bar{a}),\quad \Gamma^{\bar{a}}_{\xi\xi}=\bar{a}(1-\bar{a})(2-\bar{a}),\quad \Gamma^{\bar{a}}_{\theta\theta}=\bar{a}(1-\bar{a})(2-\bar{a}),
\\
\Gamma^{\xi}_{\xi\xi}=1-\bar{a},\quad \quad \Gamma^{\xi}_{\theta\theta}=1-\bar{a},
\end{gathered}
$$
which implies $\mathcal{Q}_{\tilde{m}}(T_4,v)$.
\end{proof}

\begin{lemma} \label{lem.5.8}
For $n\geq 4$ and  $\lambda \in [\frac{1}{2}-\delta',1-2\delta]$, if $\alpha\in(\frac{n-1}{\lambda}-\frac{1}{8},\tfrac{n-1}{\lambda})$ then in $\widetilde{\Omega}_4$, 
$$
|T'_4|\mathcal{Q}_{\tilde{m}}(T_4,v) \leq  \lambda \langle\mathcal{F}_{\tilde{m}}(T_4,v),\nabla T'_4\rangle_{\tilde{m}}. 
$$
\end{lemma}
\begin{proof}
Let $\alpha'=\lambda\alpha\in(n-1-\frac{\lambda}{8}, n-1)$. We only need to show that
\begin{equation}\label{eq.5.4.1}
\begin{aligned}
D_4\tfrac{|\slashpar_{\theta}v|^2}{(1-\bar{a})^2}+D_5(-\gamma_0v^2) 
&\leq \lambda C_3(\tfrac{|\slashpar_{\theta}v|^2}{(1-\bar{a})^2}-\gamma_0v^2),
\\
D_3(|\bar{a}(2-\bar{a})\partial_{\bar{a}}v-\bar{b}\partial_{\bar{b}}v|^2
-| \bar{b}\partial_{\bar{b}}v|^2)
&\leq E_1 ( \bar{a}(2-\bar{a})\partial_{\bar{a}}v-\bar{b}\partial_{\bar{b}}v|^2 +| \bar{b}\partial_{\bar{b}}v|^2)
+ E_2|\partial_{\bar{a}}v|^2, 
\end{aligned}
\end{equation}
where 
$$
\begin{aligned}
E_1=&\ \lambda C_1-D_1=
(\lambda-\tfrac{1}{2}+\delta')\tfrac{(1-\bar{a})^2}{2-\bar{a}}(\tfrac{\bar{a}(2-\bar{a})}{2})^{\delta'-\frac{1}{2}} 
+\tfrac{\bar{a}+2\lambda(1-\bar{a})}{2(2-\bar{a})^2}
\\ &\ +
\tfrac{\alpha'-2+\bar{a}}{(2-\bar{a})^2}(\tfrac{\bar{a}(2-\bar{a})}{2})^{\delta'+\frac{1}{2}} 
+ \tfrac{\alpha'\bar{a}}{(2-\bar{a})^2}(\tfrac{1}{2}-(\tfrac{\bar{a}(2-\bar{a})}{2})^{\delta'+\frac{1}{2}}  ) ,
\\
E_2=&\ \lambda C_2-D_2
=
\alpha'\bar{a}(1-\bar{a})(1-(\tfrac{\bar{a}(2-\bar{a})}{2})^{\delta'+\frac{1}{2}})
+(1-\lambda)\bar{a}(1-\bar{a}) . 
\end{aligned}
$$
First, 
$$
\begin{aligned}
\lambda C_3-D_4
=&\ 
\tfrac{(\alpha'-n+2)\bar{a}}{2}+(n-2)\bar{a}(\tfrac{\bar{a}(2-\bar{a})}{2})^{\delta'+\frac{1}{2}}+\tfrac{(\alpha'-2)(1-\bar{a})^2}{2-\bar{a}}(\tfrac{\bar{a}(2-\bar{a})}{2})^{\delta'+\frac{1}{2}}
\\&\ 
+\tfrac{(1-\bar{a})^2}{2-\bar{a}}[\lambda+(1-2\delta'-2\lambda)(\tfrac{\bar{a}(2-\bar{a})}{2})^{\delta'+\frac{1}{2}}]
\\
\geq&\ 
\lambda \tfrac{(1-\bar{a})^2}{2-\bar{a}}[1-(\tfrac{\bar{a}(2-\bar{a})}{2})^{\delta'+\frac{1}{2}}] \geq 0.
\end{aligned}
$$
Similarly,
$$
\begin{aligned}
\lambda C_3-D_5=&\ 
\tfrac{(\alpha'-n)\bar{a}}{2}+n\bar{a}(\tfrac{\bar{a}(2-\bar{a})}{2})^{\delta'+\frac{1}{2}}+\tfrac{(1-\bar{a})^2}{2-\bar{a}}(\alpha'-2)(\tfrac{\bar{a}(2-\bar{a})}{2})^{\delta'+\frac{1}{2}}
\\ &\ 
+\tfrac{(1-\bar{a})^2}{2-\bar{a}}[\lambda+(1-2\delta'-2\lambda)(\tfrac{\bar{a}(2-\bar{a})}{2})^{\delta'+\frac{1}{2}}]
\\
\geq &\ 
\tfrac{\bar{a}}{2}[\alpha'-n+n\bar{a}(2-\bar{a})+(\alpha'-2)(1-\bar{a})^2] + \lambda \tfrac{(1-\bar{a})^2}{2-\bar{a}}[1-(\tfrac{\bar{a}(2-\bar{a})}{2})^{\delta'+\frac{1}{2}}]
\\
\geq &\ 
\tfrac{\bar{a}(1-\bar{a})^2}{2}[2\alpha'-n-2+\tfrac{\lambda}{4}] 
\geq   0.
\end{aligned}
$$
 Hence we prove the first inequality in (\ref{eq.5.4.1}). 
 
 For the second inequality in (\ref{eq.5.4.1}), there are two cases. First, if $(\tfrac{\bar{a}(2-\bar{a})}{2})^{\delta'+\frac{1}{2}}\geq \frac{1}{2}$, then $D_3\leq 0$. So we only need to show $-D_3\leq E_1$. In this case, notice that $0\leq 1-\bar{a}\leq \frac{\sqrt{2}}{2}$. Then
$$
\begin{aligned}
E_1+ D_3\geq &\ \tfrac{n+\alpha'+(n-2-\alpha'+2\lambda)(1-\bar{a})}{2(2-\bar{a})^2}
-\tfrac{n+(n-\alpha')(1-\bar{a})}{(2-\bar{a})^2}(\tfrac{\bar{a}(2-\bar{a})}{2})^{\delta'+\frac{1}{2}}
\\ \geq &\ 
\tfrac{n+\alpha'+(n-1-\alpha'-2\delta')(1-\bar{a})}{2(2-\bar{a})^2}
-\tfrac{n+(n-\alpha')(1-\bar{a})}{(2-\bar{a})^2}(\tfrac{\bar{a}(2-\bar{a})}{2})^{\delta'+\frac{1}{2}}
:= f(\delta').
\end{aligned}
$$
Here $f'(\delta')\geq 0$ and hence
$$
\begin{aligned}
E_1+ D_3
\geq&\  \tfrac{1}{(2-\bar{a})^2}
[\tfrac{n+\alpha'+(n-1-\alpha')(1-\bar{a})}{2}
-(n+(n-\alpha')(1-\bar{a}))(\tfrac{\bar{a}(2-\bar{a})}{2})^{\frac{1}{2}}]
\\
\geq&\  
\tfrac{1}{(2-\bar{a})^2}[
\tfrac{n+\alpha'}{2}-\tfrac{n}{\sqrt{2}}-(n-\alpha'-\tfrac{n-1-\alpha'}{\sqrt{2}})(1-\bar{a})
(\tfrac{\bar{a}(2-\bar{a})}{2})^{\frac{1}{2}}
]
\\ 
\geq &\ 
\tfrac{1}{4(2-\bar{a})^2}[ (4-2\sqrt{2})n-\tfrac{3\sqrt{2}}{2}-\tfrac{5}{2}
]
\geq 0
\end{aligned}
$$
for all $n\geq 4$.  
Second, if $(\tfrac{\bar{a}(2-\bar{a})}{2})^{\delta'+\frac{1}{2}}< \frac{1}{2}$, 
then the second inequality in (\ref{eq.5.4.1})  is equivalent to 
\begin{equation}\label{eq.5.4}
2(D_3-E_1)\bar{a}(2-\bar{a}) \partial_{\bar{a}}v\bar{b}\partial_{\bar{b}}v 
\leq 2E_1|\bar{b}\partial_{\bar{b}}v |^2 +(E_2+\bar{a}^2(2-\bar{a}) ^2(E_1-D_3))|\partial_{\bar{a}}v|^2.
\end{equation}
In this case, $E_1\geq 0$ and 
$$
\begin{aligned}
&E_2+\bar{a}^2(2-\bar{a}) ^2(E_1-D_3)
\\
\geq &\ 
\alpha'\bar{a}(1-\bar{a})(1-(\tfrac{\bar{a}(2-\bar{a})}{2})^{\delta'+\frac{1}{2}})+(1-\lambda)\bar{a}(1-\bar{a})
\\ &\ 
+\bar{a}^2[(\alpha'-2+\bar{a})(\tfrac{\bar{a}(2-\bar{a})}{2})^{\delta'+\frac{1}{2}}+\tfrac{\bar{a}+2\lambda(1-\bar{a})}{2}+(\alpha'\bar{a}-(n-1)(2-\bar{a}))(\tfrac{1}{2}-(\tfrac{\bar{a}(2-\bar{a})}{2})^{\delta'+\frac{1}{2}})]
\\
\geq &
\bar{a}^2[
\alpha'(1-\bar{a})(1-(\tfrac{\bar{a}(2-\bar{a})}{2})^{\delta'+\frac{1}{2}}) +\tfrac{2-\bar{a}}{2} +
(\alpha'\bar{a}-(n-1)(2-\bar{a}))(\tfrac{1}{2}-(\tfrac{\bar{a}(2-\bar{a})}{2})^{\delta'+\frac{1}{2}})
]
\\
\geq &\ 
\bar{a}^2
(\alpha'+2-n)(2-\bar{a})(\tfrac{1}{2}-(\tfrac{\bar{a}(2-\bar{a})}{2})^{\delta'+\frac{1}{2}})
\geq 0.
\end{aligned}
$$
Hence (\ref{eq.5.4}) follows by
$$
\begin{gathered}
2E_1 (E_2+\bar{a}^2(2-\bar{a}) ^2(E_1-D_3))\geq 
(D_3-E_1)^2\bar{a}^2(2-\bar{a})^2
\\ \Longleftrightarrow \quad
E_1(2E_2+\bar{a}^2(2-\bar{a}) ^2E_1)\geq \bar{a}^2(2-\bar{a}) ^2D_3^2. 
\end{gathered}
$$
By direct computation, 
$$
\begin{aligned}
2E_2+\bar{a}^2(2-\bar{a}) ^2E_1 \geq 
\alpha' \bar{a}(2-\bar{a})^2(\tfrac{1}{2}-(\tfrac{\bar{a}(2-\bar{a})}{2})^{\delta'+\frac{1}{2}}).
\end{aligned}
$$
Hence 
$$
\begin{aligned}
E_1(2E_2+\bar{a}^2(2-\bar{a}) ^2E_1)
\geq&\ 
\alpha' [\alpha'+1+(\alpha'-2+\bar{a})(2-\bar{a})] \bar{a}^2
(\tfrac{1}{2}-(\tfrac{\bar{a}(2-\bar{a})}{2})^{\delta'+\frac{1}{2}})^2
\\ \geq &\ 
\alpha' [2\alpha'-1] \bar{a}^2
(\tfrac{1}{2}-(\tfrac{\bar{a}(2-\bar{a})}{2})^{\delta'+\frac{1}{2}})^2
\\
\geq&\  \bar{a}^2(2-\bar{a})^2D_3^2.
\end{aligned}
$$
We finish the proof. 
\end{proof}

\begin{lemma}\label{lem.5.9}
For $n\geq 4$ and in $\widetilde{\Omega}_4$
$$
\begin{gathered}
dvol_{\tilde{g}}=(1+\bar{a}^{\frac{n-1}{2}}\bar{b}^{\frac{n-1}{2}}\Theta_1(\tilde{h})) dvol_{\tilde{m}},
\\
\begin{aligned}
\langle F(T_4,v),\nabla T'_4 \rangle_{\tilde{g}} -\langle F(T_4,v),\nabla T'_4 \rangle_{\tilde{m}}
=&\ \bar{a}^{\frac{n-5}{2}}\bar{b}^{\frac{n-1}{2}+\alpha}\big(\Theta_1(\tilde{h}_{\rho\rho})(\bar{a}\partial_{\bar{a}}v,\partial_{\xi}v)+\bar{a}\Theta_1(\tilde{h})(\tilde{\partial}v,v)\big)
\\&\ 
+\bar{a}^{n-4}\bar{b}^{n-1+\alpha}\big(\Theta_2(\tilde{h}_{\rho\rho})(\bar{a}\partial_{\bar{a}}v,\partial_{\xi}v)+\bar{a}\Theta_2(\tilde{h})(\tilde{\partial}v,\tilde{\partial}v)\big),
\end{aligned}
\\
\begin{aligned}
\mathcal{Q}_{\tilde{g}} (T_4,v)-\mathcal{Q}_{\tilde{m}} (T_4,v)
= &\ \bar{a}^{\frac{n-7}{2}}\bar{b}^{\frac{n-1}{2}}\big(\Theta_1(\tilde{h}_{\rho\rho})(\bar{a}\partial_{\bar{a}}v,\partial_{\xi}v)+ 
\Theta_1(\tilde{h}_{\rho\rho})(\tilde{\partial}\tilde{h}_{\rho\rho})(\bar{a}\partial_{\bar{a}}v,\partial_{\xi}v)
\\&\
+\bar{a}\Theta_1(\tilde{h})(\tilde{\partial}v,v)
+\bar{a}\Theta_0(\tilde{h})(\tilde{\partial}\tilde{h})(\tilde{\partial}v,v)\big).
\end{aligned}
\end{gathered}
$$
\end{lemma}
\begin{proof}
The proof is similar as the proof of Lemma \ref{lem.5.4} by substituting $(a,b)$ by $(\bar{a},\bar{b})$ if restricting to $\Omega_4$ and similar as the proof of Lemma \ref{lem.5.2} if restricting to $\Omega_5$.
\end{proof}

\vspace{0.2in}
\section{Energy Estimates}\label{sec.energy}

In this section, we mainly prove the following theorem by energy estimates.
\begin{theorem}\label{thm.est}
 For $n\geq 4$ if given Cauchy data
\begin{equation}\label{cauchy}
(h^0,h^1)\in \widetilde{\mathcal{V}}^{N,\delta+\frac{n-1}{2}}_{\epsilon}
\quad\mathrm{with}\quad N>n+6,\ \delta\in(0,\tfrac{1}{2}) \textrm{ and } \epsilon>0\textrm{ small},
\end{equation}
where $\widetilde{\mathcal{V}}^{N,\delta+\frac{n-1}{2}}_{\epsilon}$ is defined in (\ref{initial.2}), then there exists a global Einstein vacuum solution $g=m+h$ such that 
$\tilde{h}=\tilde{\rho}^{-\frac{n-1}{2}}h$ is $C^{0,\delta}$ up to $S^+_1$ 
and hence the radiation field is well defined: 
\begin{equation}\label{def.RF}
\mathscr{R_F} (h^0,h^1)=\tilde{h}|_{S^{\pm}_1},
\end{equation}
which provides the leading term of $h$ along the bicharacteristic curves. 
Moreover, 
$$
\|\tilde{h}|_{S^{\pm}_1}\| _{\rho_0^{\delta}\rho_2^{\sigma-\frac{n-1}{2}} H_b^{N}(S^{\pm}_1)} \leq 
C\epsilon
$$
for some $\sigma>0$ and $C>0$. 
\end{theorem}

The theorem follows from Proposition \ref{prop.est.1}, \ref{prop.est.2}, \ref{prop.est.3}, \ref{prop.est.4} and Corollary \ref{cor.6.2}, \ref{cor.6.3}, \ref{cor.6.4}.  We do energy estimates for $\tilde{h}$ satisfying the conformal transformation of reduced Einstein equations (\ref{eq.12}) by using the time-like functions defined in Section $5$. The boundary defining functions are given in Section $2$ in each domain $\widetilde{\Omega}_i$ for $1\leq i\leq 4$.  Here $N,\delta$ are fixed all over this section and $\epsilon$ is chosen to be small enough. 
The image space of $\mathscr{R_F}$ will be refined in Section \ref{sec.mollerop}.

\subsection{In $\widetilde{\Omega}_1$.} \label{sec.est1} 
To solve the equations (\ref{eq.12}) in $\widetilde{\Omega}_1$, it is in fact a local existence theorem. We prove the local existence and uniqueness theorem in the conformal setting for $T_1\in[0,\frac{3}{4}]$. The local well-posedness is a classical result. See \cite{CB1} or \cite{Ho} for more details. Here we restate it in weighted b-Sobolev spaces, which will lead to the asymptotical analysis of Einstein vacuum solutions. 

\begin{figure}[htp]
\centering
\includegraphics[totalheight=2in,width=3in]{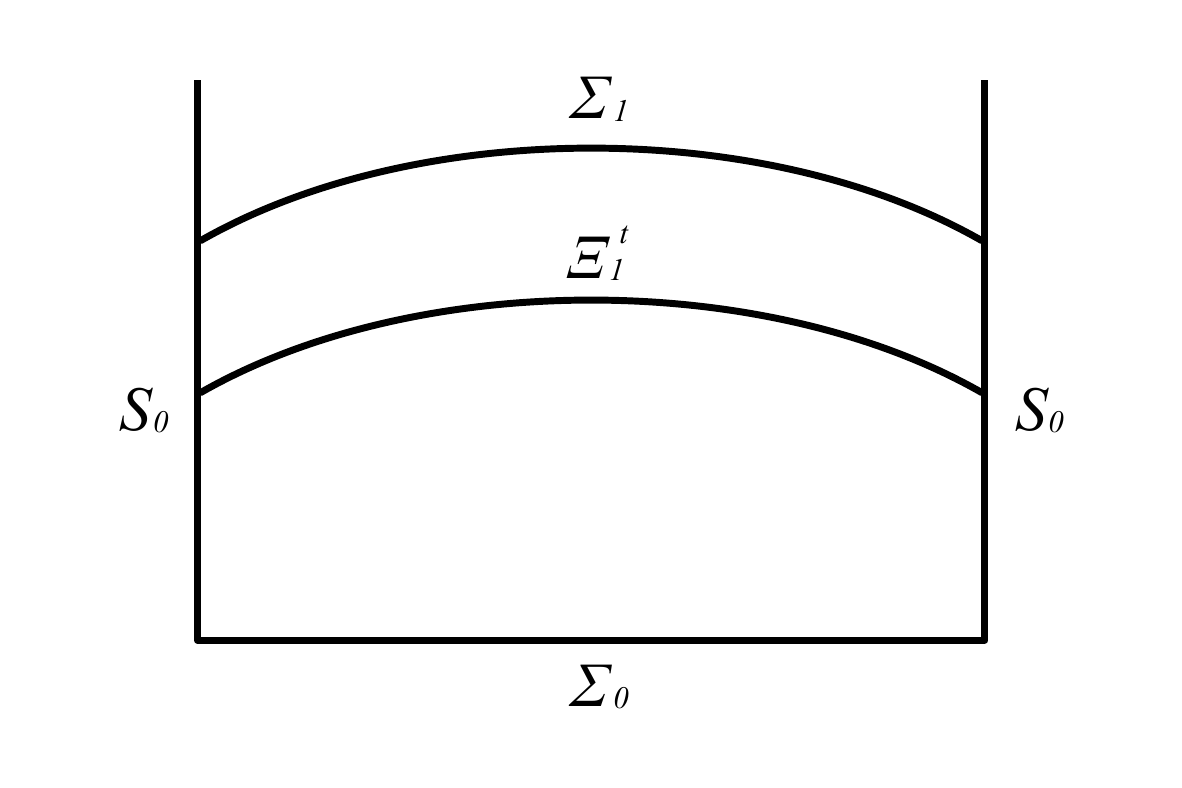}
\caption{$\Omega_1^{\mathfrak{t}}$ is bounded by $\Sigma_0$, $\Xi_1^{\mathfrak{t}}$ and $S_0$.}\label{fig.6.1}
\end{figure}

Recall that $\widetilde{\Omega}_1$ is bounded by $\Sigma_0=\{T_1=0\}$, $\Sigma_1=\{T_1=\frac{3}{4}\}$ and $S_0$. We use $T_1$ to slice the domain. For $\mathfrak{t}\in [0,\frac{3}{4}]$,  denote
$$
\Xi_1^{\mathfrak{t}}=\{T_1=\mathfrak{t}\}.
$$
The space-like hypersurface $\Xi_1^{\mathfrak{t}}$ is a p-submanifold of $X$ with induced boundary defining function $\rho_0$. Let $\Omega_1^{\mathfrak{t}}$ be the domain bounded by $\Sigma_0$, $\Xi_1^{\mathfrak{t}}$ and $S_0$. In particular, $\Omega_1^{\frac{3}{4}}=\widetilde{\Omega}_1$.
 
 Define the high energy norm on $\Xi_1^{\mathfrak{t}}$ by
\begin{equation*}
\begin{aligned}
M^N_1(\mathfrak{t};v,\tilde{g})=\ & \left(\int_{\Xi_1^{\mathfrak{t}}} \sum_{|I|\leq N} e^{-2\delta\log \rho_0}\langle F(T_1,\tilde{D}^Iv),\nabla T_1\rangle_{\tilde{g}} d\mu^{T_1}_{\tilde{g}}   \right)^{\frac{1}{2}}, \quad \forall\ \mathfrak{t}\in[0,\tfrac{3}{4}].
\end{aligned}
\end{equation*}
Here $d\mu^{T_1}_{\tilde{g}}\wedge dT_1=dvol_{\tilde{g}}$ and 
$\tilde{D}\in \chi(r)\mathscr{B}_0+(1-\chi(r))\mathscr{B}_1$
where $0\leq\chi\leq 1$ is a smooth cutoff function such that $\chi(r)=1$ if $r>1$ and $\chi(r)=0$ if $r<\frac{1}{2}$. Here
$$
\mathscr{B}_0=\{\partial_{\mu}: \mu=0,...,n\}, \quad \mathscr{B}_1=\{Z_{\mu\nu}:\mu,\nu=0,...,n\}.
$$
We choose $\gamma_0=-\frac{(n+1)(n-3)}{4}$ to define $\mathcal{F}_{\tilde{g}}(T_1,v)$.

With above notation and given initial data (\ref{cauchy}),  we have for some constant $C_0>0$ such that
\begin{equation}
M^N_1(0;\tilde{h},\tilde{m})<C_0\epsilon.
\end{equation}
\begin{lemma}\label{lem.6.1}
There exists $\epsilon_1>0$ such that if $\|\rho_0^{\frac{n-1}{2}}\tilde{h}\|_{L^{\infty}(\Xi_1^{\mathfrak{t}})}<\epsilon_1$ then on $\Xi_1^{\mathfrak{t}}$
\begin{equation*}
\begin{gathered}
\tfrac{1}{\sqrt{2}}M^N_1(\mathfrak{t};v,\tilde{m}) \leq M^N_1(\mathfrak{t};v,\tilde{g})\leq \sqrt{2}M^N_1(\mathfrak{t};v,\tilde{m}),
\\
\sum_{|I|\leq \frac{N}{2}+1} |\tilde{D}^Iv| \leq C\rho_0^{\delta}M_1^N(\mathfrak{t};v,\tilde{g}),
\end{gathered}
\end{equation*}
where $C$ is independent of $\epsilon_1$ if it is small enough. 
\end{lemma}
\begin{proof}
Near the boundary of $\Xi_1^{\mathfrak{t}}$, we can choose the coordinates $(\rho,\theta)$ and with this coordiantes, 
$$
d\mu_{\tilde{m}}^{T_1}=V(T_1,\rho,\theta)\frac{d\rho}{\rho}d\theta
$$
where $0<c<V\in C^{\infty}([0,\frac{3}{4}]_{T_1}\times [0,1]_{\rho}\times \mathbb{S}^{n-1})$ is bounded above and below. Hence $M^N_1(\mathfrak{t};v,\tilde{m})$ gives the b-Sobolev norm $\rho_0^{\delta}H_b^{N+1}(\Xi_1^{\mathfrak{t}})$.
The equivalence of $M^N_1(\mathfrak{t};v,\tilde{m})$ and $M^N_1(\mathfrak{t};v,\tilde{g})$ comes from the equivalence between 
$\langle F(T_1,\tilde{D}^Iv),\nabla T_1\rangle_{\tilde{m}}$ and $\langle F(T_1,\tilde{D}^Iv),\nabla T_1\rangle_{\tilde{g}}$, as well as the equivalence between the volume forms $dvol_{\tilde{m}}$ and $dvol_{\tilde{g}}$ by Lemma \ref{lem.5.2}. 
Since 
$$
(1/C')M_1^N(\mathfrak{t};v,\tilde{m})\leq \sum_{|I|\leq \frac{N}{2}+1}\|\tilde{D}^I v\|_{\rho_0^{\delta}H_b^{N+1-|I|}(\Xi_1^{\mathfrak{t}})}
\leq C'M_1^N(\mathfrak{t};v,\tilde{m}),
$$
for some $C'>0$, the second inequality comes from Sobolev embedding theorem. See Lemma \ref{lem.2.3}. 
\end{proof}

We find out the existence of $h$ to the equations (\ref{eq.2}) (resp. $\tilde{h}$ to the euqations (\ref{eq.12})) by considering it as the limit of a sequence $h^l$ (resp. $\tilde{h}^l$), which are solutions to the linear equations (resp. their conformal transformation) for $l\geq 0$ starting with $h^{-1}=0$ (resp. $\tilde{h}^{-1}=0$): 
\begin{equation*}
\begin{gathered}
\Box_{g^l} h^{l+1}_{\mu\nu}= F_{\mu\nu}(h^l)(\partial h^l,\partial h^{l+1}),
\\
\big(\Box_{\tilde{g}^l}+\gamma(\tilde{h}^l)\big) \tilde{h}^{l+1} = \tilde{\rho}_0^{\frac{n-1}{2}}\tilde{F}(\tilde{h}^{l},\tilde{h}^{l+1}).
\end{gathered}
\end{equation*}
Here $g^l=m+h^l, \tilde{g}^l=\tilde{\rho}^2g^l, h^l=\tilde{\rho}^{\frac{n-1}{2}}\tilde{h}^l$ and $h^l$ has Cauchy data $(h^0,h^1)$. Notice that there is no guarantee that $(t,x)$ are wave coordinates w.r.t.  $g^l$. However, we don't need this property here. In the conformal transformation of equations, 
\begin{equation*}
\begin{gathered}
 \gamma(\tilde{h}^l)=\gamma_0+\rho_0^{\frac{n-1}{2}} \Theta_1(\tilde{h}^l) ,
\\
 \tilde{F}(\tilde{h}^l,\tilde{h}^{l+1})=\Theta_0(\tilde{h}^l)(\tilde{\partial}\tilde{h}^l,\tilde{\partial}\tilde{h}^{l+1}) +\Theta_0(\tilde{h}^l)(\tilde{\partial}\tilde{h}^l,\tilde{h}^{l+1}) +\Theta_0(\tilde{h}^l)(\tilde{h}^l,\tilde{\partial}\tilde{h}^{l+1}) +\Theta_0(\tilde{h}^l)(\tilde{h}^l,\tilde{h}^{l+1}).
\end{gathered}
\end{equation*}
We also study the equation for $h^{l+1}-h^l$ and $\tilde{h}^{l+1}-\tilde{h}^l$ to get the convergence of $h^l$ and $\tilde{h}^l$:
\begin{equation}\label{eq.29}
\begin{gathered}
\Box_{g^l} (h^{l+1}-h^l)= (\Box_{g^{l-1}}-\Box_{g^{l}})h^l
+(F(\partial h^l,\partial h^{l+1})-F(\partial h^{l-1},\partial h^{l})),
\\
\begin{aligned}
\big(\Box_{\tilde{g}^l}+\gamma(\tilde{h}^l)\big)(\tilde{h}^{l+1}-\tilde{h}^{l}) =&\ 
 (\Box_{\tilde{g}^{l-1}}-\Box_{\tilde{g}^l})\tilde{h}^l
+\big(\gamma(\tilde{h}^{l-1})-\gamma(\tilde{h}^l)\big)\tilde{h}^l
\\
&\ 
+ \tilde{\rho}_0^{\frac{n-1}{2}}\big(\tilde{F}(\tilde{h}^l,\tilde{h}^{l+1}) -\tilde{F}(\tilde{h}^{l-1},\tilde{h}^{l})\big).
\end{aligned}
\end{gathered}
\end{equation}
\begin{lemma}\label{lem.6.2}
There exists $\epsilon_1'>0$ such that if 
$$
M^N_1(\mathfrak{t};\tilde{h}^i,\tilde{g}^{i-1})<\epsilon_1', \quad\forall\ \mathfrak{t}\in[0,\frac{3}{4}]
$$ 
is true for $i=0,...,l$ then 
\begin{equation*}
M^N_1(\mathfrak{t};\tilde{h}^{l+1},\tilde{g}^{l})\leq CM^N_1(0;\tilde{h}^{l+1},\tilde{g}^{l}),\quad 
\forall\ \mathfrak{t}\in[0,\frac{3}{4}].
\end{equation*}
Here $C$ is a constant independent of $l, \mathfrak{t}, \epsilon_1'$ and $M^N_1(0;\tilde{h}^{l+1},\tilde{g}^{l})$ for $\epsilon_1'$ small enough. 
\end{lemma}
\begin{proof}
When $i=0$, since $\tilde{h}^{-1}=0$, the condition in Lemma \ref{lem.6.1} is automatically satisfied. 
Hence $\|\tilde{h}^0\|_{L^{\infty}(\Xi_1^{\mathfrak{t}})}<C'\epsilon_1'$. Here $C'$ is the constant in Lemma \ref{lem.6.1} independent of $l, \mathfrak{t}$.
Choose $\epsilon_1'$ such that $C'\epsilon'<\epsilon_1$ and we can apply Lemma \ref{lem.6.1} until $i=l+1$, i.e.  on $\Xi_1^{\mathfrak{t}}$, 
\begin{equation*}
\sum_{|I|\leq \frac{N}{2}+1} |\tilde{D}^I\tilde{h}^{i}|  \leq C' M_1^N(\mathfrak{t};\tilde{h}^{i},\tilde{g}^{i-1})
\end{equation*}
for $i=0,...,l+1$. 
To get the high energy estimate for $\tilde{h}^{l+1}$, we first deduce the equation for $\tilde{D}^I\tilde{h}^{l+1}$ for $|I|=k\leq N$:
\begin{equation*}
\begin{aligned}
\big(\Box_{\tilde{g}^l}+\gamma(\tilde{h}^l)\big)\tilde{D}^I \tilde{h}^{l+1} =&\  \tilde{D}^I\rho_0^{\frac{n-1}{2}}\tilde{F}(\tilde{h}^l)(\tilde{h}^l,\tilde{h}^{l+1}) + [\Box_{\tilde{g}^l}+\gamma(\tilde{h}^l),\tilde{D}^I]\tilde{h}^{l+1}
\\
=&\ 
\rho_0^{\frac{n-1}{2}} \sum_{|\alpha|\leq k+1,i\leq k+1,|\alpha|+i\leq k+2,1\leq j}\Theta_0(\tilde{h}^{l})(\tilde{\partial}^{\alpha_1}\tilde{h}^{l},..,\tilde{\partial}^{\alpha_j}\tilde{h}^{l}, \tilde{\partial}^{i}\tilde{h}^{l+1})
\end{aligned}
\end{equation*}
By Lemma \ref{lem.3}, 
\begin{equation*}
(M_1^N(\mathfrak{t};\tilde{h}^{l+1},\tilde{g}^l)^2-(M_1^N(0;\tilde{h}^{l+1},\tilde{g}^l))^2 =\int_{\Omega_1^{\mathfrak{t}}}\sum_{|I|\leq N} \mathrm{div}_{\tilde{g}^l}(e^{-2\delta \log\rho_0 } F(T_1,\tilde{D}^I\tilde{h}^{l+1}))dvol_{\tilde{g}^l}.
\end{equation*}
By Lemma \ref{lem.5.1} and Lemma \ref{lem.5.2},  we have
\begin{equation*}
\begin{aligned}
\partial_{\mathfrak{t}}(M_1^N(\mathfrak{t};\tilde{h}^{l+1},\tilde{g}^l))^2=&\ \int_{\Xi_1^{\mathfrak{t}}}\sum_{|I|\leq N} \mathrm{div}_{\tilde{g}^l}(e^{-2\delta \log\rho_1 } \mathcal{F}_{\tilde{g}^l}(T_1,\tilde{D}^I\tilde{h}))d\mu^{T_1}_{\tilde{g}^l}
\\
\leq &\ 2C''(M_1^N(\mathfrak{t};\tilde{h}^{l},\tilde{g}^l)+1)(M_1^N(\mathfrak{t};\tilde{h}^{l+1},\tilde{g}^l))^2
\\
\leq &\  2C''(1+2\epsilon_1') \big(M_1^N(\mathfrak{t};\tilde{h}^{l+1},\tilde{g}^l)\big)^2,
\end{aligned}
\end{equation*}
which implies that
\begin{equation*}
\begin{gathered}
M_1^N(\mathfrak{t};\tilde{h}^{l+1},\tilde{g}^l)\leq e^{C''(1+2\epsilon_1') \mathfrak{t}}M_1^N(0;\tilde{h}^{l+1},\tilde{g}^l)<CM_1^N(0;\tilde{h}^{l+1},\tilde{g}^l),
\quad \forall \mathfrak{t}\in [0,\tfrac{3}{4}]. 
\end{gathered}
\end{equation*}
Here $C>0$ is a constant independent of $\epsilon_1'$ if it is small enough. We finish the proof.
\end{proof}

\begin{proposition}\label{prop.est.1}
There exists $\epsilon_1''>0$ such that if $\epsilon <\epsilon_1''$ then the equations (\ref{eq.12}) have a unique solution in $\widetilde{\Omega}_1$ such that
\begin{equation}\label{eq.6.1}
M^N_1(\mathfrak{t};\tilde{h},\tilde{m}) \leq C\epsilon , \quad \forall\ \mathfrak{t}\in[0,\tfrac{3}{4}],
\end{equation}
where $C>0$ is a constant independent of $\epsilon_1''$ if it is small enough. 
\end{proposition}
\begin{proof}
Since $h^{-1}=0$ and $h^0$ is the solution of standard wave equation on Minkowski space-time, we can apply Lemma \ref{lem.6.2} and get $M^N_1(\mathfrak{t};\tilde{h}^{0},\tilde{g}^{-1})\leq C'M^N_1(0;\tilde{h}^{0},\tilde{g}^{-1})< C'C_0\epsilon_1''$. Here $C'>0$ is the constant in Lemma \ref{lem.6.2} independent of $\mathfrak{t}\in[0,\frac{3}{4}]$. 
Choose $\epsilon_1''>0$ small such that $C'C_0\epsilon_1''<\epsilon_1'$ and
\begin{equation*}
M^N_1(\mathfrak{t};\tilde{h}^{l},\tilde{g}^{l-1})<C'M^N_1(0;\tilde{h}^{l},\tilde{g}^{l-1})<\epsilon_1' , \quad\forall\ l\geq 0,\ \mathfrak{t}\in[0,\frac{3}{4}]. 
\end{equation*}
We do energy estimates for $\tilde{h}^{l+1}-\tilde{h}^{l}$ which satisfies equations (\ref{eq.29}) and hence 
\begin{equation*}
\begin{aligned}
\big(\Box_{\tilde{g}^l}+\gamma(\tilde{h}^l)\big)& \tilde{D}^I (\tilde{h}^{l+1}-\tilde{h}^l) 
=  \tilde{D}^I \rho_0^{\frac{n-1}{2}}\big(\tilde{F}(\tilde{h}^l)(\tilde{h}^l,\tilde{h}^{l+1}) -\tilde{F}(\tilde{h}^{l-1})(\tilde{h}^{l-1},\tilde{h}^{l})\big)
\\&\
+ [\Box_{\tilde{g}^l}+\gamma(\tilde{h}^l),\tilde{D}^I](\tilde{h}^{l+1}-\tilde{h}^l) +\tilde{D}^I\big((\Box_{\tilde{g}^{l-1}}-\Box_{\tilde{g}^l})\tilde{h}^l +(\gamma(\tilde{h}^{l-1})-\gamma(\tilde{h}^{l}))\tilde{h}^l\big).
\end{aligned}
\end{equation*}
Here the estimates for $\tilde{h}^{l+1}-\tilde{h}^l$ are deduced in the same way as in the proof of Lemma \ref{lem.6.2} :
\begin{equation*}
\begin{aligned}
&\partial_{\mathfrak{t}}(M_1^N(\mathfrak{t};\tilde{h}^{l+1}-\tilde{h}^l,\tilde{g}^l))^2
\\
=&\  \int_{\Xi_1^{\mathfrak{t}}} \sum_{|I|\leq N} \mathrm{div}_{\tilde{g}^l}(e^{-2\delta \log\rho_0 } \mathcal{F}_{\tilde{g}^l}(T_1,\tilde{D}^I(\tilde{h}^{l+1}-\tilde{h}^l)))d\mu^{T_1}_{\tilde{g}^l}
\\
\leq&\   2C''\big(M_1^N(\mathfrak{t};\tilde{h}^{l+1}-\tilde{h}^{l},\tilde{g}^l)\big)^2+ 2C''\epsilon_1' M_1^N(\mathfrak{t};\tilde{h}^{l}-\tilde{h}^{l-1},\tilde{g}^l)M_1^N(\mathfrak{t};\tilde{h}^{l+1}-\tilde{h}^l,\tilde{g}^l)
\end{aligned}
\end{equation*}
where $C''>0$ is a constant independent of $\epsilon_1'$. 
This implies that
\begin{equation*}
\begin{aligned}
M_1^N(\mathfrak{t};\tilde{h}^{l+1}-\tilde{h}^{l},\tilde{g}^l) 
\leq&\  C''\epsilon_1'e^{C''\mathfrak{t}} \int_0^{\mathfrak{t}} e^{-C''\mathfrak{t}'}M_1^N(\mathfrak{t}';\tilde{h}^{l}-\tilde{h}^{l-1},\tilde{g}^l) d\mathfrak{t}'
\\
\leq&\  2C''\epsilon_1'e^{C''\mathfrak{t}} \int_0^{\mathfrak{t}} e^{-C''\mathfrak{t}'}M_1^N(\mathfrak{t}';\tilde{h}^{l}-\tilde{h}^{l-1},\tilde{g}^{l-1}) d\mathfrak{t}'.
\end{aligned}
\end{equation*}
For $l\geq 0$, let
\begin{equation*}
\mu_l = \sup_{\mathfrak{t}\in[0,\frac{3}{4}]} M_1^N(\mathfrak{t};\tilde{h}^{l}-\tilde{h}^{l-1},\tilde{m}). 
\end{equation*}
Then $\mu_0\leq \epsilon_1'$ and for $l\geq 1$, $\mu_{l} \leq 2C''\epsilon_1' e^{\frac{3}{4}C''}\mu_{l-1}.$
Choose $\epsilon_1'$ and hence $\epsilon_1''$ even smaller if necessary such that $2C''\epsilon_1'e^{\frac{3}{4}C''}\leq \frac{1}{2}$.  Then for any $\mathfrak{t}\in[0,\frac{3}{4}]$, $l,k\geq 1$
\begin{equation*}
M_1^N(\mathfrak{t};\tilde{h}^{l+k}-\tilde{h}^l,\tilde{m})
=\sum_{i=0}^{k-1}M_1^N(\mathfrak{t};\tilde{h}^{l+i+1}-\tilde{h}^{l+i},\tilde{m})
\leq (\tfrac{1}{2})^{l+1} \sum_{i=0}^{\infty} (\tfrac{1}{2})^i\mu_0\leq (\tfrac{1}{2})^l \epsilon_1'.
\end{equation*}
Hence $\tilde{h}^l$ converges to $\tilde{h}$ with $M_1^N(\mathfrak{t};\tilde{h},\tilde{m})\leq 2\epsilon_1'$ for all $\mathfrak{t}\in[0,\frac{3}{4}]$. Here $\tilde{h}$ provides a solution to equations (\ref{eq.12}). Choose $\epsilon_1'$ and hence $\epsilon_1''$ even smaller if necessary such that Lemma \ref{lem.6.2} holds for $\tilde{h}$: 
$$
M_1^N(\mathfrak{t};\tilde{h},\tilde{m})\leq 2C'M_1^N(0;\tilde{h},\tilde{m}).
$$
We finish proving the existence. 

For the uniqueness, suppose there is another solutions $\tilde{h}'$ with the same Cauchy data such that $M_1^N(\mathfrak{t};\tilde{h}',\tilde{m})\leq C M_1^N(0;\tilde{h}',\tilde{m})$ for some constant $C>0$. Then choose $\epsilon_1''$ even smaller if necessary such that $M_1^N(\mathfrak{t};\tilde{h}',\tilde{m})<\epsilon_1'$. Notice that $\tilde{h}-\tilde{h}'$ satisfy the following equations:
\begin{equation*}
\begin{gathered}
(\Box_{\tilde{g}}+\gamma(\tilde{h}))(\tilde{h}-\tilde{h}') = (\Box_{\tilde{g}'}-\Box_{\tilde{g}})\tilde{h}'+ (\gamma(\tilde{h}')-\gamma(\tilde{h}))\tilde{h}' +\rho_0^{\frac{n-1}{2}}(\tilde{F}(\tilde{h})(\tilde{h},\tilde{h})- \tilde{F}(\tilde{h}')(\tilde{h}',\tilde{h}')),
\\
\begin{aligned}
\ \big(\Box_{\tilde{g}}+\gamma(\tilde{h})\big)\tilde{D}^I (\tilde{h}-\tilde{h}') 
=&\  \tilde{D}^I\rho_0^{\frac{n-1}{2}}\big(\tilde{F}(\tilde{h})(\tilde{h},\tilde{h}) -\tilde{F}(\tilde{h}')(\tilde{h}',\tilde{h}')\big)
+ [\Box_{\tilde{g}^l}+\gamma(\tilde{h}^l),\tilde{D}^I](\tilde{h}^{l+1}-\tilde{h}^l)
\\
&\ + D^I\big((\Box_{\tilde{g}'}-\Box_{\tilde{g}})\tilde{h}'+(\gamma(\tilde{h}')-\gamma(\tilde{h}))\tilde{h}'\big).
\end{aligned}
\end{gathered}
\end{equation*}
A similar estimates as above shows that
\begin{equation*}
\begin{gathered}
\partial_{\mathfrak{t}}M_1^N(\mathfrak{t};\tilde{h}-\tilde{h}',\tilde{g}) 
\leq  C'''M_1^N(\mathfrak{t};\tilde{h}-\tilde{h}',\tilde{g}),
\\
\Longrightarrow\quad M_1^N(\mathfrak{t};\tilde{h}-\tilde{h}',\tilde{g}) \leq e^{C'''\mathfrak{t}} M_1^N(0;\tilde{h}-\tilde{h}',\tilde{g})=0.
\end{gathered}
\end{equation*}
for some constant $C'''>0$. 
Therefore $\tilde{h}=\tilde{h}'$. 
\end{proof}

\begin{remark}
When restricting to $\widetilde{\Omega}_1$, all the proofs work for spacial dimension $n\geq 3$.
\end{remark}

\subsection{In $\widetilde{\Omega}_2$.}
To solve the equations (\ref{eq.12}) up to $\widetilde{\Omega}_2$ with given Cauchy data (\ref{cauchy}), we first notice that the solution we find in $\widetilde{\Omega}_1$ can be extended a bit over $\Sigma_1$ since $\widetilde{\Omega}_1$ is closed. If the Cauchy data is small enough, we show that this extension will never stop in $\widetilde{\Omega}_2$ until it arrives $S^+_1$. 

\begin{figure}[htp]
\centering
\includegraphics[totalheight=2in,width=3in]{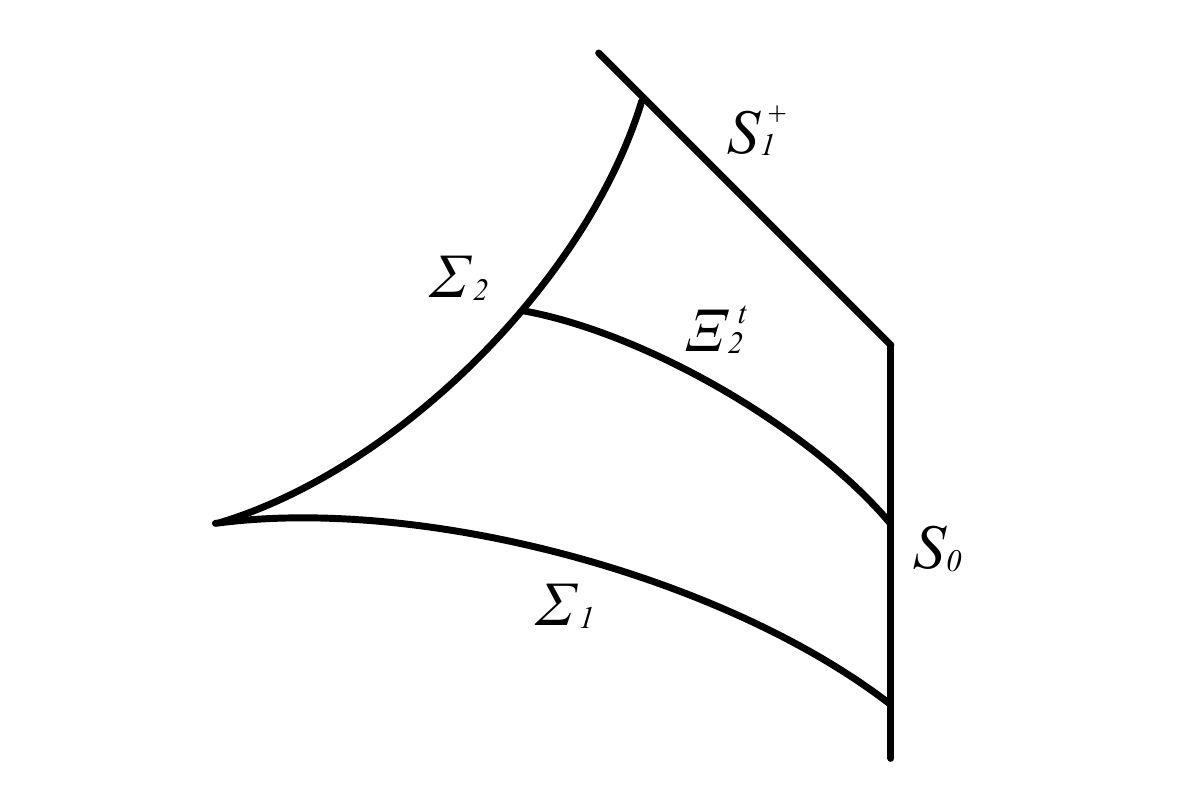}
\caption{$\Omega_2^{\mathfrak{t}}$ is bounded by $\Sigma_1$, $\Sigma_2$, $\Xi_1^{\mathfrak{t}}$ and $S_0$.}\label{fig.6.2}
\end{figure}

Recall that $\widetilde{\Omega}_2$ is bounded by $\Sigma_1=\{T_1=\frac{3}{4}\}$,  $\Sigma_2=\{T_2=-\ln \tau_0\}\cap \widetilde{\Omega}_2$ for $\tau_0>8$ and $S_0, S^+_1$.  We use $T_2'$ to slice the domain.  For $\mathfrak{t}\in[-\frac{1}{4},0]$, denote by
$$
\Xi_2^{\mathfrak{t}}=\{T_2'=\mathfrak{t}\}\cap \widetilde{\Omega}_2
$$
the space-like hypersurface w.r.t. $\tilde{m}$ with defining function $T_2'$. Notice that $\Xi_2^{\mathfrak{t}}$ can be viewed as a closed subset of a p-submanifold of $X$, which has induced boundary defining function $\rho_0=b$. 
Let $\Omega_2^{\mathfrak{t}}$ be the domain bounded by $\Sigma_1$, $\Sigma_2$, $S_0$ and $\Xi_2^{\mathfrak{t}} $. Then $\Omega_2^0=\widetilde{\Omega}_2$. 

Define the high energy norm on $\Xi_2^{\mathfrak{t}}$ and $\Sigma_2$ by 
\begin{equation*}
\begin{aligned}
M_2^N(\mathfrak{t};v,\tilde{g}) =&\ \left(\int_{
\Xi_{2}^{\mathfrak{t}}} \sum_{|I|\leq N} e^{-2\delta\log \rho_0}\langle \mathcal{F}_{\tilde{g}}(T_2,\tilde{D}^Iv),\nabla T_2'\rangle_{\tilde{g}} d\mu^{T_2'}_{\tilde{g}}   \right)^{\frac{1}{2}},
\\
L_2^N(\mathfrak{t};v,\tilde{g}) =&\ \left(\int_{\Sigma_2\cap \{-\frac{1}{4}\leq T_2'\leq \mathfrak{t}\}} \sum_{|I|\leq N} e^{-2\delta\log \rho_0}\langle \mathcal{F}_{\tilde{g}}(T_2,\tilde{D}^Iv),\nabla T_2\rangle_{\tilde{g}} d\mu^{T_2}_{\tilde{g}}   \right)^{\frac{1}{2}}.
\end{aligned}
\end{equation*}
Here $d\mu^{T_2'}_{\tilde{g}}\wedge dT_2'=dvol_{\tilde{g}}$, $d\mu^{T_2}_{\tilde{g}}\wedge dT_2=dvol_{\tilde{g}}$ and $\tilde{D}\in \mathscr{B}_2=\{Z_{\mu\nu}: \mu,\nu=0,...,n\}$.
We choose $\gamma_0=-\frac{(n-3)(n-1)}{4}$ in the definition of $\mathcal{F}_{\tilde{g}}(T_2,v)$. Recall that we choose $\delta'\in (0,\delta)$ such that $\frac{1}{2}-\delta'<1-2\delta$ in the definition of $T_2$. 

With above notation, Proposition \ref{prop.est.1} implies that for some $C_1>0$,
$$
M_2^N(-\tfrac{1}{4};\tilde{h},\tilde{m})<C_1\epsilon, \quad L_2^N(-\tfrac{1}{4};\tilde{h},\tilde{m}) =0.
$$

\begin{lemma}\label{lem.6.3}
For $n\geq 4$, there exists $\epsilon_2>0$ such that if $\|\rho_0^{\frac{n-1}{2}}\tilde{h}\|_{L^{\infty}(\Xi_2^{\mathfrak{t}})}+ \|\rho_0^{\frac{n-1}{2}}\rho_1^{-\delta}\tilde{h}_{\rho\rho}\|_{L^{\infty}(\Xi_2^{\mathfrak{t}})}<\epsilon_2$ then
\begin{equation*}
\begin{gathered}
 \tfrac{1}{\sqrt{2}}M_2^N(\mathfrak{t};v,\tilde{m}) \leq  M_2^N(\mathfrak{t};v,\tilde{g})\leq \sqrt{2}M_2^N(\mathfrak{t};v,\tilde{m}),
\\
\sum_{|I|\leq \frac{N}{2}+2} |\tilde{D}^Iv| \leq C\rho_0^{\delta} M_2^N(\mathfrak{t};v,\tilde{g}),
\\
\sum_{|I|\leq \frac{N}{2}+1} |\tilde{D}^I\partial_a v| \leq C\rho_0^{\delta} \rho_1^{-\frac{1}{2}}M_2^N(\mathfrak{t};v,\tilde{g}),
\end{gathered}
\end{equation*}
for some constant $C$ independent of $\epsilon_2$ if it is small enough. 
Moreover, there exists $\epsilon_2'>0$ such that  if $\sum_{|I|\leq \frac{N}{2}+1}\|\rho_0^{\frac{n-1}{2}}\tilde{D}^I\tilde{h}\|_{L^{\infty}(\Xi_2^{\mathfrak{t}})}<\epsilon_2'$, then
$$
\begin{gathered}
\sum_{|I|\leq \frac{N}{2}+1} |\tilde{D}^I\tilde{h}_{\rho\rho}| \leq C' \rho_0^{\delta}\rho_1^{\frac{1}{2}} M_2^N(\mathfrak{t};\tilde{h},\tilde{m}),
\\
\sum_{|I|\leq N}\|\rho_0^{\frac{n-1}{2}}\tilde{D}^I\tilde{h}_{\rho\rho}\|_{\rho_0^{\delta}H_b^0(\Xi_2^{\mathfrak{t}})} \leq C' \rho_1^{\frac{1}{2}} M_2^N(\mathfrak{t};\tilde{h},\tilde{m}),
\end{gathered}
$$
for some constant $C'$ independent of $\epsilon'_2$ if it is small enough. 
\end{lemma}

\begin{proof}
Similar as the proof of Lemma \ref{lem.6.1}, the equivalence between $M_2^N(\mathfrak{t};v,\tilde{m}) $ and $M_2^N(\mathfrak{t};v,\tilde{g})$ comes from the equivalence between the quadratic forms $\langle \mathcal{F}_{\tilde{g}}(T_2,v),\nabla T'_2\rangle_{\tilde{m}}$ and $\langle \mathcal{F}_{\tilde{g}}(T_2,v),\nabla T'_2\rangle_{\tilde{g}}$, as well as it between the volume forms $dvol_{\tilde{m}}$, $dvol_{\tilde{g}}$ by  Lemma \ref{lem.5.3} and Lemma \ref{lem.5.4}. Here on $\Xi_2^{\mathfrak{t}}$, we can choose the coordinates $(b,\theta)$, so
$$
d\mu^{T_2'}_{\tilde{m}}=\frac{db}{b}d\theta.
$$
Hence $M_2^N(\mathfrak{t};v,\tilde{m}) $ gives the b-Sobolev norm: for some $C''>0$
$$
\sum_{|I|\leq N+1}\|\tilde{D}^Iv\|_{\rho_0^{\delta}H_b^{N+1-|I|}(\Xi_2^{\mathfrak{t}})}
\leq C''M_2^N(\mathfrak{t};\tilde{h},\tilde{m}),
$$
The  second and  third inequalities come from Sobolev embedding theorem. See Lemma \ref{lem.2.3}

Next we show that  the harmonic gauge conditions make $\tilde{h}_{\rho\rho}$ have a bit more decay than $\tilde{h}$ when approaching $S_1^+$. By Lemma \ref{lem.4},
$$
\begin{aligned}
\tilde{h}_{\rho\rho}=&\  \int_0^{b} b'\partial_{b'}\tilde{h}_{\rho\rho} \frac{db'}{b'}
\\
= &\ 
\int_0^{b} a\partial_a\tilde{h}_{\rho\rho}+ a\tilde{\partial}\tilde{h}+\Theta_1((ab')^{\frac{n-1}{2}}\tilde{h})(\tilde{D}\tilde{h})+\Theta_1((ab')^{\frac{n-1}{2}}\tilde{h})(\tilde{h}) \frac{db'}{b'}
\\
\leq &\ 
\left(\int_0^{b}b'^{\delta}\frac{db'}{b'}\right) 
a^{\frac{1}{2}}\big(1+\Theta_1(\|(ab)^{\frac{n-1}{2}}\tilde{h}\|_{\infty}) \big)M_2^N(\mathfrak{t};\tilde{h},\tilde{m})
\\
\leq &\  C\rho_0^{\delta}\rho_1^{\frac{1}{2}} M_2^N(\mathfrak{t};\tilde{h},\tilde{m}).
\end{aligned}
$$
So for $i+j+|I|\leq \frac{N}{2}+1, j>0$, 
$$
\begin{aligned}
|(a\partial a)^{i}(b\partial b)^{j}\slashpar_{\theta}^{I}\tilde{h}_{\rho\rho}|
=&\ \left|(a\partial a)^{i}(b\partial b)^{j-1}\slashpar_{\theta}^{I}
\left(a\partial_a\tilde{h}_{\rho\rho}+ a\tilde{\partial}\tilde{h}+\Theta_1((ab)^{\frac{n-1}{2}}\tilde{h})(\tilde{D}\tilde{h})+\Theta_1((ab)^{\frac{n-1}{2}}\tilde{h})(\tilde{h})\right)\right|
\\
\leq&\ Ca^{\frac{1}{2}}b^{\delta} M_2^N(\mathfrak{t};\tilde{h},\tilde{m});
\end{aligned}
$$
and for $i+|I|\leq \frac{N}{2}+1$, 
$$
\begin{aligned}
|(a\partial a)^{i}\slashpar_{\theta}^{I}\tilde{h}_{\rho\rho}|
=&\ \int_0^{b}  \left|(a\partial a)^{i}\slashpar_{\theta}^{I}\left(a\partial_a\tilde{h}_{\rho\rho}+ a\tilde{\partial}\tilde{h}+\Theta_1((ab')^{\frac{n-1}{2}}\tilde{h})(\tilde{D}\tilde{h})+\Theta_1((ab')^{\frac{n-1}{2}}\tilde{h})(\tilde{h})\right)\right| \frac{db'}{b'}
\\
\leq &\ 
 Ca^{\frac{1}{2}} M_2^N(\mathfrak{t};\tilde{h},\tilde{m})\left(\int_0^{b}b'^{\delta}\frac{db'}{b'}\right).
\end{aligned}
$$
And the last inequality is proved similarly as above: for $i+j+|I|\leq N, j>0$,
$$
\begin{aligned}
&\int_{\mathbb{S}^{n-1}}\int_0^{\rho_0} \left|(a\partial a)^{i}(b\partial b)^{j}\slashpar_{\theta}^{I}\tilde{h}_{\rho\rho}\right|^2 b^{n-1-2\delta}\frac{db}{b}d\theta
\\
= &\int_{\mathbb{S}^{n-1}}\int_0^{\rho_0} 
\left| (a\partial a)^{i}(b\partial b)^{j-1}\slashpar_{\theta}^{I} 
[(a\partial a\tilde{h}_{\rho\rho}+a\tilde{\partial}\tilde{h}+\Theta_1((ab)^{\frac{n-1}{2}}\tilde{h})(\tilde{D}\tilde{h})+\Theta_1((ab)^{\frac{n-1}{2}}\tilde{h})(\tilde{h})] \right|^2 b^{n-1-2\delta}\frac{db}{b}d\theta
\\
\leq &\ C' a\left(M_2^N(\mathfrak{t};\tilde{h},\tilde{m})\right)^2;
\end{aligned}
$$
and for $i+|I|\leq N$, 
$$
\begin{aligned}
&\int_{\mathbb{S}^{n-1}}\int_0^{\rho_0} \left|(a\partial a)^{i}\slashpar_{\theta}^{I}\tilde{h}_{\rho\rho}\right|^2 b^{n-1-2\delta}\frac{db}{b}d\theta
\\
= &\int_{\mathbb{S}^{n-1}}\int_0^{\rho_0} 
\left|(a\partial a)^{i}\slashpar_{\theta}^{I} 
\int_0^b (a\partial a\tilde{h}_{\rho\rho}+a\tilde{\partial}\tilde{h}+\Theta_1((ab)^{\frac{n-1}{2}}\tilde{h})(\tilde{D}\tilde{h})+\Theta_1((ab)^{\frac{n-1}{2}}\tilde{h})(\tilde{h})\frac{db'}{b'} \right|^2 b^{n-1-2\delta}\frac{db}{b}d\theta
\\
\leq& C'a\left(M_2^N(\mathfrak{t};\tilde{h},\tilde{m})\right)^2  \int_0^{\rho_0}
\left( \int_0^b b'^{2\delta}\frac{db'}{b'}\right) b^{n-1-2\delta}\frac{db}{b}
\\
\leq &\ C'' a\left(M_2^N(\mathfrak{t};\tilde{h},\tilde{m})\right)^2. 
\end{aligned}
$$
We finish the proof. 
\end{proof}

\begin{proposition}\label{prop.est.2}
For $n\geq 4$, there exists $\epsilon_2''>0$ such that if $\epsilon < \epsilon_2''$, then the equations (\ref{eq.12}) have a unique solution in $\widetilde{\Omega}_2$ such that
\begin{equation}\label{eq.6.2}
M_2^N(\mathfrak{t};\tilde{h},\tilde{g})+L_2^N(\mathfrak{t};\tilde{h},\tilde{g}) 
\leq C\epsilon |\mathfrak{t}|^{\delta-\frac{1}{2}} ,
\quad\forall\ \mathfrak{t}\in[-\tfrac{1}{4},0],
\end{equation}
where $C>0$ is a constant depending on $\delta, \delta'$ but independent of $\epsilon_2''$ if it is small enough. 
\end{proposition}
\begin{proof}
By the proof of existence theorem in $\widetilde{\Omega}_1$, we see that first we can choose $\epsilon_2''>0$ small such that the we can solve the equations (\ref{eq.12}) for $\mathfrak{t}<\mathfrak{t}_0<0$ with the conditions in Lemma \ref{lem.6.3} holds. 
Then by Lemma \ref{lem.3}, for $\mathfrak{t}<\mathfrak{t}_0$
$$
\begin{aligned}
&(M_2^N(\mathfrak{t};\tilde{h},\tilde{g}))^2+ (L_2^N(\mathfrak{t};\tilde{h},\tilde{g}))^2
-(M_2^N(-\tfrac{1}{4};\tilde{h},\tilde{g}))^2 \\
=& \int_{\Omega_2^{\mathfrak{t}}} \sum_{|I|\leq N}\mathrm{div}_{\tilde{g}} (e^{-2\delta \log\rho_0}
\mathcal{F}_{\tilde{g}}(T_2,\tilde{D}^I\tilde{h}))dvol_{\tilde{g}}. 
\end{aligned}
$$
Here the divergence term can be estimated according to Lemma \ref{lem.6}, \ref{lem.5.3}, \ref{lem.5.4}:
\begin{equation*}
\begin{aligned}
 &\ \int_{\Xi_{2}^{\mathfrak{t}}} \sum_{|I|\leq N}\mathrm{div}_{\tilde{g}} \big(e^{-2\delta \log\rho_0}\mathcal{F}_{\tilde{g}}(T_2,\tilde{D}^I\tilde{h})\big) dvol_{\tilde{g}}^{T_2'}
\\
=&\ \int_{\Xi_{2}^{\mathfrak{t}}} \sum_{|I|\leq N} b^{-2\delta} 
\big(\mathcal{Q}_{\tilde{g}}(T_2,\tilde{D}^I\tilde{h})
-2\delta\langle \mathcal{F}_{\tilde{g}}(T_2,\tilde{D}^I\tilde{h}),\nabla\log b\rangle_{\tilde{g}} 
 +\langle \nabla T_2, \nabla \tilde{D}^I\tilde{h} \rangle_{\tilde{g}}(\Box_{\tilde{g}}+\gamma_0)\tilde{D}^I\tilde{h}\big) dvol_{\tilde{g}}^{T_2'}
\\
\leq &\
[a^{-1}(1-2\delta)+2 C'a^{-\frac{1}{2}}+2C'a^{\frac{n-5}{2}-\delta'}M_2^N(\mathfrak{t};\tilde{h},\tilde{g}))] (M_2^N(\mathfrak{t};\tilde{h},\tilde{g}))^2
\end{aligned}
\end{equation*}
where $C''>0$ is a constant independent of $\epsilon_2''$ if it is small enough. Notice the faster decaying of $\tilde{h}_{\rho\rho}$ and $\tilde{D}^I\tilde{h}_{\rho\rho}$ in $\rho_1$ direction guarantees that the perturbation of $\mathcal{Q}_{\tilde{g}}, \mathcal{F}_{\tilde{g}}, \Box_{\tilde{g}}$ all fall into the smaller term and the leading term $a^{-1}(1-2\delta)$ is determined by the linear equations. Here $\partial_{\mathfrak{t}} (L_2^N(\mathfrak{t};\tilde{h},\tilde{g}))^2\geq 0$ since it is an increasing function of $\mathfrak{t}$. 
Choose $\epsilon''$ smaller if necessary such that 
$$M_2^N(\mathfrak{t};\tilde{h},\tilde{g})<|\mathfrak{t}|^{\delta-\frac{1}{2}},\quad \forall \mathfrak{t}\in [-\tfrac{1}{4},\mathfrak{t}_0)
.$$
Then we have 
$$
\begin{aligned}
& \partial_{\mathfrak{t}}M_2^N(\mathfrak{t};\tilde{h},\tilde{g}) \leq \left(a^{-1}(\tfrac{1}{2}-\delta)+C'a^{-\frac{1}{2}}
+C'a^{\delta-\delta'-1}\right)M_2^N(\mathfrak{t};\tilde{h},\tilde{g}),
\\ \Longrightarrow \ 
& \partial_{\mathfrak{t}}\big(a^{\frac{1}{2}-\delta}M_2^N(\mathfrak{t};\tilde{h},\tilde{g}) \big)
\leq \left(C'a^{-\frac{1}{2}}
+C'a^{\delta-\delta'-1}\right)\big(a^{\frac{1}{2}-\delta}M_2^N(\mathfrak{t};\tilde{h},\tilde{g})\big),
\\ \Longrightarrow\ 
&M_2^N(\mathfrak{t};\tilde{h},\tilde{g})
\leq \frac{e^{C'+\frac{C'}{(\delta-\delta')4^{\delta-\delta'} }}}{4^{\frac{1}{2}-\delta}} |\mathfrak{t}|^{\delta-\frac{1}{2}}M_2^N(-\tfrac{1}{4};\tilde{h},\tilde{g}). 
\end{aligned}
$$
Hence
$$
M_2^N(\mathfrak{t};\tilde{h},\tilde{g})\leq C''|\mathfrak{t}|^{\delta-\frac{1}{2}}M_2^N(-\tfrac{1}{4};\tilde{h},\tilde{g}),
\quad \forall \mathfrak{t}\in [-\tfrac{1}{4},\mathfrak{t}_0)
$$
for some $C''>0$ independent of $\epsilon_2''$ if it is small enough. Notice that
$$
\begin{gathered}
\partial_{\mathfrak{t}} (L_2^N(\mathfrak{t};\tilde{h},\tilde{g}))^2 
\leq a^{2\delta-2}[(1-2\delta)+2C'a^{\frac{1}{2}}+2C'C''a^{\delta-\delta'}M^N_2(-\tfrac{1}{4};\tilde{h},\tilde{g}) ](C'' M^N_2(-\tfrac{1}{4};\tilde{h},\tilde{g}) )^2,
\\
\Longrightarrow 
L_2^N(\mathfrak{t};\tilde{h},\tilde{g}) \leq C''' |\mathfrak{t}|^{\delta-\frac{1}{2}} M^N_2(-\tfrac{1}{4};\tilde{h},\tilde{g}),\quad \forall \mathfrak{t}\in [-\tfrac{1}{4},\mathfrak{t}_0)
\end{gathered}
$$
where $C'''>0$ is a constant independent of $\epsilon''_2$ too. 
So we prove (\ref{eq.6.2}) for $\mathfrak{t}\in [-\tfrac{1}{4},\mathfrak{t}_0)$. Moreover, 
$$
\begin{gathered}
\begin{aligned}
\sum_{|I|\leq \frac{N}{2}+1} \|\rho_0^{\frac{n-1}{2}}\tilde{D}^I\tilde{h}\|_{L^{\infty}(\Xi_{2}^{\mathfrak{t}})} 
\leq\ & \sum_{|I|\leq \frac{N}{2}+1} \left( \|\tilde{D}^I\tilde{h}\|_{L^{\infty}(\Sigma_1)} +C''''\int_{-\frac{1}{4}}^{\mathfrak{t}} |\mathfrak{t}'|^{-\frac{1}{2}}M_2^N(\mathfrak{t}';\tilde{h},\tilde{g})d\mathfrak{t}'\right)
\\
\leq\ & \sum_{|I|\leq \frac{N}{2}+1} \left( \|\tilde{D}^I\tilde{h}\|_{L^{\infty}(\Sigma_1)}+CC''''(\tfrac{1}{4})^{\delta}\epsilon\right),
\end{aligned}
\\
\|\rho_0^{\frac{n-1}{2}}\rho_0^{-\delta}\tilde{h}_{\rho\rho}\|_{L^{\infty}(\Xi_{2}^{\mathfrak{t}})} \leq 
C''''|\mathfrak{t}|^{\frac{1}{2}-\delta}M_2^N(\mathfrak{t};\tilde{h},\tilde{g})\leq CC''''\epsilon ,
\end{gathered}
$$
for some $C''''>0$ independent of $\epsilon_2''$ if it is small enough. 
Hence choosing $\epsilon_2''$ even smaller if necessary such that above two inequalities guarantee that the conditions in Lemma \ref{lem.6.3} holds up to $T_2'=\mathfrak{t}_0$, which allows us to extend the solution up to and hence a bit over $T_2'=\mathfrak{t}_0$ with all the above estimates hold without changing of constants. Therefore 
the extension will never stop until it arrives $\widetilde{\Omega}_2\cap S_1$.  We finish the proof. 
\end{proof}

\begin{corollary}\label{cor.6.2}
For $n\geq 4$, suppose $\epsilon  < \epsilon_2''$ where $\epsilon_2''$ is chosen such that Proposition \ref{prop.est.2} holds. Then $\tilde{h}$ is 
$C^{0,\delta}$ up to $\widetilde{\Omega}_2\cap S^{+}_1$ and
$$
\|\tilde{h}|_{S_1^{+}\cap\widetilde{\Omega}_2}\|_{\rho_0^{\delta}H_b^N(S_1^{+}\cap\widetilde{\Omega}_2)} \leq C\epsilon
$$
for some $C>0$ depending on $\delta,\delta' $ but independent of $\epsilon_2''$ if it is small enough. 
\end{corollary}
\begin{proof}
By Lemma \ref{lem.6.3} and Proposition \ref{prop.est.2}, for some constant $C'>0$
$$
|\partial_a\tilde{h}|\leq C'a^{\delta-1}b^{\delta} \epsilon. 
$$
Hence $\tilde{h}$ is $C^{0,\delta}$ up to $\widetilde{\Omega}_2\cap S_1^{+}$ and the radiations field $\tilde{h}|_{\widetilde{\Omega}_2\cap S_1}$ is well defined.  Moreover, define
$$
Q_2^N(\mathfrak{t};v,\tilde{g}) = \left(\int_{\Xi_{2}^{\mathfrak{t}}} 
\sum_{k+|I|\leq N} |(b\partial_b)^k\slashpar_{\theta}^I v|^2 d\mu^{T_2'}_{\tilde{g}} \right)^{\frac{1}{2}}.
$$
Then for some $C'', C''', C''''>0$
$$
\begin{gathered}
 \partial_{\mathfrak{t}}(Q_2^N(\mathfrak{t};\tilde{h},\tilde{g}))^2 \leq 
2|\mathfrak{t}|^{-\frac{1}{2}} M_2^N(\mathfrak{t};\tilde{h},\tilde{g}) Q_2^N(\mathfrak{t};\tilde{h},\tilde{g})\big(C''+C'''Q_2^N(\mathfrak{t};\tilde{h},\tilde{g})\big),
\\
Q^N_2(-\tfrac{1}{4};\tilde{h},\tilde{g}) \leq C''''M^N_2(-\tfrac{1}{4};\tilde{h},\tilde{m}).
\end{gathered}
$$
Here the $C'''$ term comes from the derivative of the volume form $d\mu^{T_2'}_{\tilde{g}}$. Since
$$
A=\int_{-\frac{1}{4}}^0 C'''|\mathfrak{t}|^{-\frac{1}{2}} M_2^N(\mathfrak{t};\tilde{h},\tilde{g}) d\mathfrak{t}
\leq CC''' \epsilon \int_{-\frac{1}{4}}^0 |\mathfrak{t}|^{\delta-1} d\mathfrak{t}
$$
is bounded, this implies that 
$$
Q_2^N(\mathfrak{t};\tilde{h},\tilde{g})
 \leq  e^A \int_{-\frac{1}{4}}^{0} C''|\mathfrak{t}|^{-\frac{1}{2}}M_2^N(\mathfrak{t};\tilde{h},\tilde{g}) d\mathfrak{t} 
+Q_2^N(-\tfrac{1}{4};\tilde{h},\tilde{g})
\leq C'''''\epsilon,\quad \forall \mathfrak{t}\in [-\tfrac{1}{4},0]. 
$$
Finally, $Q_2^N(0;\tilde{h},\tilde{g})$ gives the norm bound of $\tilde{h}|_{S_1^+\cap\widetilde{\Omega}_2}$. 
\end{proof}

\subsection{In $\widetilde{\Omega}_3$.}
To solve the equations (\ref{eq.12}) in  $\widetilde{\Omega}_3$ we first notice that with Cauchy data small enough,  the solution we find in $\widetilde{\Omega}_1\cup \widetilde{\Omega}_2$ can be extended to $\widetilde{\Omega}_3''$ since it is a compact domain in the interior of $X$. 
We show that if the Cauchy data is small enough then this extension can continue into $\widetilde{\Omega}_3'$ and never stop until it arrives $S_1^+$. 

\begin{figure}[htp]
\centering
\includegraphics[totalheight=1.5in,width=3in]{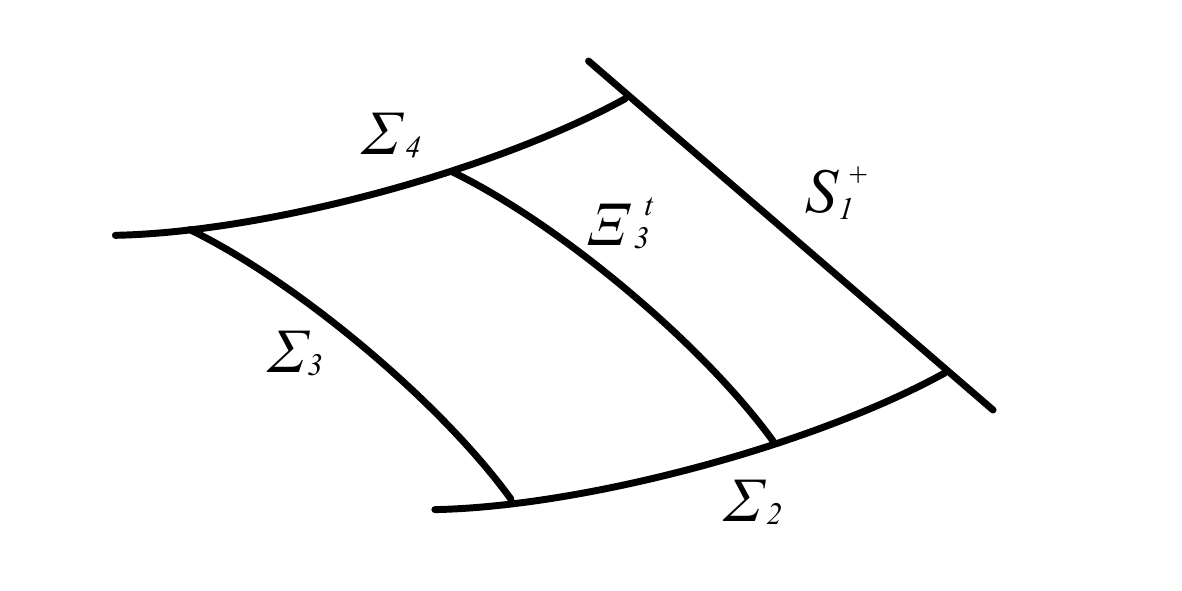}
\caption{$\Omega_3^{\mathfrak{t}}$ is bounded by $\Sigma_2$, $\Sigma_3$, $\Sigma_4$ and $\Xi_3^{\mathfrak{t}}$.}\label{fig.6.3}
\end{figure}

Recall that $\Sigma_2=\{T_2=-\ln \tau_0\}\cap \widetilde{\Omega}_2$,  $\Sigma_3=\{T_3'=-1\}\cap \widetilde{\Omega}_3$ and 
$\Sigma_4=\{T_3=\tau_0\}$. For $\mathfrak{t}\in[-1,0]$, denote by 
$$
\Xi_3^{\mathfrak{t}}=\{T_3'=\mathfrak{t}\}\cap \widetilde{\Omega}_3, 
$$ 
the space like hypersurfaces w.r.t. $\tilde{m}$ and $\Omega_3^{\mathfrak{t}}$ the domain bounded by $\Sigma_2,\Sigma_3,\Sigma_4$ and $\Xi^{\mathfrak{t}}_3$. Then $\Omega_3^0=\widetilde{\Omega}_3'$. 

Define the high energy norm on $\Sigma_2$, $\Sigma_4$ and  $\Xi^{\mathfrak{t}}_3$ by 
$$
\begin{aligned}
M_3^N(\mathfrak{t};v,\tilde{g}) =&\ \left(\int_{\Xi_3^{\mathfrak{t}}} 
\sum_{|I|\leq N} \langle \mathcal{F}_{\tilde{g}}(T_3,\tilde{D}^Iv),\nabla T_3'\rangle_{\tilde{g}} 
d\mu^{T_3'}_{\tilde{g}}   \right)^{\frac{1}{2}},
\\
L_3^N(\mathfrak{t};v,\tilde{g}) =&\ \left(\int_{\Sigma_4\cap \{-1\leq T_3'\leq\mathfrak{t}\}} 
\sum_{|I|\leq N} \langle \mathcal{F}_{\tilde{g}}(T_3,\tilde{D}^Iv),\nabla T_3\rangle_{\tilde{g}} 
d\mu^{T_3}_{\tilde{g}}   \right)^{\frac{1}{2}},
\\
\tilde{L}_3^N(\mathfrak{t};v,\tilde{g}) =&\ \left(\int_{\Sigma_2\cap \{-1\leq T_3'\leq\mathfrak{t}\}} 
\sum_{|I|\leq N} \langle \mathcal{F}_{\tilde{g}}(T_3,\tilde{D}^Iv),\nabla T_2\rangle_{\tilde{g}} 
d\mu^{T_2}_{\tilde{g}}   \right)^{\frac{1}{2}}.
\end{aligned}
$$
Here $d\mu_{\tilde{g}}^{T_3'}\wedge dT'_3=dvol_{\tilde{g}}$, 
$d\mu_{\tilde{g}}^{T_3}\wedge dT_3=dvol_{\tilde{g}}$, 
$d\mu_{\tilde{g}}^{T_2}\wedge dT_2=dvol_{\tilde{g}}$ and $\tilde{D}\in \mathscr{B}_3=\{\partial_{\tau},\rho\partial_{\rho},Z_{ij}: i,j=1,...,n\}.$
We choose 
$\gamma_0=-\frac{(n-1)(n-3)}{4}$ in the definition of $\mathcal{F}_{\tilde{g}}(T_3,v)$ 
in $\widetilde{\Omega}'_3$.  

Notice that in $\widetilde{\Omega}_2$, we use the quadratic form $\langle\mathcal{F}_{\tilde{g}}(T_2,\tilde{D}^Iv),\nabla T_2\rangle_{\tilde{g}}$ on $\Sigma_2$ to define the norm $L_2^N(T_2';v,\tilde{g})$. Here the two quadratic forms $\langle\mathcal{F}_{\tilde{g}}(T_3,\tilde{D}^Iv),\nabla T_2\rangle_{\tilde{g}}$ and $\langle\mathcal{F}_{\tilde{g}}(T_2,\tilde{D}^Iv),\nabla T_2\rangle_{\tilde{g}}$ are equivalent 
on $\Sigma_2\cap \widetilde{\Omega}'_3$, i.e.
$$
(1/C)\langle\mathcal{F}_{\tilde{g}}(T_2,\tilde{D}^Iv),\nabla T_2\rangle_{\tilde{g}}\leq 
\langle\mathcal{F}_{\tilde{g}}(T_3,\tilde{D}^Iv),\nabla T_2\rangle_{\tilde{g}}
\leq C\langle\mathcal{F}_{\tilde{g}}(T_2,\tilde{D}^Iv),\nabla T_2\rangle_{\tilde{g}}
$$
for some constant $C>0$. On the other hand,  $T_2'$ and $T_3'$ are equivalent on $\Sigma_2\cap \widetilde{\Omega}'_3$, i.e. 
$$
c<\frac{\partial T_2'}{\partial T_3'} <\frac{1}{c}
$$
for some constant $c>0$. Hence the estimate of $L_2^N(T_2';\tilde{h},\tilde{g})$ in $\widetilde{\Omega}_2$ give the estimate of $\tilde{L}_3^N(T_3';\tilde{h},\tilde{g}) $. 
By Proposition \ref{prop.est.1}, Proposition \ref{prop.est.2} and  classical theorems for wave equations in compact domain, we have
\begin{equation}\label{eq.30}
\begin{gathered}
M_3^N(-1;\tilde{h},\tilde{m})< C_3\epsilon,\quad
L_3^N(-1;\tilde{h},\tilde{m})< C_3\epsilon,\quad
\tilde{L}_3^N(-1;\tilde{h},\tilde{m})=0;
\\
\partial_{\mathfrak{t}} (\tilde{L}_3^N(\mathfrak{t};\tilde{h},\tilde{g}))^2
< C^2_3|\mathfrak{t}|^{2\delta-2}\epsilon^2, \quad \tilde{L}_3^N(\mathfrak{t};\tilde{h},\tilde{g})\leq C_3|\mathfrak{t}|^{\delta-\frac{1}{2}} \epsilon;
\\
\sum_{|I|\leq \frac{N}{2}+1}\left|\tilde{D}^I\tilde{h}_{\rho\rho}|_{\Sigma_2}\right|\leq C_3\rho_1^{\delta}\epsilon, 
\end{gathered}
\end{equation}
where $C_3>0$ independent of $\epsilon$ if it is small enough. In particular we choose 
$\delta'\in (0,\delta)$ 
 in the definition of $T_3$ such that $\frac{1}{2}-\delta'<1-2\delta$ and $\tau_0>8$ large in the definition of $T_3'$. Similar as Lemma \ref{lem.6.3}, we have
\begin{lemma}\label{lem.6.4}
Assume $n\geq 4$. 
There exists $\epsilon_3>0$ such that if $\|\tilde{h}\|_{L^{\infty}(\Xi_3^{\mathfrak{t}})}+ 
\|\rho_1^{-\delta}\tilde{h}_{\rho\rho}\|_{L^{\infty}(\Xi_3^{\mathfrak{t}})}<\epsilon_3$,  then on $\Xi_3^{\mathfrak{t}}$ for $\tilde{D}\in \mathscr{B}_3$,
\begin{equation*}
\begin{gathered}
\tfrac{1}{\sqrt{2}}M_3^N(\mathfrak{t};v,\tilde{m}) \leq  M_3^N(\mathfrak{t};v,\tilde{g})
\leq \sqrt{2}M_3^N(\mathfrak{t};v,\tilde{m}),
\\
\sum_{|I|\leq \frac{N}{2}+1} |\tilde{D}^Iv| \leq C M_3^N(\mathfrak{t};v,\tilde{g}),
\\ 
\sum_{|I|\leq \frac{N}{2}} |\tilde{D}^I\partial_a v| \leq C\rho_1^{-\frac{1}{2}}M_3^N(\mathfrak{t};v,\tilde{g}),
\end{gathered}
\end{equation*}
for some constant $C$ independent of $\epsilon_3$ if it is small enough. Moreover, There exists $\epsilon'_3>0$ such that if
$\sum_{|I|\leq \frac{N}{2}+1}\|\tilde{D}^I\tilde{h}\|_{L^{\infty}(\Xi_3^{\mathfrak{t}})}<\epsilon'_3$, then
$$
\begin{gathered}
\sum_{|I|\leq \frac{N}{2}+1} |\tilde{D}^I\tilde{h}_{\rho\rho}| \leq C' \rho_1^{\frac{1}{2}} M_3^N(\mathfrak{t};\tilde{h},\tilde{m}),
\\
\sum_{|I|\leq N}\|\tilde{D}^I\tilde{h}_{\rho\rho}\|_{H^0(\Xi_3^{\mathfrak{t}})} \leq C' \rho_1^{\frac{1}{2}} M_3^N(\mathfrak{t};\tilde{h},\tilde{m}).
\end{gathered}
$$
\end{lemma}
\begin{proof}
The proof of equivalence between $M_3^N(\mathfrak{t};v,\tilde{m})$ and $M_3^N(\mathfrak{t};v,\tilde{g})$ is the same as before, which is directly from the equivalence of quadratic forms $\langle \mathcal{F}_{\tilde{g}}(T_3,v),\nabla T_3'\rangle_{\tilde{m}} $, $\langle \mathcal{F}_{\tilde{g}}(T_3,v),\nabla T_3'\rangle_{\tilde{g}}$ and equivalence of volume forms $\mathrm{d}vol_{\tilde{m}}$, $\mathrm{d}vol_{\tilde{g}}$. And $M_3^N(\mathfrak{t};\tilde{h},\tilde{m})$ gives the Sobolev norms:
$$
\sum_{|I|\leq N+1}\|\tilde{D}^Iv\|_{H^{N+1-|I|}(\Xi_3^{\mathfrak{t}})}\leq C''M_3^N(\mathfrak{t};\tilde{h},\tilde{m})
$$
for some $C''>0$. Changing coordinates from $(\rho,\tau,\theta)$ to $(\mathfrak{t}, \tau, \theta)$, we can prove the rest inequalities similar as in $\widetilde{\Omega}_2$.  
\end{proof}

\begin{proposition}\label{prop.est.3}
For $n\geq 4$, there exists $\epsilon_3''>0$ such that if 
$\epsilon< \epsilon_3''$, then the equations (\ref{eq.12}) have a unique 
solution up to $\Sigma_4$ such that
\begin{equation}\label{eq.6.3}
M_3^N(\mathfrak{t};\tilde{h},\tilde{g})+L_3^N(\mathfrak{t};\tilde{h},\tilde{g})
\leq C\epsilon |\mathfrak{t}|^{\delta-\frac{1}{2}},
\quad\forall\ \mathfrak{t}\in[-1,0],
\end{equation}
where $C>0$ is a constant depending on $\delta, \delta'$ but independent of $\epsilon_3''$ if it is small enough. 
\end{proposition}
\begin{proof}
 We first can choose $\epsilon_3''$ such that we can solve the equation in 
$\widetilde{\Omega}_1\cup \widetilde{\Omega}_2\cup \widetilde{\Omega}''_3$ and (\ref{eq.30}) 
holds. Choose $\epsilon_3''$ even smaller
if necessary such that we can extend the solution up to $\mathfrak{t}<\mathfrak{t}_0$ for some 
$\mathfrak{t}_0\in(-1,0)$ with Lemma \ref{lem.6.4}  satisfied and
$$
M_3^N(\mathfrak{t};\tilde{h},\tilde{g})<|\mathfrak{t}|^{\delta-\frac{1}{2}}.
$$
Then by Lemma \ref{lem.3},
$$
\begin{aligned}
&(M_3^N(\mathfrak{t};\tilde{h},\tilde{g}))^2 +(L_3^N(\mathfrak{t};\tilde{h},\tilde{g}))^2
-(L_3^N(-1;\tilde{h},\tilde{g}))^2-(M_3^N(-1;\tilde{h},\tilde{g}))^2-
(\tilde{L}_3^N(\mathfrak{t};\tilde{h},\tilde{g}))^2
\\
=& \int_{\Omega_3^t} \sum_{|I|\leq N}\mathrm{div}_{\tilde{g}}
(\mathcal{F}_{\tilde{g}}(T_3,\tilde{D}^I\tilde{h}))dvol_{\tilde{g}}.
\end{aligned}.
$$
Here the divergence term can be estimated by Lemma \ref{lem.5.5}, \ref{lem.5.6} as follows:
\begin{equation*}
\begin{aligned}
 &\ \int_{\Xi_3^{\mathfrak{t}}} \sum_{|I|\leq N}\mathrm{div}_{\tilde{g}} 
(\mathcal{F}_{\tilde{g}}(T_3,\tilde{D}^I\tilde{h})) dvol_{\tilde{g}}^{T_3'}
\\
=&\ \int_{\Xi_3^{\mathfrak{t}}} \sum_{|I|\leq N}  
\big(\mathcal{Q}_{\tilde{g}}(T_3,\tilde{D}^I\tilde{h})
 +\langle \nabla T_3, \nabla \tilde{D}^I\tilde{h} \rangle_{\tilde{g}}
(\Box_{\tilde{g}}+\gamma_0)\tilde{D}^I\tilde{h}\big) dvol_{\tilde{g}}^{T_3'}
\\
\leq &\
\big( (\tfrac{1}{2}-\delta')|\mathfrak{t}|^{-1}+2C'|\mathfrak{t}|^{-\frac{1}{2}}+2C'|\mathfrak{t}|^{\frac{n-5}{2}-\delta'}M_3^N(\mathfrak{t};\tilde{h},\tilde{g})\big)
 (M_3^N(\mathfrak{t};\tilde{h},\tilde{g}))^2
\\
\leq&\ \big( (\tfrac{1}{2}-\delta')|\mathfrak{t}|^{-1}+4 C'|\mathfrak{t}|^{\delta-\delta'-1}\big)
(M_3^N(\mathfrak{t};\tilde{h},\tilde{g}))^2.
\end{aligned}
\end{equation*}
Hence
$$
\begin{aligned}
&\partial_{\mathfrak{t}}(M_3^N(\mathfrak{t};\tilde{h},\tilde{g}))^2
\leq \partial{\mathfrak{t}}(\tilde{L}_3^N(\mathfrak{t};\tilde{h},\tilde{g}))^2
+\int_{\Xi_3^{\mathfrak{t}}} \sum_{|I|\leq N}\mathrm{div}_{\tilde{g}} 
(\mathcal{F}_{\tilde{g}}(T_3,\tilde{D}^I\tilde{h})) dvol_{\tilde{g}}^{T_3'}
\\
&\quad\quad\quad\quad\quad\quad\quad
\leq C^2_3|\mathfrak{t}|^{2\delta-2}\epsilon^2 + 
\big( (\tfrac{1}{2}-\delta')|\mathfrak{t}|^{-1}+4 C'|\mathfrak{t}|^{\delta-\delta'-1}\big)
(M_3^N(\mathfrak{t};\tilde{h},\tilde{g}))^2,
\\
\Longrightarrow \quad &
\partial_{\mathfrak{t}}\big(e^{-\frac{4C'}{\delta-2\delta'+\frac{1}{2}}
|\mathfrak{t}|^{\delta-2\delta'+\frac{1}{2}}}
|\mathfrak{t}|^{\frac{1}{2}-\delta'}(M_3^N(\mathfrak{t};\tilde{h},\tilde{g}))^2\big)
\leq C^2_3\epsilon^2 e^{-\frac{4C'}{\delta-\delta'}|\mathfrak{t}|^{\delta-\delta'}}
|\mathfrak{t}|^{2\delta-\delta'-\frac{3}{2}},
\\
\Longrightarrow \quad &
(M_3^N(\mathfrak{t};\tilde{h},\tilde{g}))^2 \leq 
e^{\frac{4C'}{\delta-2\delta'+\frac{1}{2}}|\mathfrak{t}|^{\delta-2\delta'+\frac{1}{2}}} 
\left(
\tfrac{C^2_3\epsilon^2}{-2\delta+\delta'+\frac{1}{2}} |\mathfrak{t}|^{2\delta-1}
+|\mathfrak{t}|^{\delta'-\frac{1}{2}}
(M_3^N(-1;\tilde{h},\tilde{g}))^2\right)
\\
&\quad\quad\quad\quad\quad\quad
\leq \left(C''\epsilon\right)^2 |\mathfrak{t}|^{2\delta-1}.
\end{aligned}
$$
Here $\delta'<\delta<\frac{1}{2}$ and
$C''>0$ depends on $\delta,\delta'$ only if $\epsilon_3''$ small enough. 
Moreover, 
$$
\begin{gathered}
\partial_{\mathfrak{t}}((L_3^N(\mathfrak{t};\tilde{h},\tilde{g}))^2) 
\leq \partial{\mathfrak{t}}(\tilde{L}_3^N(\mathfrak{t};\tilde{h},\tilde{g}))^2
+ \big( (\tfrac{1}{2}-\delta')|\mathfrak{t}|^{-1}+4 C'|\mathfrak{t}|^{\delta-\delta'-1}\big)
(M_3^N(\mathfrak{t};\tilde{h},\tilde{g}))^2
\\
\Longrightarrow\quad  
L_3^N(\mathfrak{t};\tilde{h},\tilde{g}) \leq C'''\epsilon |\mathfrak{t}|^{\delta-\frac{1}{2}}. 
\end{gathered}
$$
where $C'''>0$  is independent of $\epsilon_3''$ if it is small enough. We have proved (\ref{eq.6.3}) for all $\mathfrak{t}\in [-1,\mathfrak{t}_0)$. Notice that by Lemma \ref{lem.6.4}, 
$$
\begin{gathered}
\|\tilde{h}\|_{L^{\infty}(\Xi_3^{\mathfrak{t}})} \leq \|\tilde{h}\|_{L^{\infty}(\Xi_3^{-1})} 
+ C''''\int_{-1}^{\mathfrak{t}} |\mathfrak{t}'|^{-\frac{1}{2}}M_3^N(\mathfrak{t}';\tilde{h},\tilde{g}) d\mathfrak{t}' \leq 
 \|\tilde{h}\|_{L^{\infty}(\Sigma_3^{-1})} + \tfrac{1}{\delta}C''''C''\epsilon, 
 \\
 \|\rho_1^{-\delta}\tilde{h}_{\rho\rho}\|_{L^{\infty}(\Xi_3^{\mathfrak{t}})} 
 \leq  \|\rho_1^{-\delta}\tilde{h}_{\rho\rho}\|_{L^{\infty}(\Sigma_2)} + C''''|\mathfrak{t}|^{\frac{1}{2}-\delta}M_3^N(\mathfrak{t};\tilde{h},\tilde{g}) \leq  \|\rho_1^{-\delta}\tilde{h}_{\rho\rho}\|_{L^{\infty}(\Sigma_2)}+C''''C''\epsilon. 
 \end{gathered}
$$
Choose $\epsilon_3''$ even smaller if necessary such that the above two inequalities guarantee that Lemma 
\ref{lem.6.4} holds. Then when we extend the solution up to and a bit over $\mathfrak{t}_0$, all the  above estimates are still true with the same constants. Hence the extension does not stop until it arrives $S_1^{+}$. 
\end{proof}

\begin{corollary}\label{cor.6.3}
 For $n\geq 4$ suppose $\epsilon< \epsilon_3''$ with $\epsilon_3''$ chosen in Proposition \ref{prop.est.3}.
 Then $\tilde{h}$ is $C^{\delta}$ up to $\widetilde{\Omega}_3\cap S_1^+$ and
$$
\|\tilde{h}|_{S_1^+\cap \widetilde{\Omega}_3}\|_{H^N(S_1^+\cap \widetilde{\Omega}_3)}\leq C\epsilon
$$
for some constant $C>0$ depending on $\delta$ but independent of $\epsilon_3''$ 
if it is small enough. 
\end{corollary}
\begin{proof}
By Lemma \ref{lem.6.4} and Proposition \ref{prop.est.3}, for some constant $C'>0$
$$
|\partial_{\rho}\tilde{h}|\leq C'\epsilon\rho^{\delta-1} .
$$
Hence $\tilde{h}$ is $C^{0,\delta}$ up to $\widetilde{\Omega}_3\cap S_1^{+}$ and 
the radiations field $\tilde{h}|_{\widetilde{\Omega}_3\cap S_1}$ is well defined.  
Moreover, define
$$
Q_3^N(\mathfrak{t};v,\tilde{g}) = \left(\int_{\Xi_3^{\mathfrak{t}}} 
\sum_{k+|I|\leq N} |\partial_{\tau}^k\slashpar_{\theta}^I v|^2 d\mu^{T_3'}_{\tilde{g}} \right)^{\frac{1}{2}}.
$$
Then for some $C'',C''',C''''>0$,
$$
\begin{gathered}
 \partial_{\mathfrak{t}}(Q_3^N(\mathfrak{t};\tilde{h},\tilde{g}))^2 \leq 
2 |\mathfrak{t}|^{-\frac{1}{2}} M_3^N(\mathfrak{t};\tilde{h},\tilde{g}) 
Q_3^N(\mathfrak{t};\tilde{h},\tilde{g})(C''+C'''Q_3^N(\mathfrak{t};\tilde{h},\tilde{g})),
\\
Q^N_3(-1;\tilde{h},\tilde{g}) \leq C''''M^N_3(-1;\tilde{h},\tilde{m}).
\end{gathered}
$$
Similar as the proof of Corollary 2, we can show that
$$
Q_3^N(\mathfrak{t};\tilde{h},\tilde{g})
\leq C\epsilon, \quad \forall \mathfrak{t}\in [-1,0]. 
$$
Finally, $Q_3^N(0;\tilde{h},\tilde{g}) $ gives the norm bound of $\tilde{h}|_{S_1^+\cap\widetilde{\Omega}_3}$.
\end{proof}

\subsection{In $\widetilde{\Omega}_4$.} 
To Solve the equation (\ref{eq.12}) in $\widetilde{\Omega}_4$, we first notice that we can extend the solution a bit over $\Sigma_4$ on a compact region. We show that if $\epsilon>0$ small enough then this extension can be taken all over $\widetilde{\Omega}_4$ and hence we solve the equation globally. 
\begin{figure}[htp]
\centering
\includegraphics[totalheight=1.5in,width=3in]{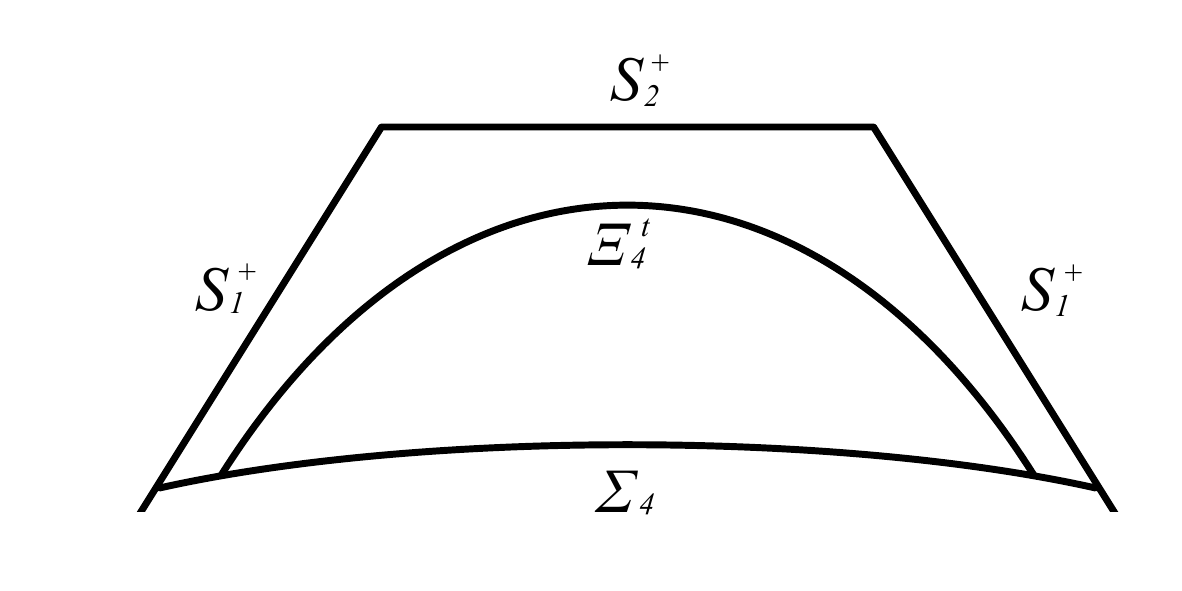}
\caption{$\Omega_4^{\mathfrak{t}}$ is bounded by $\Sigma_4$ and $\Xi_{4}^{\mathfrak{t}}$. }\label{fig.6.4}
\end{figure}

Recall that $\Sigma_4=\{T_3=\tau_0\}$. Let $\mathfrak{t}_0=\inf _{\widetilde{\Omega}_4}\{T_4'\}$. 
For $\mathfrak{t}\in [\mathfrak{t}_0, 0]$, denote by 
$$
\Xi_4^{\mathfrak{t}} =\widetilde{\Omega}_4\cap \{T_4'=\mathfrak{t}\}
$$
the space-like hypersurfaces w.r.t $\tilde{m}$ with defining function $T'_4$ and $\Omega_4^{\mathfrak{t}}$ the domain bounded by $\Sigma_4$ and $\Xi_4^{\mathfrak{t}}$. Then $\Omega_4^0=\widetilde{\Omega}_4$. 

Define the high energy norm on $\Xi_4^{\mathfrak{t}}$ and $\Sigma_4$ by
$$
\begin{aligned}
M_4^N(\mathfrak{t};v,\tilde{g}) =&\ \left( \int_{\Xi_4^{\mathfrak{t}}} \sum_{|I|\leq N}
\langle \mathcal{F}_{\tilde{g}}(T_4,\tilde{\rho}^{-\frac{n-1}{2}}\tilde{D}^I\tilde{\rho}^{\frac{n-1}{2}}v),\nabla T_4'
\rangle_{\tilde{g}} d\mu_{\tilde{g}}^{T_4'}
\right)^{\frac{1}{2}}, 
\\
L_4^N(\mathfrak{t};v,\tilde{g}) =&\ \left( \int_{\Sigma_4 \cap \{\mathfrak{t}_0\leq T_4'\leq \mathfrak{t}\}} \sum_{|I|\leq N}
\langle \mathcal{F}_{\tilde{g}}(T_4,\tilde{\rho}^{-\frac{n-1}{2}}\tilde{D}^I\tilde{\rho}^{\frac{n-1}{2}}v),\nabla T_3
\rangle_{\tilde{g}} d\mu_{\tilde{g}}^{T_3}
\right)^{\frac{1}{2}}. 
\end{aligned}
$$
Here $d\mu_{\tilde{g}}^{T_4'}\wedge dT'_4=dvol_{\tilde{g}}$, 
$d\mu_{\tilde{g}}^{T_3}\wedge dT_3=dvol_{\tilde{g}}$ and $\tilde{D}\in \mathscr{B}_4=\{Z_{\mu\nu}: \mu,\nu=0,...,n\}$. We choose 
$\gamma_0=-\frac{n^2-1}{4}$ in the definition of $\mathcal{F}_{\tilde{g}}(T_4,v)$. 
Notice that, $M_4^N$ is defined in a different way from $M_2^N$ and $M_3^N$. By the proof of Lemma \ref{lem.4.7}
$$\tilde{\rho}^{-\frac{n-1}{2}}\tilde{D}^I\tilde{\rho}^{\frac{n-1}{2}}v=\tilde{D}^Iv+\sum_{|J|<|I|}c_{IJ}\tilde{D}^Jv$$
for some $c_{IJ}\in C^{\infty}(\overline{\Omega_4\cap\Omega_5})$. This modified definition will simplify the energy estimates in Proposition \ref{prop.est.4}. (If we define $M_2^N$ and $M_3^N$ in above way, then we can erase $|\mathfrak{t}|^{-\frac{1}{2}}$ terms in the estimates of divergence terms when proving Proposition \ref{prop.est.2} and \ref{prop.est.3}.)

Notice that 
the two quadratic forms $\langle \mathcal{F}_{\tilde{g}}(T_3,\tilde{D}^Iv),\nabla T_3\rangle_{\tilde{g}}$
and $\langle\mathcal{F}_{\tilde{g}}(T_4,\tilde{D}^Iv),\nabla T_3\rangle_{\tilde{g}}$ are equivalent 
on $\Sigma_4$, i.e.  
$$
(1/C)\langle \mathcal{F}_{\tilde{g}}(T_3,\tilde{D}^Iv),\nabla T_3\rangle_{\tilde{g}} 
\leq \langle\mathcal{F}_{\tilde{g}}(T_4,\tilde{D}^Iv),\nabla T_3\rangle_{\tilde{g}}
\leq C\langle \mathcal{F}_{\tilde{g}}(T_3,\tilde{D}^Iv),\nabla T_3\rangle_{\tilde{g}}
$$
for some $C>0$. Moreover, $T_3'$ and $T_4'$ are also equivalent on $\Sigma_4$, i.e.
$$
c<\frac{\partial T_4'}{\partial T_3'}<\frac{1}{c}
$$
for some $c>0$. Hence the estimate of $L_3^N(T_3';\tilde{h},\tilde{g})$ in $\widetilde{\Omega}_3$ gives the estimate of $L_4^N(T_4';\tilde{h},\tilde{g})$. 
By Proposition \ref{prop.est.1}, \ref{prop.est.2}, \ref{prop.est.3}, for all $\mathfrak{t}\in[\mathfrak{t}_0,0]$
\begin{equation}\label{eq.6.4.1}
\begin{gathered}
M_4^N(\mathfrak{t}_0;\tilde{h},\tilde{g})=0,\quad
L_4^N(\mathfrak{t}_0;\tilde{h},\tilde{g})=0,\quad
\\
L_4^N(\mathfrak{t};\tilde{h},\tilde{g})
< C_4|\mathfrak{t}|^{\delta-\frac{1}{2}}\epsilon,\quad
\sum_{|I|\leq \frac{N}{2}+1}\left|\tilde{D}^I\tilde{h}_{\rho\rho}|_{\Sigma_4}\right|\leq C_4\rho_1^{\delta}\epsilon,
\end{gathered}
\end{equation}
where $C_4>0$ independent of $\epsilon$ if it is small enough. In particular,  in the definition of $T_4'$ we choose $\alpha \in (1, \frac{n-1}{1-2\delta})$ such that Lemma \ref{lem.5.8} holds for $\lambda=1-2\delta$, and in the definition of $T_4$ we choose $\delta'\in (0,\delta)$ which is close to $\delta$ such that $\delta-\delta'<\frac{1}{2}-\delta$ and $\alpha\in (1,\frac{n-1}{1-2\delta'})$.

Similar as Lemma \ref{lem.6.4}, we have

\begin{lemma}\label{lem.6.5}
Assume $n\geq 4$.  
There exists $\epsilon_4>0$ such that if $\|\rho_2^{\frac{n-1}{2}}\tilde{h}\|_{L^{\infty}(\Xi_4^{\mathfrak{t}})}+ 
\|\rho_1^{-\delta}\rho_2^{\frac{n-1}{2}}\tilde{h}_{\rho\rho}\|_{L^{\infty}(\Xi_4^{\mathfrak{t}})}<\epsilon_4$,  then
on $\Xi_4^{\mathfrak{t}}$ for $\tilde{D}\in \mathscr{B}_4$, 
\begin{equation*}
\begin{gathered}
\tfrac{1}{\sqrt{2}}M_4^N(\mathfrak{t};v,\tilde{m}) \leq  M_4^N(\mathfrak{t};v,\tilde{g})
\leq \sqrt{2}M_4^N(\mathfrak{t};v,\tilde{m}),
\\
\sum_{|I|\leq \frac{N}{2}+1} |\tilde{D}^Iv| \leq C M_4^N(\mathfrak{t};v,\tilde{g}),
\\ 
\sum_{|I|\leq \frac{N}{2}} |\tilde{D}^I\partial_a v| \leq C\rho_1^{-\frac{1}{2}}M_4^N(\mathfrak{t};v,\tilde{g}),
\end{gathered}
\end{equation*}
for some constant $C$ independent of $\epsilon_4$ if it is small enough. Moreover, there exists $\epsilon_4'>0$ such that if $\sum_{|I|\leq \frac{N}{2}}\|\rho_2^{\frac{n-1}{2}}\tilde{h}\|_{L^{\infty}(\Xi_4^{\mathfrak{t}})}<\epsilon_4'$ then
$$
\begin{gathered}
\sum_{|I|\leq \frac{N}{2}+1} |\tilde{D}^I\tilde{h}_{\rho\rho}| \leq C' \rho_1^{\frac{1}{2}} (\ln\rho_2) M_4^N(\mathfrak{t};\tilde{h},\tilde{m}),
\\
\sum_{|I|\leq N}\|\rho_2^{\frac{\alpha}{2}}\tilde{D}^I\tilde{h}_{\rho\rho}\|_{L^2(\Xi_4^{\mathfrak{t}},\rho_2^{\alpha}d\mu_{\tilde{m}}^{T_4'})} \leq C' |\mathfrak{t}|^{\frac{1}{2}} \left|\ln|t|\right| M_4^N(\mathfrak{t};\tilde{h},\tilde{m}).
\end{gathered}
$$
\end{lemma}
\begin{proof}
The proof of equivalence between $M_4^N(\mathfrak{t};v,\tilde{m})$ and $M_4^N(\mathfrak{t};v,\tilde{g})$ is the same as before, which is directly from the equivalence of quadratic forms $\langle \mathcal{F}_{\tilde{g}}(T_4,v),\nabla T_4'\rangle_{\tilde{m}} $, $\langle \mathcal{F}_{\tilde{g}}(T_4,v),\nabla T_4'\rangle_{\tilde{g}}$ and equivalence of volume forms $\mathrm{d}vol_{\tilde{m}}$, $\mathrm{d}vol_{\tilde{g}}$. And $ M_4^N(\mathfrak{t};\tilde{h},\tilde{m})$ gives the Sobolev norm:
$$
\sum_{|I|\leq N+1}\|\tilde{D}^I\tilde{h}\|_{L^2(\Xi_4^{\mathfrak{t}},\rho_2^{\alpha}d\mu_{\tilde{m}}^{T_4'})} 
\leq C''M_4^N(\mathfrak{t};\tilde{h},\tilde{m})
$$
for some $C''>0$. The Sobolev embedding inequalities here come from thinking of $\Xi_4^{\mathfrak{t}}$ as a union of two domains, $\{\bar{a}\geq \sqrt{\mathfrak{t}}\}$ a bounded domain with compact metric and  $\{\bar{a}\leq \sqrt{\mathfrak{t}}\}$ a domain in manifold with boundary which carries b-metric. 
Changing coordinates from $(a,b,\theta)$ to $(\mathfrak{t}, b, \theta)$, we can prove the rest inequalities similar as in $\widetilde{\Omega}_2$.  
\end{proof}

\begin{proposition}\label{prop.est.4}
For $n\geq 4$, there exists $\epsilon_4''>0$ such that if 
$\epsilon< \epsilon_4''$, then the equations (\ref{eq.12}) have a unique 
solution all over $\widetilde{\Omega}_4$ such that
\begin{equation}\label{eq.6.4}
M_4^N(\mathfrak{t};\tilde{h},\tilde{g})
\leq C\epsilon |\mathfrak{t}|^{\delta-\frac{1}{2}},
\quad\forall\ \mathfrak{t}\in[\mathfrak{t}_0,0],
\end{equation}
where $C>0$ is a constant depending on $\delta, \delta'$ but independent of $\epsilon_4''$ if it is small enough. 
\end{proposition}
\begin{proof}
 We first can choose $\epsilon_4''$ such that we can solve the equation in 
$\widetilde{\Omega}_1\cup \widetilde{\Omega}_2\cup \widetilde{\Omega}_3$ and (\ref{eq.6.4.1}) 
holds. Choose $\epsilon_4''$ even smaller
if necessary such that we can extend the solution to $\mathfrak{t}<\mathfrak{t}_1$ for some 
$\mathfrak{t}_1\in(\mathfrak{t}_0,0)$ with Lemma \ref{lem.6.5}  satisfied and
$$
|\mathfrak{t}|^{\frac{1}{2}-\delta}
M_4^N(\mathfrak{t};\tilde{h},\tilde{g})\leq 1.
$$
Then
By Lemma \ref{lem.3}, if denoting $\tilde{h}^I=\tilde{\rho}^{-\frac{n-1}{2}}\tilde{D}^I\tilde{\rho}^{\frac{n-1}{2}}\tilde{h}$, then
$$
\begin{aligned}
(M_4^N(\mathfrak{t};\tilde{h},\tilde{g}))^2 -(L_4^N(\mathfrak{t};\tilde{h},\tilde{g}))^2
=& \int_{\Omega_4^{\mathfrak{t}}} \sum_{|I|\leq N}\mathrm{div}_{\tilde{g}}
(\mathcal{F}_{\tilde{g}}(T_4,\tilde{h}^I))dvol_{\tilde{g}}.
\end{aligned}
$$
Here the divergence term can be estimated by Lemma \ref{lem.5.7}, \ref{lem.5.8}, \ref{lem.5.9} as follows: 
\begin{equation*}
\begin{aligned}
 &\ \int_{\Xi_4^{\mathfrak{t}}} \sum_{|I|\leq N}\mathrm{div}_{\tilde{g}} 
(F(T_4,\tilde{h}^I)) dvol_{\tilde{g}}^{T_4'}
\\
=&\ \int_{\Xi_4^{\mathfrak{t}}} \sum_{|I|\leq N}  
\big(\mathcal{Q}_{\tilde{g}}(T_4,\tilde{h}^I)
 +\langle \nabla T_4, \nabla \tilde{h}^I\rangle_{\tilde{g}}
(\Box_{\tilde{g}}+\gamma_0)\tilde{h}^I\big) dvol_{\tilde{g}}^{T_4'}
\\
\leq &\
[(1-2\delta)|\mathfrak{t}|^{-1}+
2C'|\mathfrak{t}|^{\frac{n-5}{2}-\delta'}|\ln|\mathfrak{t}||M_4^N(\mathfrak{t};\tilde{h},\tilde{g})]
 (M_4^N(\mathfrak{t};\tilde{h},\tilde{g}))^2
\\
\leq&\ [ (1-2\delta)|\mathfrak{t}|^{-1}+2 C'|\mathfrak{t}|^{\delta-\delta'-1}|\ln|\mathfrak{t}||]
(M_4^N(\mathfrak{t};\tilde{h},\tilde{g}))^2.
\end{aligned}
\end{equation*}
Hence
$$
\begin{gathered}
\partial_{\mathfrak{t}}(M_4^N(\mathfrak{t};\tilde{h},\tilde{g}))^2
\leq \partial_{\mathfrak{t}}(L_4^N(\mathfrak{t};\tilde{h},\tilde{g}))^2 + 
[(1-2\delta)|\mathfrak{t}|^{-1}+2 C'|\mathfrak{t}|^{\delta-\delta'-1}|\ln|\mathfrak{t}||] 
(M_3^N(\mathfrak{t};\tilde{h},\tilde{g}))^2
\\
\Longrightarrow \quad 
\partial_{\mathfrak{t}}\big(
f(\mathfrak{t})
|\mathfrak{t}|^{1-2\delta}(M_4^N(\mathfrak{t};\tilde{h},\tilde{g}))^2
\big)
\leq f(\mathfrak{t})
|\mathfrak{t}|^{1-2\delta} \partial_{\mathfrak{t}}(L_4^N(\mathfrak{t};\tilde{h},\tilde{g}))^2
\end{gathered}
$$
where 
$
f(\mathfrak{t})= e^{\int_{t_0}^{\mathfrak{t}}2 C'|\mathfrak{t}'|^{\delta-\delta'-1}|\ln|\mathfrak{t}'|| d\mathfrak{t}'}
$
is positive and uniformly bounded for $\mathfrak{t}\in[\mathfrak{t}_0,0]$. This implies 
$$
M_4^N(\mathfrak{t};\tilde{h},\tilde{g}) 
\leq C|\mathfrak{t}|^{\delta-\frac{1}{2}}\epsilon, \quad \forall \mathfrak{t}\in[\mathfrak{t}_0,\mathfrak{t}_1). 
$$
Here
$C$ depends on $\delta,\delta'$ only if $\epsilon_4''$ small enough.  So we have proved (\ref{eq.6.4}) for $\mathfrak{t}\in[\mathfrak{t}_0,\mathfrak{t}_1)$. 
By Sobolev embedding theorem and Lemma \ref{lem.6.5}, this gives
$$
\begin{gathered}
\sum_{|I|\leq\frac{N}{2}+1}|\rho_2^{\frac{n-1}{2}}\tilde{D}^I\tilde{h}|\leq C'''\rho_2^{\sigma'}\epsilon , 
\quad
 |\rho_2^{\frac{n-1}{2}}\tilde{h}_{\rho\rho}| < C''' \rho_1^{\delta} \rho_2^{\sigma'}(\ln \rho_2)\epsilon ,
\end{gathered}
$$
where $\sigma'=\frac{n-1}{2}-\alpha(\frac{1}{2}-\delta)>0$, and $C'''$ only depends on $\delta, \delta',\alpha $ if $\epsilon_4''$ is small enough. 
Choose $\epsilon_4''$ even smaller if necessary such that 
$M_4^N(\mathfrak{t};\tilde{h},\tilde{g})<C\epsilon |\mathfrak{t}|^{\delta-\frac{1}{2}}$ implies the conditions in Lemma 
\ref{lem.6.5}. Then we can extend the solution up to and a bit over $\Xi_4^{\mathfrak{t}_1}$ such that all the above inequalities still
hold with the same constants.  Hence the extension does not stop until it arrives at $\mathfrak{t}=0$, i.e. we solve the equation (\ref{eq.12}) globally.
\end{proof}

\begin{corollary}\label{cor.6.4}
 For $n\geq 4$ suppose $\epsilon<\epsilon_4''$ with $\epsilon_4''$ chosen in Proposition \ref{prop.est.4}.
 Then $\tilde{h}$ is $C^{\delta}$ up to $\widetilde{\Omega}_4\cap S_1^+$ and for any $\sigma\in(0,\sigma')$ where $\sigma'=\frac{n-1}{2}-\alpha(\frac{1}{2}-\delta)$,
$$
\|\tilde{h}_{S_1^+}\|_{\rho_2^{\sigma-\frac{n-1}{2}}H_b^{N}(\widetilde{\Omega}_4\cap S_1^+)}\leq C\epsilon
$$
for some constant $C>0$ depending on $\alpha,\delta,\delta', \sigma$ only but independent of $\epsilon_4''$ 
if it is small enough. 
\end{corollary}
\begin{proof}
By Lemma \ref{lem.6.5} and Proposition \ref{prop.est.4}, in $\widetilde{\Omega}_4$ for some constant $C'>0$
$$
\sum_{|I|\leq \frac{N}{2}+1}|\tilde{D}^I\tilde{h}|\leq C'\bar{a}^{\delta-\frac{1}{2}}\bar{b}^{\alpha(\delta-\frac{1}{2})} \epsilon,
\quad
\sum_{|I|\leq \frac{N}{2}}|\partial_{\bar{a}}\tilde{D}^I\tilde{h}|\leq C'\bar{a}^{\delta-1} \bar{b}^{\alpha(\delta-\frac{1}{2})}\epsilon.
$$
Hence $\tilde{h}$ is $C^{0,\delta}$ up to $\widetilde{\Omega}_4\cap S_1^{+}$ and 
the radiations field $\tilde{h}|_{\widetilde{\Omega}_4\cap S_1}$ is well defined.  
Moreover, define
$$
\begin{aligned}
Q_4^N(\mathfrak{t};v,\tilde{g}) = \left(\int_{\Xi_4^{\mathfrak{t}}} \bar{b}^{n-1-2\sigma+\alpha}
\sum_{|I|\leq N} |\tilde{D}^I v|^2 d\mu^{T_4'}_{\tilde{g}} \right)^{\frac{1}{2}},
\\
P_4^N(\mathfrak{t};v,\tilde{g}) = \left(\int_{\Sigma_4\cap\{\mathfrak{t}_0\leq T_4'<\mathfrak{t}\}} \sum_{|I|\leq N} |\tilde{D}^I v|^2 d\mu^{T_3}_{\tilde{g}} \right)^{\frac{1}{2}}.
\end{aligned}
$$
Then for some $C''>0$,
$$
\begin{gathered}
 \begin{aligned}
 \partial_{\mathfrak{t}}(Q_4^N(\mathfrak{t};\tilde{h},\tilde{g}))^2 \leq 
&\ 2 |\mathfrak{t}|^{\frac{n-1-2\sigma}{2\alpha}-1} M_4^N(\mathfrak{t};\tilde{h},\tilde{g})Q_4^N(\mathfrak{t};v,\tilde{g}) \big(C''
+C''Q_4^N(\mathfrak{t};v,\tilde{g}) \big)
\\
&\ +2C''L_4^N(\mathfrak{t};\tilde{h},\tilde{g}) P_4^N(\mathfrak{t};\tilde{h},\tilde{g}),
\\
 \partial_{\mathfrak{t}}(P_4^N(\mathfrak{t};\tilde{h},\tilde{g}))^2 \leq &\ 
 2C'' L_4^N(\mathfrak{t};\tilde{h},\tilde{g}) P_4^N(\mathfrak{t};\tilde{h},\tilde{g}),
 \end{aligned}
  \\
 Q_4^N(\mathfrak{t}_0;\tilde{h},\tilde{g})= P_4^N(\mathfrak{t}_0;\tilde{h},\tilde{g})=0.
 \quad\quad\quad\quad\quad
  \quad\quad\quad\quad\quad
   \quad\quad\quad\quad\quad
    \quad\quad\quad
\end{gathered}
$$
Here $\frac{1}{2}>\frac{n-1-2\sigma}{2\alpha}>\frac{1}{2}-\delta$, which implies that
$$
Q_4^N(\mathfrak{t};\tilde{h},\tilde{g})
 \leq C\epsilon, \quad P_4^N(\mathfrak{t};\tilde{h},\tilde{g})
 \leq C\epsilon\quad \forall \mathfrak{t}\in [\mathfrak{t}_0,0], 
$$
where $C$ depends on $\alpha,\delta,\delta', \sigma$. Finally, $Q_4^N(0;\tilde{h},\tilde{g}) $ gives the norm bound of $\tilde{h}|_{S_1^+\cap\widetilde{\Omega}_4}$.
\end{proof}

\vspace{0.2in}
\section{Nonlinear M\o ller Wave Operator}\label{sec.mollerop}
For $n\geq 4$, suppose $(h^0,h^1)\in \widetilde{\mathcal{V}}^{N,\delta}_{\epsilon}$ with $N\geq n+6,\delta\in(0,\frac{1}{2})$ and $\epsilon>0$ small such that 
Theorem \ref{thm.est} holds.  
Then the global solution $h$ obtained in Theorem \ref{thm.est} provides a true solution to Einstein equation (\ref{eq.1}) since $h$ satisfies the harmonic gauge condition (\ref{harmonic.3}) globally by Theorem {\ref{thm.lr}}. This imposes constraint conditions on the radiation field. 
\begin{equation*}
\partial_{\tau}(\tilde{h}_{\mu\rho}|_{S_1^+} )
-\tfrac{1}{2}\theta_{\mu}\partial_{\tau}(\mathrm{tr}_m\tilde{h}|_{S_1^+})=0, \quad \mu=0,1,...,n.
\end{equation*}
This is equivalent to 
\begin{equation}\label{harmonic.char}
(\tilde{h}_{\mu\rho}|_{S_1^+} )
-\tfrac{1}{2}\theta_{\mu}(\mathrm{tr}_m\tilde{h}|_{S_1^+})=0,\quad \mu=0,1,...,n.
\end{equation}
when $\tilde{h}|_{S_1^+}$ has some decay as $\tau\rightarrow -\infty$.
To refine the image space for the radiation field, first let us consider the Cauchy data with only conormal singularity at $\partial\Sigma_0$. 
\begin{lemma} \label{lem.7.1}
If $(h_0,h_1)\in\widetilde{\mathcal{V}}^{N,\delta}_{\epsilon}\cap 
\mathscr{A}^{\frac{n-1}{2}+\delta}(\Sigma_0)$, then 
$h\in \mathscr{A}^{\frac{n-1}{2}+\delta, \frac{n-1}{2}, \sigma}(X) $, where $\sigma>0$ is given in Theorem \ref{thm.est}. 
\end{lemma}
\begin{proof}
For all $N'> N$, we estimate $M_i^{N'}(\mathfrak{t};\tilde{h},\tilde{g})$ in the same way as we do for $M_i^{N}(\mathfrak{t};\tilde{h},\tilde{g})$. The only difference is that in this case, we can replace the cubic term $(M_i^{N}(\mathfrak{t};\tilde{h},\tilde{g}))^3$ in the estimates of divergence term by 
$C'M_i^N(\mathfrak{t},\tilde{h})(M_i^{N'}(\mathfrak{t},\tilde{h}))^2$ as a result of interpolation. Hence the nonlinear estimating reduces to linear estimating: 
$$
\begin{aligned}
&\partial_{\mathfrak{t}}M_1^{N'}(T_1;\tilde{h},\tilde{g}) \leq C' M_1^{N'}(T_1;\tilde{h},\tilde{g}),
\\
&\partial_{\mathfrak{t}}M_i^{N'}(T'_i;\tilde{h},\tilde{g}) \leq 
((\tfrac{1}{2}-\delta)|T'_i|^{-1}+C'|T'_i|^{\delta-\delta'-1}|\ln |T'||) M_i^{N'}(T'_i;\tilde{h},\tilde{g}), \quad i=2,3,4
\end{aligned}
$$
for some $C'>0$. For some $C_1, C_2,C_3,C_4>0$, we have
$$
\begin{aligned}
&M_1^{N'}(T_1;\tilde{h},\tilde{g}) \leq C_1 M_1^{N'}(0;\tilde{h},\tilde{g}) ,\\
&M_i^{N'}(T'_i;\tilde{h},\tilde{g}) \leq C_i |T'_i|^{\delta-\frac{1}{2}} M_1^{N'}(0;\tilde{h},\tilde{g}), \quad i=2,3,4.
\end{aligned}
$$
Then by the same proof as Corollary \ref{cor.6.2}, \ref{cor.6.3}, \ref{cor.6.4}, we have 
$$
\|\tilde{h}_{S_1^+}\|_{\rho_0^{\delta}\rho_2^{\sigma-\frac{n-1}{2}}H_b^{N'}(S_1^{+})} \leq CM_1^{N'}(0;\tilde{h},\tilde{g}) 
$$
for some $C$ only depending on $N', \delta, \delta', \alpha$. Since $N'$ is arbitrary large, $\tilde{h}$ only have conormal singularity at $\partial X$. By Lemma \ref{lem.2.3}, we have $\tilde{h}\in \mathcal{A}^{\delta,0,\sigma-\frac{n-1}{2}}$. Hence $h=\tilde{\rho}^{\frac{n-1}{2}}\tilde{h}\in \mathscr{A}^{\frac{n-1}{2}+\delta, \frac{n-1}{2}, \sigma}(X) $. 
\end{proof}

\begin{lemma}\label{lem.7.2}
If $(h_0,h_1)\in\widetilde{\mathcal{V}}^{N,\delta}_{\epsilon}\cap 
\mathscr{A}^{\frac{n-1}{2}+\delta}(\Sigma_0)$, then
$h\in \mathscr{A}^{\frac{n-1}{2}+\delta, \frac{n-1}{2}, \frac{n-1}{2}+\delta}(X).$
\end{lemma}
\begin{proof}
Applying the linear wave operator $\Box_{m}$ to the solution $h$ gives
\begin{equation*}
\begin{aligned}
\Box_{m}h_{\mu\nu} =\ &  \Box_{g}h_{\mu\nu}+(\Box_{m}-\Box_{g})h_{\mu\nu}
= F_{\mu\nu}(\partial h,\partial h) -H^{\alpha\beta}\partial_{\alpha}\partial_{\beta}h_{\mu\nu}
\in \mathscr{A}^{n+1+2\delta, n-1, 2+2\sigma}(X).
\end{aligned}
\end{equation*}
Hence $h\in \mathscr{A}^{\frac{n-1}{2}+\delta, \frac{n-1}{2},2\sigma}(X)$ if $2\sigma<\frac{n-1}{2}+\delta$. See \cite{MW} for details. 
Repeating $k$ times until $2^k\sigma\geq \frac{n-1}{2}+\delta$ and finally we have
$h\in \mathscr{A}^{\frac{n-1}{2}+\delta, \frac{n-1}{2},\frac{n-1}{2}+\delta}(X)$
\end{proof}

Fix Cauchy data $(h^0,h^1)\in \widetilde{\mathcal{V}}^{N,\delta}_{\epsilon}\cap 
\mathscr{A}^{\frac{n-1}{2}+\delta}(\Sigma_0)$ and let $h$ be the solution obtained in Theorem \ref{thm.est}.   Consider the Cauchy problem for linear wave equation with background metric $g=m+h$: 
\begin{equation}\label{eq.7.1}
\Box_{g} k_{\mu\nu} = F_{\mu\nu}(h)(\partial k,\partial h), \quad
(k_{t=0}, \partial_tk|_{t=0})=(k^0,k^1).
\end{equation}
Then $k$ globally exists and the radiation field is well defined. Moreover,  $k=h$ is a solution if $(k^0,k^1)=(h^0,h^1)$. Denote by $\tilde{k}=\tilde{\rho}^{\frac{1-n}{2}}k$ and define the following map
\begin{equation*}
{}^g\mathcal{R_F}: \rho_0^{\lambda}C^{\infty}\times \rho_0^{\lambda+1}C^{\infty}\ni (k^0,k^1)\rightarrow \tilde{k}|_{S^+_1}\in (\rho_0\rho_2)^{\lambda-\frac{n-1}{2}}C^{\infty}(S_1^+).
\end{equation*}
 for $\lambda\in \mathbb{C}, \Re\lambda\in (\frac{n-1}{2},\frac{n}{2})$.
\begin{lemma}\label{lem.7.3}
For $\Re\lambda\in(\frac{n-1}{2},\frac{n}{2})$, 
${}^g\mathcal{R_F}$ and ${}^m\mathcal{R_F}$ have the same boundary operator. 
\end{lemma}
\begin{proof}
With $(k^0,k^1)\in \rho_0^{\lambda}C^{\infty}\times \rho_0^{\lambda+1}C^{\infty}$, by similar energy estimates in the proof of Lemma \ref{lem.7.1}, we first have $k\in \mathscr{A}^{\lambda,\frac{n-1}{2},\sigma}$. Then applying $\Box_{m}$ as in the proof of Lemma \ref{lem.7.2}, we get $k\in \mathscr{A}^{\lambda,\frac{n-1}{2},\lambda}$. 
Consider the equation 
\begin{equation*}
\Box_{m} k'_{\mu\nu} =0, \quad
(k'_{t=0}, \partial_tk'|_{t=0})=(k^0,k^1).
\end{equation*}
Then
\begin{equation*}
\begin{aligned}
\Box_{m} (k-k') =-H^{\mu\nu}\partial_{\mu}\partial_{\nu}k+F(h)(\partial k, \partial h)
\in\mathcal{A}^{\frac{n-1}{2}+\delta+\lambda+2,n-1,\frac{n-1}{2}+\delta+\lambda+2}.
\end{aligned}
\end{equation*}
Hence $k'-k\in \mathcal{A}^{\frac{n-1}{2}+\delta+\lambda,\frac{n-1}{2},\frac{n-1}{2}+\delta+\lambda}
+\mathcal{A}^{\frac{n-1}{2}+\delta+\lambda+2,\frac{n-1}{2}, n-2}$. Since $\Re\lambda<n-2$ for $n\geq 4$, this term contributes zero to the boundary operator. See \cite{MW} for details.
\end{proof}
By Lemma \ref{lem.7.3} and the mapping property of ${}^m\mathcal{R_F}$ given in \cite{MW}, we show the following lemma.  
\begin{lemma}\label{lem.7.4}
The linear M\o ller wave operator ${}^g\mathcal{R_F}$ defines a continuous map for $\delta\in (0,\frac{1}{2})$ 
\begin{equation*}
\begin{aligned}
{}^g\mathcal{R_F}: \rho_0^{\frac{n-1}{2}+\delta} &H_b^{N+1}(\Sigma_0)\times \rho_0^{\frac{n+1}{2}+\delta}H_b^{N+1}(\Sigma_0)
\\
&\longrightarrow (\rho_0\rho_2)^{\delta} [H_b^{\frac{1}{2}-\delta}(\overline{\mathbb{R}};H^{N+\frac{1}{2}+\delta} (\mathbb{S}^{n-1}))\cap L^2(\mathbb{S}^{n-1};H_b^{N+1}(\overline{\mathbb{R}}))]
\end{aligned}
\end{equation*}
Here  $\|{}^g\mathcal{R_F}\|<C$ for some constant $C>0$ independent of $(h^0,h^1)\in \widetilde{\mathcal{V}}^{N,\delta}_{\epsilon}\cap 
\mathscr{A}^{\frac{n-1}{2}+\delta}(\Sigma_0)$. Moreover, 
${}^m\mathcal{R_F}$ is an isomorphism. 
\end{lemma}

Denote by $\widetilde{\mathcal{W}}^{N,\delta}_{\epsilon}$ be the collection of elements in 
$$(\rho_0\rho_2)^{\delta} [H_b^{\frac{1}{2}-\delta}(\overline{\mathbb{R}};H^{N+\frac{1}{2}+\delta} (\mathbb{S}^{n-1}))\cap L^2(\mathbb{S}^{n-1};H_b^{N+1}(\overline{\mathbb{R}}))]$$ with norm less than $\epsilon$ and satisfying the harmonic gauge condition (\ref{harmonic.char}). Since $\delta>0$, $\widetilde{\mathcal{W}}^{N,\delta}_{\epsilon}$ is a small neighborhood of $0$ in the Sobolev space $\widetilde{\mathcal{W}}^{N,\delta}_{\infty}$.

\begin{theorem}\label{thm.op}
With the assumption in Theorem \ref{thm.est} and choosing $\epsilon>0$ even smaller if necessary, the nonlinear M\o ller wave operator defines a continuous map and open map
\begin{equation*}
\mathscr{R_F}: \widetilde{\mathcal{V}}^{N,\delta}_{\epsilon}\longrightarrow 
\widetilde{\mathcal{W}}^{N,\delta}_{C\epsilon}
\end{equation*}
for  some $C>0$. 
\end{theorem}
\begin{proof}
First, $\mathscr{R_F}$ is well defined and continuous. When restricting to $(h^0,h^1)\in  \widetilde{\mathcal{V}}^{N,\delta}_{\epsilon} \cap \mathscr{A}^{\frac{n-1}{2}+\delta}(\Sigma_0)$, it is obvious that
$$
\|\mathscr{R_F}\|=\|{}^g\mathcal{R_F}\|< C. 
$$
where $C$ is given in Lemma \ref{lem.7.4}. By density argument, it holds for all $(h^0,h^1)\in  \widetilde{\mathcal{V}}^{N,\delta}_{\epsilon}$. 
Secondly, the linearization of $\mathscr{R_F}$ at $(0,0)$ is ${}^m\mathcal{R_F}|_{\mathcal{T}^{N,\frac{n-1}{2}+\delta}}$, where $\mathcal{T}^{N,\frac{n-1}{2}+\delta}$ is the tangent space of $\widetilde{\mathcal{V}}^{N,\delta}_{\epsilon}$.
Consider the wave equation 
$$
\Box_m h=0, \quad (h|_{t=0},\partial_gh|_{t=0}) = (h^0, h^1).
$$
And let 
$$
{}^m\Gamma_{\mu}(h)=m^{\alpha\beta}\partial_{\alpha}h_{\mu\beta}-\partial_{\mu}
(\mathrm{tr}_m h), \quad\mu=0,1,...,n.
$$
Then  the solution space of the linearization of harmonic gauge conditions  (\ref{constraint.6}) and 
(\ref{constraint.7})  is
$$
\mathcal{T}^{N,\frac{n-1}{2}+\delta}=\{(h^0, h^1)\in \rho_0^{\frac{n-1}{2}+\delta} H_b^{N+1}(\Sigma_0)\times \rho_0^{\frac{n+1}{2}+\delta}H_b^{N+1}(\Sigma_0):
{}^m\Gamma_{\mu}(h)|_{t=0}, \ \partial_t{}^m\Gamma_{\mu}(h)|_{t}=0
\}
$$
Since $\Box_{m}{}^m\Gamma_{\mu}(h)={}^m\Gamma_{\mu}(\Box_m h)$, we have the restriction map 
$$
{}^m\mathcal{R_F} |_{\mathcal{T}^{N,\frac{n-1}{2}+\delta}}: \mathcal{T}^{N,\frac{n-1}{2}+\delta} \longrightarrow \widetilde{\mathcal{W}}^{N,\delta}_{\infty}
$$
is still an isomorphism.  By implicit function theorem, $\mathscr{R_F}: \widetilde{\mathcal{V}}^{N,\delta}_{\epsilon}\longrightarrow 
\widetilde{\mathcal{W}}^{N,\delta}_{C\epsilon}$  is open if $\epsilon$ is small enough. 
\end{proof}

\textbf{Proof of Theorem \ref{mainthm}}: It follows from Theorem \ref{thm.est} and Theorem \ref{thm.op} directly.

\vspace{0.2in}
\section{Appendix}
In this section, we prove some mapping property of $P=|\triangle|^{\frac{1}{2}}$ on $\mathbb{R}^n$ for $n\geq 3$ by the same method we used in \cite{MW}, where we show the mapping properties of Radon transform and Radiation field map for standard wave equation. 

Define for  $u\in C_c^{\infty}(\mathbb{R}^n)$
	\begin{equation}\label{app.1}
	Pu(x) =\frac{1}{(2\pi)^n}\int_{\mathbb{R}^n}\int_{\mathbb{R}^n} e^{i(x-y)\cdot \xi} |\xi| u(y)dy d\xi.
	\end{equation}
Then a similar proof as for the inverse Radon transform in \cite{He} shows that: 
	$$
	Pu=\begin{cases}\frac{(-i)^n}{2(2\pi)^{n-1}} \mathcal{M}\partial_s^n\mathcal{R}u
	= \frac{1}{2(2\pi)^{n-1}} \mathcal{M}\mathcal{R}(\triangle^{\frac{n}{2}}u)&\ \textrm{n=even}\\
	\frac{(-i)^n}{2(2\pi)^{n-1}} \mathcal{M}\mathcal{H}\partial_s^n\mathcal{R}u
	=\frac{-i}{2(2\pi)^{n-1}} \mathcal{M}\mathcal{H}\partial_s\mathcal{R}(\triangle^{\frac{n-1}{2}}u) &\ \textrm{n=odd}
	\end{cases}
	$$
where $\triangle=-\Sigma_{i=1}^n\partial_i^2$ and $\mathcal{R}$ is the Radon transform
	$$
	(\mathcal{R}u)(s,\omega) = \int_{x\cdot \omega=s} u(x)dH
	$$
with $dH$ the induced Lebesgue measure on hyperplane $x\cdot \omega=s$ and $\mathcal{M}, \mathcal{H}$ are defined as follows: for $v\in C_c^{\infty}(\mathbb{R}\times \mathbb{S}^{n-1})$, 
	$$
	\begin{aligned}
	(\mathcal{M}v)(x) =&\  \int_{\mathbb{S}^{n-1}} v(x\cdot\omega,\omega) d\omega, \\
	(\mathcal{H}v)(s,\omega)=&\ \frac{i}{\pi} \lim_{(\epsilon,N)\rightarrow (0,\infty)} \int_{\epsilon<|s-s'|<N} (s-s')^{-1} v(s',\omega)ds'.
	\end{aligned}
	$$
\begin{proposition}\label{app.1}
For $n\geq 3$ and $\lambda\in (0,n-1)$, $P$ extends to an isomorphism: 
	$$
	P: \rho_0^{\lambda} H_b^{N}(\overline{\mathbb{R}^n}) \longrightarrow  \rho_0^{\lambda+1}H_b^{N-1}(\overline{\mathbb{R}^n}).
	$$
satisfying $P^2=\triangle$.
\end{proposition}
\begin{proof}
By the mapping property of $\triangle$, we only need to show the above map is continuous. For $\Re\lambda\in (0,n-1)$, by a similar proof as in \cite{MW}, $P$ extends to a continuous map
	$$
	\rho^{\lambda}C^{\infty}(\overline{\mathbb{R}}^n)\longrightarrow \rho^{\lambda+1}C^{\infty}(\overline{\mathbb{R}}^n)+\rho^{n+1}C^{\infty}(\overline{\mathbb{R}}^n)
	$$
which restricts to the boundary to define
	$$
	\begin{gathered}
	B(\tau)f=\rho^{-\tau-1}P\rho^{\tau}\tilde{f}|_{\rho=0}: C^{\infty}(\mathbb{S}^{n-1})\longrightarrow C^{\infty}(\mathbb{S}^{n-1})\\
	\mathrm{where}\ \tilde{f}\in C^{\infty}(\overline{\mathbb{R}}^n), \ \tilde{f}|_{\rho=0}=f. 
	\end{gathered}
	$$
Here $B(\tau)$ is a meromorphic family of semicalssical pseudodifferential operator with simple poles at $(n+\mathbb{N}_0)\cup (-\mathbb{N})$ . 
Moreover, let $A(\tau):C^{\infty}(\mathbb{S}^{n-1})\longrightarrow C^{\infty}(\mathbb{S}^{n-1})$ defined by the Schwartz kernel $(\theta\cdot \omega)_+^{\tau-n}$. Then
	$$
	\begin{gathered}
	B(\tau)=\begin{cases}C_n(\tau)A(n-1-\tau)A(\tau) &\ n=\mathrm{even}\\
	\frac{i}{\pi}C_n(\tau)[-C_+(\tau+1)A(n-1-\tau)+C_-(\tau+1)\Theta A(n-1-\tau)] A(\tau) &\ n=\mathrm{odd}
	\end{cases}
	\end{gathered}
	$$
where $\Theta$ is the antipodal map and 
	$$
	C_n(\tau)= \frac{(-i)^n}{(2\pi)^{n-1}}\prod_{i=0}^{n-1}(\tau-i),\ C_{\pm}(\tau+1)=\int_0^{\infty}(1\mp t)^{-1}t^{-\tau-1}dt.
	$$
Hence the Schwartz kernel of $B(\tau)$ can be split into three parts: 
	$$
	\begin{gathered}
	\chi(B(\tau))(\theta,\theta')=B_0(\theta,\theta')+B_{sc}(\theta,\theta')+B_{\pm}(\theta,\pm\theta'),\quad \mathrm{where}
	\\
	B_0\in \eta^{-1}\Psi_{sc}^{-n}(\mathbb{S}^{n-1}),\ B_{sc}\in \eta^{-1}\Psi_{sc}^{1}(\mathbb{S}^{n-1}),\ B_{\pm}\in \mathcal{S}(\mathbb{R}_{\mathrm{Im}\tau},\Psi^{1}(\mathbb{S}^{n-1}))
	\end{gathered}
	$$
with $\eta=\frac{1}{\mathrm{Im}\tau}$ the semiclassical parameter as $\mathrm{Im}\tau\rightarrow 0$. From standard estimates on the boundedness of semiclassical families, $B(\tau)$ defines two continuous maps:
	$$
	\begin{aligned}
	&B(\tau): H^N(\mathbb{S}^{n-1})
	\longrightarrow H^{N-1}(\mathbb{S}^{n-1}), \quad \|B(\tau)\|\leq C;
	\\
	&B(\tau): H^{N}(\mathbb{S}^{n-1})
	\longrightarrow H^{N}(\mathbb{S}^{n-1}),\quad \|B(\tau)\|\leq C(1+|\mathrm{Im}\tau|). 
	\end{aligned}
	$$ 
By Mellin transform and taking $\alpha=\mathrm{Re}\tau$, we finish the proof. 
\end{proof}

\end{document}